%% file: generalized_fast_km.tex
\documentclass[10pt]{amsart}
\usepackage{hyperref}
\hypersetup{
	colorlinks=true,
	linkcolor=blue,
	citecolor=red,
	urlcolor=blue,
}
\input{definitions}

\numberwithin{equation}{section}

\allowdisplaybreaks

\usepackage[backend=bibtex, giveninits=true, doi=false,isbn=false,url=false,eprint=false, maxnames=50]{biblatex}

\begin{document}

% title
\title[Generalized Fast-KM]{\rev{Fast Krasnoselskii--Mann Method with Overrelaxation and Preconditioners}}

% authors
\author{Radu I. Bo\c{t}$^*$}
\author{Enis Chenchene$^*$}
\address{\phantom{1}$^*$Faculty of Mathematics, University of Vienna, Oskar-Morgenstern-Platz 1, 1090 Vienna, Austria}
\email{radu.bot@univie.ac.at, enis.chenchene@univie.ac.at}
\thanks{{The research of RIB and EC has been supported by the Austrian Science Fund (FWF), projects W 1260 and P 34922-N, respectively.}}

\author{Jalal M. Fadili$^\dagger$}
\address{\phantom{1}$^\dagger$Normandie Univ, ENSICAEN, CNRS, GREYC, France.}
\email{Jalal.Fadili@ensicaen.fr}

\begin{abstract}
We study accelerated Krasnoselskii--Mann-type methods with preconditioners in both continuous and discrete time. From a continuous-time model, we derive a generalized fast Krasnoselskii--Mann method, providing a new yet simple proof of convergence that \rev{leads to} unprecedented flexibility in parameter tuning, \rev{allowing up to twice larger relaxation parameters}. Our analysis unifies inertial and anchoring acceleration mechanisms and offers a broad range of parameter choices, which prove beneficial in practice. Additionally, we extend our analysis to the case where preconditioners are allowed to be degenerate, unifying the treatment of various splitting methods and establishing the weak convergence of shadow sequences.
\end{abstract}

\maketitle

Keywords: Inertial Accelerations, Nesterov Acceleration, Halpern Method, Fast-Krasnoselskii--Mann {Method}, Degenerate Proximal Point Algorithm

\section{Introduction}

Let $T\colon H\to H$ be a nonexpansive operator on the {real} Hilbert space $H$. Our objective is to solve the fixed-point problem:
\begin{equation}\label{eq:intro_fixed_point_problem}
	\text{Find} \ x \in H \ \text{such that} \ x = T(x)\,,
\end{equation}
with fast methods in continuous and discrete time, and apply these results to splitting methods providing convergence of the so-called shadow sequences (to be made precise later).

If $T =I - \nabla f$, where $f \colon H\to \R$ is a convex and smooth function whose set of minimizers is nonempty, \eqref{eq:intro_fixed_point_problem} is equivalent to minimizing $f$. The most popular accelerated method to minimize $f$ in discrete time is the celebrated Nesterov's {algorithm} \cite{Nesterov1983AMF}. Su, Boyd, and Cand\`es in \cite{sbc16} later discovered that Nesterov's {algorithm} can be understood as a specific {finite differences} discretization of the following ODE
\begin{equation}\label{eq:sbc}
	\ddot{x}(t) + \frac{\alpha}{t}\dot{x}(t) + \nabla f(x(t)) = 0\,, \quad \text{for} \ t \geq t_0 {> 0}\,
\end{equation}
{with $\alpha = 3$} and initial data {$(x(t_0),\dot{x}(t_0))=(x_0,v_0) \in H^2$}. Although the convergence of $x(t)$ to a minimizer of $f$ remains an open problem {for $\alpha=3$} (except in one dimension where convergence holds \cite{acr19}), the trajectories of \eqref{eq:sbc} satisfy $f(x(t)) - {\min} f = \cO(t^{-2})$ {as $t \rightarrow +\infty$}, which is an improvement over the $o(t^{-1})$ rate of {the} standard gradient flow. Later, in \cite{acpr18, cd15}, it has been shown that replacing the critical parameter $\alpha = 3$ with any $\alpha > 3$ ensures the {weak} convergence of $x(t)$ {to a minimizer of $f$}, and improves the rate {of convergence for the function values} from $\cO(t^{-2})$ to $o(t^{-2})$. 

However, these inertial acceleration methods often suffer from significant oscillations. This issue was addressed by Attouch, Peypouquet, and Redont \cite{apr16}, and later by Attouch, Chbani, Fadili, and Riahi \cite{acfr20, acfr23}, who proposed to couple the {asymptotically vanishing viscous damping} term in \eqref{eq:sbc} {to an additional} \textit{{geometric} Hessian-driven damping term}, previously introduced in \cite{aabr02, amr12}, resulting in the system
\begin{equation}\label{eq:acfr}
\ddot{x}(t) + \frac{\alpha}{t}\dot{x}(t) + \beta(t) \frac{d}{dt} \nabla f(x(t)) + \nabla f(x(t))= 0\,, \quad \text{for} \ t \geq t_0\,,
\end{equation}
which, {for $\alpha > 3$ and appropriate coefficient $\beta(t) > 0$}, has the important benefit of retaining the $o(t^{-2})$ rate on the {objective}, while ensuring the weak convergence of trajectories, and attenuating oscillations. 

\rev{
Inspired by \eqref{eq:acfr}, we here study the much more general case where $\nabla f$ is replaced by any monotone and Lipschitz continuous operator $Q\colon H\to H$:
\begin{equation}\label{eq:intro_general_second_order}
	\ddot{x}(t) + \frac{\alpha}{t}\dot{x}(t) + \beta(t) \frac{d}{dt} Q(x(t)) + b(t) Q(x(t))= 0\,, \quad \text{for} \ t \geq t_0\,.
\end{equation}
For instance, to solve \eqref{eq:intro_fixed_point_problem}, {it is natural} to consider $Q:=I-T$, which is indeed monotone and Lipschitz, and the equilibria of \eqref{eq:intro_general_second_order} are solutions to \eqref{eq:intro_fixed_point_problem}. Note that while oscillations attenuation with Hessian-driven damping are quite well-understood and intuitive in optimization contexts (where $Q = \nabla f$), the interpretation of the geometric damping for general monotone operators is not always clear and remains an open problem. In fact, the geometric damping might not always reduce oscillations for general monotone operators as we will see in some case in our numerical results. Nonetheless, we will still stick with this terminology in this paper for consistency reasons.
}

Although the system in \eqref{eq:acfr} generates fast converging trajectories for a range of admissible parameters (cf.~\cite[Theorem 2.1]{acfr23}), when {one turns} to {the general} operator setting {\eqref{eq:intro_general_second_order}} the situation becomes significantly more delicate. A notable contribution in this direction was made by Bo\c t, Csetnek and Nguyen in \cite{bot2023}, who studied \eqref{eq:intro_general_second_order} with $b(t):= \frac{1}{2}(\dot \beta(t) + \frac{\alpha}{t}\beta(t))$, the so-called fast Optimistic Gradient Descent Ascent (Fast-OGDA) system. {In particular, they showed that when} $\beta$ is constant and $\alpha >2$, this dynamical system generates trajectories weakly converging to a zero of $Q$ (provided that {the set of zeros of $Q$ is {non-empty}}) with {a convergence rate of} $\|Q(x(t))\| = o(t^{-1})$ {as $t \rightarrow +\infty$}.   Therefore,  \eqref{eq:intro_general_second_order} can be seen as an accelerated {version} of {the first order system}
\begin{equation}\label{eq:monotone_flow}
	\dot{x}(t) + Q(x(t)) = 0\,, \quad \text{for} \ t \geq t_0\,.
\end{equation}
{Indeed, when} $Q$ is a cocoercive operator, {it is known that each solution trajectory generated by \eqref{eq:monotone_flow} weakly converges to a zero of $Q$} with a convergence rate of $\|Q(x(t))\| = o(t^{-\frac{1}{2}})$ as $t \rightarrow +\infty$ \cite{attouch2023fast}. \rev{To further emphasize this crucial difference, we refer to rates of the form $\cO(t^{-1})$, or $o(t^{-1})$ as \emph{fast rates}.}

It turns out that the specific choice $b(t):= \frac{1}{2}(\dot \beta(t) + \frac{\alpha}{t}\beta(t))$ in \eqref{eq:intro_general_second_order} considered by \cite{bot2023} is ultimately due to the impossibility of {exploiting} function values in {the} Lyapunov analysis. This immediately leads us to the natural question: 
\begin{center}{\textit{Can Fast-OGDA in \cite{bot2023} be extended to accommodate more general parameters $b$, enabling new fast and flexible numerical
algorithms for solving \eqref{eq:intro_fixed_point_problem}?}}\end{center}

In this paper, we show that is possible to consider the more general choice\footnote{Observe that $\theta\in (1, \alpha-1)$ is a convex combination of $1$ and $\alpha-1$ with parameter $\eta \in (0,1)$. The redundancy in notation will be justified later in the analysis, which turns out more convenient to write with respect to $\eta$.}
\begin{equation}\label{eq:intro_b_bcf}
	b(t):= \rev{(1-\eta)} \dot \beta(t) + \rev{\theta} \frac{\beta(t)}{t}\,, \quad \text{for} \quad \eta \in (0, 1) \quad \mbox{and} \quad  \rev{\theta :=(1-\eta) + \eta (\alpha - 1)}\,,
\end{equation}
which recovers \cite{bot2023} by setting $\eta = \tfrac{1}{2}$, \rev{i.e., $\theta = \frac{\alpha}{2}$}. Our convergence guarantees for a general parameter $\eta \in (0, 1)$ {cover} those in \cite{bot2023}: in particular, we establish that, e.g., if $\alpha >2$ and $\beta$ is constant, we have weak convergence of the trajectories to zeros of $Q$ with a convergence rate of {$\|Q(x(t))\| = o(t^{-1})$ as $t \rightarrow +\infty$}. \rev{Additionally, we show that in the two edge cases $\eta=0$ or $\eta=1$ the rate $\|Q(x(t))\|=o(t^{-1})$ cannot be achieved in general, unveiling the tightness of our analysis and an intriguing fundamental limitation in the choice of the parameter function $b$.}

\rev{The Fast-OGDA system is in turn related to a seemingly completely different mechanism: the Tikhonov regularization, also known as anchoring mechanism. Specifically, for $\beta(t)\equiv 1$ and $\eta = 0$, the dynamical system \eqref{eq:intro_general_second_order}, can be easily seen to be equivalent to:
\begin{equation}\label{eq:intro_tychonoff}
\dot x(t) + Q(x(t)) + \frac{\alpha-1}{t}(x(t)-{v}) = 0\,, \quad \text{for} \ t \geq t_0\,,
\end{equation} 
with anchor point $v := x_0 + \frac{t_0}{\alpha-1}(\dot{x}(t_0) + Q(x_0))$, cf.~Section \ref{sec:connection_to_tichonov}. The dynamical system \eqref{eq:intro_tychonoff} is nothing but the Tikhonov regularization of the monotone operator flow \eqref{eq:monotone_flow}, see \cite{browder76, reich77, cps08}, that in fact not only generates trajectories that converge strongly to the projection of $v$ onto $\zer Q$, but also exhibits the rate $\|Q(x(t))\|=\cO(t^{-1})$ as $t \rightarrow +\infty$ for all $\alpha \geq 2$.
	
The system \eqref{eq:intro_tychonoff} can also be recognized as the continuous-time model of the celebrated Halpern method which is a classical method in fixed-point theory \cite{bk24, spr23}. A simplified version of \eqref{eq:intro_tychonoff} was studied in \cite{spr23} in finite-dimensional spaces. In the discrete-time setting, Sabach and Shtern in \cite{ss17} and, independently, Lieder in \cite{lieder21}, proved that Halpern's method achieves the {convergence rate of} $\cO(k^{-1})$ {as $k \rightarrow +\infty$} on the norm of the fixed-point residual. Interestingly, this cannot be improved to $o(k^{-1})$ in normed spaces; see \cite{cc23}. The connection between Halpern's method and momentum-type algorithms have been observed previously in \cite{cc23, ykr24, TranDinh24}.
}

\rev{The \( o(k^{-1}) \) convergence rate of the fixed-point residual norm is achieved by Algorithm~\ref{alg:fast_km}, the main focus of our paper. This algorithm solves the fixed-point problem \eqref{eq:intro_fixed_point_problem} and can be viewed as a discrete-time analogue of the continuous dynamics \eqref{eq:intro_general_second_order}, with \( b \) given by \eqref{eq:intro_b_bcf}, \( \beta \equiv 1 \), and \( Q = I - T \). It yields an accelerated variant of the classical Krasnoselskii--Mann method \cite{krasnoselskii1955, mann1953}, which for $x^0 \in H$ and \emph{relaxation parameter} $\theta \in (0, 1)$, iterates 
\begin{equation}\label{eq:krasnoselskii_mann}
	x^{k+1} = x^k + \theta \big(T(x^k) - x^k\big)\,, \quad \text{for} \  k \geq 0 .
\end{equation}
The scheme \eqref{eq:krasnoselskii_mann} is known to only exhibit the rate $o(k^{-\frac{1}{2}})$ as $k \to+\infty$ on the fixed-point residual; see \cite{liang2015convergence,Davis2016}. We can see that Algorithm~\ref{alg:fast_km} extends \eqref{eq:krasnoselskii_mann} by incorporating a \emph{momentum} term, governed by the standard coefficient $\alpha_k := 1 - \frac{\alpha}{k + \sigma}$, and features \emph{vanishing} relaxation parameters $\theta_k := \frac{\theta}{k + \sigma}$. Leveraging the enhanced flexibility of $b$ in \eqref{eq:intro_b_bcf}, it extends the Fast-Krasnoselskii--Mann method (Fast-KM) introduced by Bo\c{t} and Nguyen in \cite{botkhoa2023}, allowing for more general parameters \( \theta \in (1, \alpha - 1) \) and \( \sigma > 0 \). It results, in particular, in (almost) \emph{twice larger} relaxation parameters---from $\theta =\frac{\alpha}{2}$ to $\theta$ arbitrarily close to $\alpha-1$, as well as a general parameter $\sigma>0$. As we will show in Section~\ref{sec:num}, this enhanced flexibility leads to notable practical improvements, consistently outperforming the baseline method in several benchmark problems.} \rev{We strongly rely on the important interplay between continuous and discrete-time dynamics} to derive a convergence proof for Algorithm~\ref{alg:fast_km}  that matches almost exactly the continuous-time approach, \rev{making it easily accessible using the latter as a \emph{guideline}.} Additionally, this allows us to cover, in a unifying fashion, the important special case of $\alpha=\sigma=2$, for which the method reduces to the Optimal Halpern Method \cite{lieder21}, \rev{as well as the edge case $\theta=1$, for which Algorithm \ref{alg:fast_km} reduces to a Halpern-type method, much like \eqref{eq:intro_tychonoff} in continuous time}. 

Since Algorithm~\ref{alg:fast_km} allows us to find fixed points of nonexpansive operators, it emerges as a unified algorithmic framework to design \rev{accelerated variants of} \textit{operator splitting methods}, such as primal-dual \cite{cp11}, forward--backward \cite{bt09}, Douglas--Rachford \cite{dr56, drs_mercier_lions}, Davis--Yin \cite{dy17, rfp13}, and several other extensions. Following the recent insights in \cite{HeYuan12, bredies2021degenerate, bredies2022graph}, these methods can {be} studied as fixed-point iterations with respect to nonexpansive mappings in a renormed Hilbert space, see Section \ref{sec:splitting}. In this context, the metric might actually be \textit{degenerate}, i.e., only inducing a Hilbert seminorm, and thus the analysis requires a careful adaptation. We do so by delineating a \textit{degenerate} Opial Lemma, a fundamental ingredient that allows us to adapt the convergence analysis of Algorithm~\ref{alg:fast_km} almost effortlessly to this degenerate setting. We apply this result to accelerate the convergence of a recent extension of the Douglas--Rachford splitting method \cite{drs_mercier_lions, bredies2022graph}, ensuring the weak convergence of the so-called \textit{shadow} sequences in an infinite-dimensional {Hilbert space} setting.

\rev{
\begin{algorithm}[t]
\caption{(Generalized) Fast Krasnoselskii--Mann Method}
\label{alg:fast_km}
{\bf Pick} $\sigma >0$, \rev{$\alpha > 2$, $\theta \in (1, \alpha - 1)$}\\
{\bf Initialize} \rev{$x^{-1}, x^0 \in H$}\\
\For{$k =0, 1, 2,\dots$}{
\begin{equation*}
	\rev{x^{k+1} = x^k + \frac{\rev{\theta}}{k+\sigma} \big(T(x^k) - x^k \big) +  \left(1-\frac{\alpha}{k+\sigma}\right)\big(T(x^k)-T(x^{k-1})\big)}
\end{equation*}}
\end{algorithm}
}

\newpage
\subsection{Contributions} Our main contributions are summarized as follows:
\begin{itemize}
	 \item We introduce a generalization of the Fast-OGDA system introduced by Bo\c t, Nguyen and Csetnek in \cite{bot2023} with $b$ as in \eqref{eq:intro_b_bcf}, \rev{and show that if $\beta$ is constant, our analysis is tight in the sense that the edge cases in the parameter bounds \eqref{eq:intro_b_bcf} do not lead to fast little-$o$ convergence rates.}
	 \item We introduce Algorithm~\ref{alg:fast_km} as a discretization of the proposed continuous-time model and propose a new convergence proof that matches almost exactly the continuous-time approach. The method generalizes \cite{botkhoa2023} integrating two degrees of freedom: it considers a general parameter $\sigma>0$ and a general parameter \rev{$\theta\in (1,\alpha-1)$}, \rev{while maintaining the fast little-$o$ rate on the fixed-point residual.} Additionally, we cover the convergence analysis also in the edge case $\alpha =\sigma = 2$, which corresponds to the Optimal Halpern Method \cite{ykr24}, thereby providing a unified analysis of \textit{inertial} and \textit{anchoring} (Tikhonov) acceleration mechanisms.
    \item We establish an Opial-type lemma with degenerate preconditioners, and adapt our convergence proof to this setting. We further discuss how this allows us to systematically obtain convergence results for several accelerated splitting methods, establishing the weak convergence of the so-called \textit{shadow sequences} in infinite dimensional Hilbert spaces, thus closing a gap in \cite{botkhoa2023}.
    \item We present numerical experiments that demonstrate that the Fast-KM method with \rev{large relaxation parameters} seemingly outperforms the baseline choice in \cite{botkhoa2023}.
\end{itemize}

\subsection{Notation and Preliminaries}\label{sec:preliminaries_operators}

In this section, we fix the notation and introduce some preliminary results on nonexpansive operators in Hilbert spaces and their connection to monotone operators; see \cite{BCombettes} for a comprehensive account. An operator $T\colon H\to H$ is said to be \textit{nonexpansive} if
\begin{equation}
	\|T(x)- T(x')\|\leq \|x-x'\|\,,\quad \text{for all} \ x, x' \in H\,.
\end{equation}
The set of \textit{fixed points} of $T$ is the defined by $\Fix T := \{x \in H \mid x = T(x)\}$. In this paper, we always assume that $\Fix T \neq \emptyset$. If $T$ is nonexpansive, the operator $T_s := sT + (1-s)I$ for $s \in (0, 1)$ is said to be $s$-\textit{averaged}. $T_{1/2}$ is called firmly nonexpansive. It follows from \cite[Remark~4.34(iv)]{BCombettes} that $I-T_s$ is $\frac{1}{2s}$-cocoercive, i.e., it satisfies the following inequality:
\begin{equation}
	\bra x - x', (I-T_s)(x) - (I-T_s)(x')\ket \geq \frac{1}{2s}\|(I-T_s)(x)- (I-T_s)(x')\|^2\,, \quad \text{for} \ x, x' \in H\,.
\end{equation}
In particular, $I-T_s$ is a \textit{monotone} operator. A set-valued mapping $A\colon H\to 2^H$ is \textit{monotone} if it obeys
\begin{equation}
	\bra  a - a', u - u' \ket  \geq 0\,, \quad \text{for all} \ (u, a), (u', a')\in \gra A\,,
\end{equation}
where the set $\gra A\subseteq H \times H$ is the \textit{graph} of $A$ and is defined by $\gra A := \{(u, a) \mid a \in A(u)\}$. A monotone operator is called \textit{maximal monotone} if its graph is maximal (with respect to the inclusions of sets) among the class of monotone operators. The inverse of a set-valued mapping $A$ is denoted by $A^{-1}$ and defined \rev{through its graph} by $\gra A^{-1}:= \{(a, u) \mid (u, a) \in \gra A\}$. The \textit{resolvent} of $A$ is the mapping $J_{A}:=(I+A)^{-1}$, where $I$ denotes the identity operator on $H$. From \cite[Proposition~23.8]{BCombettes}, $A$ is {maximal monotone} if and only if $J_A$ is a single-valued mapping from $H$ to $H$, which is $1$-{cocoercive}, or simply \textit{cocoercive} (or firmly nonexpansive). Moreover, it is clear that the operator $T := 2J_{A}-I$ is nonexpansive. Conversely, from \cite[Corollary~23.9]{BCombettes}, if $T$ is nonexpansive, there exists a maximal monotone operator $A$ such that $T = 2J_{A}-I$. Furthermore, fixed points of $T$ correspond to \textit{zeros} of $A$, i.e., elements of $\Fix T = \zer A := \{u \in H \mid 0 \in A(u)\}$. We also denote $\rightharpoonup$ as weak convergence and $\to$ as strong convergence. \rev{Eventually, we denote $\llbracket1, n\rrbracket:=\{1,\dots, n\}$ for all $n \in \N$.}

\section{Continuous-time System and Analysis}

In this section, we introduce a general dynamical system that extends Fast-OGDA from \cite{bot2023}. Let $Q\colon H\to H$ be a monotone and Lipschitz continuous operator\footnote{Note that the Lipschitz continuity of $Q$ is only useful to establish the existence and uniqueness of a strong solution to \eqref{eq:second_order_dynamic_general} which in turn entails the almost everywhere differentiability of $t \mapsto Q(x(t))$.} and consider the dynamical system:
\begin{equation}\label{eq:second_order_dynamic_general}
	\ddot{x}(t) + \frac{\alpha}{t}\dot{x}(t) + \beta(t) \frac{d}{dt} Q(x(t)) + b(t)Q(x(t))=0\,, \quad t \geq t_0  > 0\,,
\end{equation}
where $b$ is defined for all $t \geq t_0$ by
\begin{equation}\label{eq:def_b}
	b(t) :=\rev{(1-\eta)} \dot{\beta}(t) + \rev{\theta}\frac{\beta(t)}{t}\,, \quad \rev{\text{where} \quad  \theta:=(1-\eta)+\eta (\alpha-1)}\,,
\end{equation}
with the additional parameter $\eta\in (0,1)$. The system \eqref{eq:second_order_dynamic_general} recovers the choice in \cite{bot2023} by setting $\eta= \tfrac{1}{2}$, for which $\rev{\theta} = \tfrac{\alpha}{2}$, and can be leveraged to approach solutions to \eqref{eq:intro_fixed_point_problem} by setting $Q=I-T$, since indeed:
\begin{equation*}
	\zer Q = \zer (I-T) = \Fix T \,.
\end{equation*}

We will need that \eqref{eq:second_order_dynamic_general} admits at least one strong global solution for every starting point $x(0)=x_0 \in H$, i.e., such that $x$ and $\dot x$ are locally absolutely continuous on $[t_0, +\infty)$, and $x$ satisfies \eqref{eq:second_order_dynamic_general} \textit{almost everywhere} (a.e.)~on $[t_0, +\infty)$. In turn, this will entail that for every locally absolutely continuous trajectory $t \mapsto x(t)$ the trajectory $t\mapsto Q(x(t))$ is also locally absolutely continuous. In fact, using \cite[Proposition~6.2.1]{Haraux91}, Lipschitz continuity of $Q$ and mild conditions on $\beta(t)$, it can be shown that the system \eqref{eq:second_order_dynamic_general} has a unique strong solution trajectory. In fact, Lipschitz continuity of $Q$ can even be weakened to Lipschitz \rev{continuity} over the bounded subsets of $H$ to show the existence and uniqueness of a maximal solution. The passage to a global solution would use a classical contradiction argument and boundedness of the trajectory that we will show in our Lyapunov analysis.

\subsection{Lyapunov analysis}
To study the second order dynamical system \eqref{eq:second_order_dynamic_general}, we  introduce, for $x^* \in \zer Q \neq \emptyset$, the energy function $E^\lambda \colon [t_0, +\infty)\to \R_+$ that will play a fundamental role in our analysis,
\begin{equation}\label{eq:lyapunov_second_order_continuous}
	\begin{aligned}
	&E^\lambda(t) := \delta(t)  \bra  Q(x(t)), x(t)-x^* \ket  + \frac{\xi(t)}{2} \|Q(x(t))\|^2 + \frac{c}{2}\|x(t)-x^*\|^2 + \frac{1}{2}\|v^\lambda(t)\|^2\,, \\
	&v^\lambda(t) := \lambda (x(t)-x^*) + t \big( \dot{x}(t) + \rev{(1-\eta)} \beta(t) Q(x(t))\big)\,,
	\end{aligned}
\end{equation}
for some constant $\lambda \geq 0$, and parameters given by
\begin{equation}\label{eq:def_parameters}
	\delta(t) := \rev{\eta}\lambda \beta(t) t\,,\quad \xi(t) := \eta(1-\eta)\beta(t)^2 t^2\,,\quad c := \lambda(\alpha - 1-\lambda)\,, \quad \text{for} \ t \geq t_0\,.
\end{equation}
This choice of the parameters allows us to establish the following result, whose proof is postponed to Appendix \ref{sec:appendix_lemmas}.
\begin{lem}\label{lem:derivative_of_Lyapunov_second_order}
Let $x^* \in \zer Q \neq \emptyset$, $x\colon [t_0, +\infty)\to H$ be a solution trajectory to \eqref{eq:second_order_dynamic_general}, and $E^\lambda$ the energy function defined in \eqref{eq:lyapunov_second_order_continuous}. Then, for almost all $t \geq t_0$, we have
\begin{equation}\label{eq:lyapunov_analysis_descent}
	\begin{aligned}
		\frac{d}{dt}E^\lambda(t) &= \lambda t \rev{\eta} \left( \dot{\beta}(t) + \frac{2-\alpha}{t}\beta(t)\right)   \bra  Q(x(t)), x(t)-x^* \ket  -\frac{t}{2}(\alpha-1-\lambda) \|\dot{x}(t) + \beta(t) Q(x(t))\|^2 \\
		&  \quad + \beta(t) t^2 \left( (1-\eta)\eta \left(\dot{\beta}(t) +\frac{2-\alpha}{t}\beta(t) \right) + \frac{1}{2}(\alpha-1-\lambda)\frac{\beta(t)}{t}\right) \|Q(x(t))\|^2\\
		& \quad  -\rev{\eta}t^2\beta(t)   \bra   \dot{x}(t), \frac{d}{dt} Q(x(t))  \ket   -  \frac{t}{2}(\alpha-1-\lambda)\|\dot{x}(t)\|^2\,.
	\end{aligned}
\end{equation}
\end{lem}

Lemma~\ref{lem:derivative_of_Lyapunov_second_order} naturally leads us to consider the following \textit{growth condition} for  $t \mapsto \beta(t)$:
\begin{equation}\label{eq:grow_condition}
	0 \leq \frac{\dot{\beta}(t)t}{\beta(t)} \leq \alpha - 2 - \varepsilon\,, \quad \text{for} \ \varepsilon \in [0,\alpha-2]\,, \ \text{and all} \ t \geq t_0\,,
\end{equation}
which with a suitable rearrangement, writes equivalently as
\begin{equation}\label{eq:growth_condition_2}
	\dot \beta(t) + \frac{2-\alpha}{t} \beta(t) \leq {-\varepsilon \frac{\beta(t)}{t}}\,, \quad \text{for all} \ t \geq t_0\,.
\end{equation}
This makes the coefficient in front of the term $\|Q(x(t))\|^2$ negative and bounded away from zero at least for $\lambda$ sufficiently close to $\alpha-1$, hence leading to the desired energy decrease. This observation is crucial to establish the following convergence result.
\begin{thm}\label{thm:second_order_direct}
Let $Q\colon H\to H$ be a monotone and Lipschitz continuous operator with $\zer Q \neq \emptyset$. Let {$x^* \in \zer Q$}, $x\colon [t_0, +\infty)\to H$ be a solution trajectory to \eqref{eq:second_order_dynamic_general} for some $\eta \in (0,1)$ and $\alpha \geq 2$. Suppose that \eqref{eq:grow_condition} holds with $\varepsilon \geq 0$. {Then, it holds}:
	\begin{equation}\label{eq:big_O_rates}
		 \bra  Q(x(t)), x(t)-x^* \ket  =\cO \left(\frac{1}{t\beta(t)}\right) \quad\text{and} \quad  \|Q(x(t))\|\ = \cO \left(\frac{1}{t\beta(t)}\right) \quad \mbox{as} \ t \rightarrow +\infty\,.
	\end{equation}
	Additionally, if $\alpha >2$ and $\varepsilon>0$, then the following holds:
	\begin{enumerate}[label=(\roman*)]
		\item \label{item:integral_bounds} We have the following integral bounds: 
		$$	\int_{t_0}^\infty\beta(t) \bra  Q(x(t)), x(t)-x^* \ket dt < +\infty \,,\quad  \int_{t_0}^\infty t \beta(t)^2 \|Q(x(t))\|^2dt < +\infty\,, \quad  \int_{t_0}^{\infty} t \|\dot{x}(t)\|^2 dt <+\infty\,.$$
		\item \label{item:iters} $x(t)$ converges weakly to a zero of $Q$ as $t\to + \infty$.
		\item \label{item:little_o_rates} The following \textit{fast} rates of convergence hold:
		\begin{equation*}
		\bra  Q(x(t)), x(t)-x^* \ket  =o\left(\frac{1}{t\beta(t)}\right)\,, \quad 
  \|Q(x(t))\| = o\left(\frac{1}{t\beta(t)}\right) \quad \mbox{and} \quad
			\|\dot{x}(t)\| = o\left(\frac{1}{t}\right) \quad \mbox{as} \ t \rightarrow +\infty\,.
		\end{equation*}
	\end{enumerate}
\end{thm}

\rev{
\begin{rmk}[The gap function rate]\label{rmk:gap_function_rate_cont}
The quantity $\bra  Q(x(t)), x(t)-x^* \ket$ is obviously nonnegative by monotonicity of $Q$. To get the significance of the rates related to this quantity, that we dub \emph{gap function}, we first observe that in the simple case where $Q=\nabla f$, where $f$ is a smooth convex function, then $f(x(t))-\min_H f \leq \bra Q(x(t)), x(t) - x^*\ket$ by convexity, whence we recover the same convergence rate on the objective. This remains true even when minimizing $f+g$ where $f$ is an $L$-smooth function and $g$ is a proper lower semicontinuous convex function with $\mathrm{argmin}(f+g) \neq \emptyset$. We then take $Q(x)=\tau^{-1}\left(x - M(x)\right)$, where $\tau \in (0,2/L)$, $M(x)=\prox_{\tau g}\left(x-\tau\nabla f(x)\right)$, and $\prox_{\tau g}$ is the proximal mapping of $\tau g$.

We will see in Section~\ref{sec:splitting} that this gap function plays also a key role in applications to primal-dual-type algorithms, where it serves as a bound for the primal-dual gap, see Example \ref{example:chambolle_pock}.
\end{rmk}
}

\begin{proof}
To lighten notation, we abbreviate $(t)$ with $t$ as a subscript. Additionally, we write $Q_t$ instead of $Q(x(t))$ and $\partial_t$ instead of $\frac{d}{dt}$. Setting $\lambda := \alpha-1$ in \eqref{eq:lyapunov_analysis_descent}, discarding the negative terms, we get using \eqref{eq:growth_condition_2} and monotonicity of $Q$, that for almost all $t \geq t_0$:
	\begin{equation*}
		 \partial_t E^{\alpha-1}_t \leq -\varepsilon(\alpha-1)\rev{\eta} \beta_t \bra  Q_t, x_t-x^* \ket  -  \rev{\eta} t^2 \beta_t  \bra  \dot{x}_t, \partial_t Q_t \ket  \leq 0\,,
	\end{equation*}
\rev{where the last inner product is non-positive since, using monotonicity of $Q$,
\begin{equation}
	-\bra \dot x_t, \partial_t Q_t\ket = - \lim_{h \to 0} \left \langle \frac{x_{t+h}-x_t}{h}, \frac{Q(x_{t+h}) - Q(x_t)}{h} \right \rangle \leq 0\,.
\end{equation}
}
	Thus, $E_t^{\alpha-1}$ is nonincreasing and since all the terms in \eqref{eq:lyapunov_second_order_continuous} are nonnegative, we get
	\begin{equation*}
		(\alpha - 1)\rev{\eta} t\beta_t  \bra  Q_t, x_t - x^* \ket  + \frac{\eta(1-\eta)}{2} \beta_t^2 t^2 \|Q_t\|^2 \leq E^{\alpha-1}_{t_0}\,, \quad \text{for all} \ t\geq t_0 \,,
	\end{equation*}
	hence \eqref{eq:big_O_rates}. 
 
Let us now turn to claims \ref{item:integral_bounds}-\ref{item:little_o_rates}. Suppose that \eqref{eq:grow_condition} holds for some $\varepsilon >0$. Thus, \eqref{eq:lyapunov_analysis_descent}  and monotonicity of $Q$ yield for all $0 \leq \lambda \leq \alpha - 1$ and almost all $t \geq t_0$:
	\begin{equation*}
		\begin{aligned}
			 \partial_t E^\lambda_t  &\leq - \varepsilon \lambda \rev{\eta}  \beta_t  \bra  Q_t, x_t-x^* \ket  -\frac{t}{2}(\alpha-1-\lambda) \| \dot{x}_t \|^2  + \beta_t^2 t \left( -(1-\eta)\eta \varepsilon + \frac{1}{2}(\alpha-1-\lambda)\right) \|Q_t\|^2\,.
		\end{aligned}
	\end{equation*}
	Let $\bar{\lambda}:=\alpha-1-\varepsilon\eta(1-\eta)$. Observe that $\bar{\lambda} \in [3(\alpha-2)/4,\alpha-1]$. We then have 
	\begin{equation*}
		 0 \leq \alpha-1-\lambda \leq  \varepsilon\eta(1-\eta) \,, \quad \text{for all} \ \lambda \in [\bar{\lambda}, \alpha-1]\,.
	\end{equation*}
	It then follows that for all $\lambda \in [\bar{\lambda}, \alpha-1]$ and almost all $t \geq t_0$:
	\begin{equation}\label{eq:E_dot_2}
	\partial_t E^\lambda_t \leq -\lambda \rev{\eta} \varepsilon \beta_t  \bra  Q_t, x_t-x^* \ket  -\frac{ t}{2}(\alpha-1-\lambda) \| \dot{x}_t \|^2 - \frac{1}{2}\beta_t^2 t (1-\eta)\eta \|Q_t\|^2\,.
	\end{equation}
	Integrating this inequality and using that the energy function $E^\lambda$ is nonnegative, we get \ref{item:integral_bounds}.\smallskip
	
	To show \ref{item:iters}, we start by expressing $E_t^{\lambda} - E_t^{\alpha-1}$ which reads
	\begin{equation*}
		E_t^{\lambda} - E_t^{\alpha-1}  = -\rev{\eta}(\alpha-1-\lambda) \beta_t t  \bra  Q_t, x_t-x^* \ket  \\
		 + \frac{\lambda}{2}(\alpha-1-\lambda)\|x_t-x^*\|^2 + \frac{1}{2}\|v_t^\lambda\|^2 - \frac{1}{2}\|v_t^{\alpha-1}\|^2 .
	\end{equation*}
	Using the three point identity, we get
	\begin{align*}
	\frac{1}{2}\|v_t^\lambda\|^2 - \frac{1}{2}\|v_t^{\alpha-1}\|^2 
	&= - \frac{1}{2}\|v_t^\lambda-v_t^{\alpha-1}\|^2 + \bra v_t^\lambda,v_t^\lambda-v_t^{\alpha-1}\ket \\
	&=  - \frac{(\alpha-1-\lambda)(\alpha-1+\lambda)}{2}\|x_t-x^*\|^2 -(\alpha-1-\lambda)t \bra  x_t-x^*, \dot{x}_t + \rev{(1-\eta)} \beta_t Q_t  \ket . 
	\end{align*}
	Plugging this above we arrive at
	\begin{equation}\label{eq:difference_lyapunovs}
		\begin{aligned}
		E_t^{\lambda} - E_t^{\alpha-1} 	& = - \rev{\eta} (\alpha - 1 - \lambda) \beta_t t  \bra  Q_t, x_t-x^* \ket \\
		& \quad -\frac{(\alpha-1)(\alpha-1-\lambda)}{2} \|x_t-x^*\|^2 -(\alpha-1-\lambda)t \bra  x_t-x^*, \dot{x}_t + \rev{(1-\eta)} \beta_t Q_t  \ket \\
		& =  -(\alpha-1-\lambda) \beta_t t  \bra  Q_t, x_t-x^* \ket \\
		& \quad -\frac{(\alpha-1)(\alpha-1-\lambda)}{2} \|x_t-x^*\|^2 - (\alpha-1-\lambda)t \bra  x_t-x^*, \dot{x}_t  \ket \,.
	\end{aligned}
	\end{equation}
	Therefore, we get for almost all $t \geq t_0$
	\begin{equation}\label{eq:difference_lyapunovs_2}
		E_t^\lambda - E_t^{\alpha-1} = -(\alpha-1-\lambda) \Big[(\alpha-1)q(t) + t\dot{q}(t) -(\alpha-1)I(t)\Big]\,,
	\end{equation}
	where we have denoted:
	\begin{equation*}
			q(t):= \frac{1}{2}\|x_t-x^*\|^2 + I(t)\,,\quad \text{and} \quad I(t):= \int_{t_0}^t \beta_s \bra  Q_s, x_s-x^* \ket  ds\,.
	\end{equation*}	
	Note that $I(t)$ admits a limit as $t \rightarrow +\infty$ by \ref{item:integral_bounds}. Moreover, from \eqref{eq:E_dot_2} we also get that for each $\lambda \in [\bar{\lambda}, \alpha -1]$ the limit $\lim_{t\to +\infty} E_t^\lambda$ exists. In particular the limit of $E_t^{\lambda} - E_t^{\alpha-1}$ exists. Hence, we deduce from \eqref{eq:difference_lyapunovs_2} that $\lim_{t \to +\infty} q(t) + \frac{t}{\alpha-1} \dot{q}(t)$ exists. This implies, thanks to \cite[Lemma~A.2]{apr16} and since $\alpha >2$, that also $\lim_{t\to +\infty} q(t)$ exists. In view of the definition of $q(t)$, $I(t)$ admitting a limit, we get that $\lim_{t\to+\infty}\|x_t-x^*\|$ exists. $x(t)$ is in particular bounded, and one can extract a weakly convergent subsequence. Using \eqref{eq:big_O_rates} and that $t \mapsto \beta(t)$ is nondecreasing, it is immediate to see that the weak cluster point of each such weakly converging subsequence is a zero of $Q$. We now invoke the standard Opial's lemma to get \ref{item:iters}. \smallskip
	
    Let us now establish the little-$o$ rates in \ref{item:little_o_rates}. To do so, we first note that, since the limit of $t \mapsto q(t)$ exists, from \eqref{eq:difference_lyapunovs_2} we get:
	\begin{equation*}
		\lim_{t +\infty } t \dot{q}(t) = \lim_{t\to +\infty}  t \bra  \dot{x}_t + \beta_t Q_t, x_t-x^* \ket  \in \R\,.
	\end{equation*}
	Therefore, since for almost all $t \geq t_0$
	\begin{equation*}
		\begin{aligned}
			E_t^{\alpha-1}&= \frac{1}{2}\|(\alpha-1) (x_t-x^*) + t( \dot{x}_t+ \rev{(1-\eta)}\beta_t Q_t)\|^2  + \rev{\eta}(\alpha-1) \beta_t t  \bra  Q_t, x_t- x^* \ket  + \frac{\eta}{2}(1-\eta)\beta_t^2t^2\|Q_t\|^2\\
			&= \frac{1}{2}(\alpha-1)^2\| x_t-x^*\|^2 + (\alpha-1) t  \bra  x_t-x^*, \dot{x}_t + \beta_t Q_t \ket\\
			& \quad  +\frac{t^2}{2}\|\dot{x}_t + \rev{(1-\eta)} \beta_t Q_t\|^2 + \frac{\eta}{2}(1-\eta)\beta_t^2t^2\|Q_t\|^2\,,
		\end{aligned}
	\end{equation*}
	we deduce that
	\begin{equation*}
		0 \leq \ell := \lim_{t\to +\infty} \ t^2\left( \| \dot{x}_t+ \rev{(1-\eta)} \beta_t Q_t\|^2 + \eta(1-\eta )\|\beta_t Q_t\|^2\right)
	\end{equation*}
	exists.	Suppose that $\ell > 0$. Then, there exists $s \geq t_0$ such that
\begin{align*}
\int_{t_0}^{+\infty}t&\left( \| \dot{x}_t+\rev{(1-\eta)} \beta_t Q_t\|^2 + \eta(1-\eta )\|\beta_t Q_t\|^2\right) dt\\
&\geq \int_{s}^{+\infty}t^2\left( \| \dot{x}_t+\rev{(1-\eta)} \beta_t Q_t\|^2 + \eta(1-\eta )\|\beta_t Q_t\|^2\right) t^{-1} dt \geq \int_{s}^{+\infty}\frac{\ell}{2t} dt = +\infty \,,
\end{align*}
leading to a contradiction with the last two integral bounds in \ref{item:integral_bounds}. This shows that $\ell=0$ and combining this with Jensen's inequality proves the last two little-$o$ rates in \ref{item:little_o_rates}.  \rev{The first rate follows from the second by applying Cauchy--Schwarz and using the boundedness of the trajectories.}
\end{proof}
\rev{\begin{rmk}[On the growth of $\beta$]
	The parameter function $\beta$ can be understood as a \emph{time-scaling} parameter \cite{acr19_siopt}. As such, it directly impacts the rate of convergence, as shown in Theorem \ref{thm:second_order_direct}\ref{item:little_o_rates}. However, contrarily to classical \emph{autonomous} first-order dynamical systems, where the time-scaling parameter can grow arbitrarily fast while maintaining convergence, the non-autonomous nature of our second-order dynamical system---particularly the fixed decay rate of the vanishing damping parameter  $\frac{\alpha}{t}$---imposes a fundamental constraint on the growth of $\beta$. Indeed, applying Gr\"onwall's inequality to integrate \eqref{eq:grow_condition}, we get that $\beta(t)$ must satisfy
    \[
    \beta(t) \leq \beta_0 t^{\alpha-2}
    \]
    for $\beta_0 >0$.
    Clearly, the growth of $\beta(t)$ depends on the vanishing damping parameter $\alpha$. Choosing \( \beta(t) = t^{\alpha - 2 - \varepsilon} \) with \( \varepsilon \in (0, \alpha - 2) \), we can achieve arbitrarily fast decay rates \( o(t^{-\nu}) \), where \( \nu := \alpha - 1 - \varepsilon \), by increasing \( \alpha \).
\end{rmk}}

\rev{
\subsection{Edge cases and connection to Tikhonov regularization}\label{sec:connection_to_tichonov}

It is instructive to inspect the edge cases in Theorem~\ref{thm:second_order_direct}. We can recognize two different boundary regimes: $\theta \in \{1, \alpha-1\}$, which corresponds to $\eta \in \{0, 1\}$, and $\varepsilon = 0$. 

\subsubsection{Edge cases in $\theta$}\label{sec:edge_in_theta} As observed by Attouch, Chbani and Riahi in \cite{acr19_siopt}, as well as in \cite{bot2023, botkhoa2023}, while the analysis of the continuous-time system \eqref{eq:second_order_dynamic_general} can be performed with general parameter functions $\beta$ satisfying the growth condition \eqref{eq:grow_condition}, constant parameters $\beta$ are the only choices currently compatible with explicit numerical schemes. For this reason, we focus here on this important setting, and, without loss of generality, we set $\beta=1$. We have the two edge cases:

\begin{itemize}[left=0pt]
\item \textbf{$\theta=1$:} We get the system
\begin{equation}\label{eq:second_order_theta_1}
	\ddot x(t) + \frac{\alpha}{t}\dot x(t) + \frac{d}{dt}Q(x(t)) + \frac{1}{t}Q(x(t)) = 0\,, \quad \text{for} \ t \geq t_0\,.
\end{equation}
Multiplying \eqref{eq:second_order_theta_1} by $t$, we can rewrite \eqref{eq:second_order_theta_1} as
\begin{equation}
	\frac{d}{dt}\Big(t \dot x(t) + tQ(x(t))\Big) + (\alpha-1)\dot x(t) = 0\,, \quad \text{for} \ t \geq t_0\,, 
\end{equation}
which, integrating from $t_0$ to $t$ leads us to
\begin{equation}
	\dot x(t) + Q(x(t)) + \frac{(\alpha-1)}{t}(x(t) - v)=0\,, \quad \text{for} \ t \geq t_0\,,
\end{equation}
where $v:= x_0 + \frac{t_0}{\alpha-1}(\dot{x}(t_0) + Q(x_0))$. This is the so-called \emph{anchored} or \emph{Tikhonov regularization} of the  flow \eqref{eq:monotone_flow}, which is known to provide i) strong convergence to the orthogonal projection of $v$ onto $\zer Q$ and ii) the rate $\|Q(x(t))\|=\cO(t^{-1})$ for all $\alpha \geq 2$, see, e.g., \cite{bk24, bc24}. 

The rate on the residual cannot be improved to a little-$o$ rate, as the following simple counterexample in dimension 2 shows: Pick $\alpha =3$ and $Q\colon \R^2\to \R^2$ being the $90^\circ$ counter-clockwise rotation, which is indeed monotone and Lipschitz. Identifying $x(t)$ with a complex number, and $Q(x(t))$ as $i x(t)$, the general solution  $x\colon [t_0,+\infty)\to \mathbb{C} \cong \R^2 = H$ of \eqref{eq:second_order_theta_1} can be obtained explicitly and is defined for $c_1, c_2 \in \R$, as
\begin{equation}
	x(t) = \frac{c_1 (1 - i t)}{t^2} + \frac{c_2}{t^2} e^{-i t}\,, \quad \text{for $t\geq t_0$}\,.
\end{equation}
As it can be immediately observed, this only satisfies $t\|Q(x(t))\|= t \|i x(t)\| = t\|x(t)\| \geq  \frac{c_1}{2}$ for all $t \geq t_1$ and $t_1>t_0$ large enough. Therefore, the little-$o$ rates \emph{cannot} be achieved in this setting.

\item \textbf{$\theta=\alpha-1$:} It leads us to the system:
\begin{equation}\label{eq:second_order_theta_2}
	\ddot x(t) + \frac{\alpha}{t}\dot x(t) + \frac{d}{dt}Q(x(t)) + \frac{\alpha-1}{t}Q(x(t)) = 0\,, \quad \text{for} \ t \geq t_0\,.
\end{equation}
The same choice as in the first point above, i.e., with $Q(x(t))=ix(t)$, $\alpha = 3$, and $H=\R^2\cong \mathbb{C}$, leads us to a general solution $x\colon [t_0, +\infty)\to \mathbb{C}$ defined for $c_1, c_2 \in \R$, as,
\begin{equation}
	x(t) = \frac{c_1}{t^2} + \frac{c_2  (1 + i t)}{t^2}e^{-i t}\,, \quad \text{for} \ t \geq t_0\,,
\end{equation}
which, similarly as before, only satisfies $t\|Q(x(t))\|=t\|ix(t)\|=t\|x(t)\|\geq \tfrac{c_2}{2}$ for all $t\geq t_1$ and $t_1>t_0$ large enough. Also in this edge case no little-$o$ rates on the residual can be achieved. However, differently from the case $\theta=1$, the dynamics \eqref{eq:second_order_theta_2} does not seem to reduce to a known first-order dynamical system and its asymptotic properties for general monotone Lipschitz operators remain an open problem. 
\end{itemize}

\subsubsection{Saturating the growth condition}\label{sec:edge_in_eps}
Let us now consider a non-constant parameter $\beta$, but saturating the growth condition \eqref{eq:grow_condition}, i.e., with $\varepsilon = 0$, which yields, for $\alpha \geq 2$,
\begin{equation}
	\dot{\beta}(t) + \frac{\beta(t)}{t}(2-\alpha) = 0\,, \quad \text{which is solved for} \quad \beta(t) = \beta_0 t^{\alpha-2}\,, \quad \text{for} \ \beta_0 > 0\,.
\end{equation}
One can check that in this case $b(t)=(1-\eta)\dot{\beta}(t) + \rev{\theta} \frac{\beta(t)}{t} = \dot{\beta}(t) + \frac{\beta(t)}{t}$ and thus \eqref{eq:second_order_dynamic_general} reads as
\begin{equation}\label{eq:second_order_dynamic_general_emphasized}
	\ddot{x}(t) + \frac{\alpha}{t}\dot{x}(t) + \beta(t) \frac{d}{dt} Q(x(t)) + \Big(\dot{\beta}(t) + \frac{1}{t}\beta(t)\Big) Q(x(t))=0\,, \quad \text{for} \  t \geq t_0\,.
\end{equation}
Multiplying \eqref{eq:second_order_dynamic_general_emphasized} by $t$ we shall write:
\begin{equation}\label{eq:saturating_growth_derivative_form}
	\begin{aligned}
		0 &= t\ddot{x}(t) + \dot{x}(t) + t\beta(t) \frac{d}{dt} Q(x(t)) + \left(t\dot{\beta}(t) + \beta(t)\right) Q(x(t)) +  (\alpha-1) \dot{x}(t)=0\\
		& = \frac{d}{dt} \Big[ t \big(\dot{x}(t) + \beta(t) Q(x(t))\big)\Big] + (\alpha - 1)\dot{x}(t)\,.
	\end{aligned}
\end{equation}
Integrating \eqref{eq:saturating_growth_derivative_form} from $t_0$ to $t$, we unveil two interesting properties. First, if $\alpha = 1$ the system is fundamentally equivalent to a first-order system with perturbation and time-scaling (see \cite{attouch2023fast}). Secondly, if $\alpha > 1$, we can rewrite it as
\begin{equation}\label{eq:Tychonov_flow}
	\dot{x}(t) + \beta_0 t^{\alpha - 2} Q(x(t)) + \frac{\alpha-1}{t}\left(x(t)-{v}\right)  = 0\,, \quad t \geq t_0\,,
\end{equation}
where $v:= \frac{t_0}{\alpha-1}(\dot{x}(t_0) + \beta(t_0) Q(x(t_0))) + x(t_0)$. We are once again in the setting of Tikhonov regularization and continuous-time version of Halpern method: The system \eqref{eq:Tychonov_flow} generates trajectories that converge \textit{strongly} to the projection of $v$ onto the set of zeros of $Q$ with convergence rate $\|Q(x(t))\|=\cO(t^{-(\alpha-1)})$, see, e.g., \cite{cps08} and \cite[Corollary 2.16.]{bk24} for the rate.
}

\section{Convergence Analysis of Algorithm~\ref{alg:fast_km}}\label{sec:algconv}

In this section, we present the convergence analysis of Algorithm~\ref{alg:fast_km}. Let us first show that the latter can indeed be understood as a discretization of \eqref{eq:second_order_dynamic_general}. Let $(x^k)_{k \in \N}$ be the sequence generated by Algorithm~\ref{alg:fast_km}, $Q:=I-T$ and, for all \rev{$k \geq -1$}, set
\begin{equation}\label{eq:substitutions}
	z^{k+1} := T(x^k)\,, \quad \text{and} \quad Q^{k+1}:=Q(x^k) =(I-T)(x^{k})= x^k - z^{k+1}\,.
\end{equation}
Recall that Algorithm~\ref{alg:fast_km} reads, for all \rev{$k \geq 0$}, as
\rev{
\begin{equation}
	x^{k+1}-x^{k}=\frac{\theta}{k+\sigma} \big(T(x^k)-x^k\big) + \Big(1 -\frac{\alpha}{k+\sigma}\Big) \big(T(x^k)-T(x^{k-1})\big)\,.
\end{equation}
Therefore, using the definition of $Q$ and \eqref{eq:substitutions}, we get:
\begin{equation}\label{eq:xkzk}
	z^{k+2}-z^{k+1} + Q^{k+2}-Q^{k+1} = -\frac{\theta}{k+\sigma} Q^{k+1} + \Big(1 -\frac{\alpha}{k+\sigma}\Big)(z^{k+1}-z^k)\,.
\end{equation}}
Now, bringing every term to the left-hand side, we get:
\begin{equation}
	z^{k+2}- 2z^{k+1}+z^k + \frac{\alpha}{k+\sigma}(z^{k+1}-z^k) +   \left({Q}^{k+2}-{Q}^{k+1}\right) +   \frac{\rev{\theta}}{k+\sigma}{Q}^{k+1}=0\,.
\end{equation}
Denoting by $t_{k}:=k - 1 + \sigma$, we get for all \rev{$k \geq 0$}:
\begin{equation}\label{eq:discrete_second_order}
	\left(z^{k+2}- 2z^{k+1}+z^k\right) + \frac{\alpha}{t_{k+1}}\left(z^{k+1}-z^k\right) +  \left(Q^{k+2}-Q^{k+1}\right) + \frac{\rev{\theta}}{t_{k+1}} Q^{k+1}=0\,,
\end{equation}
which, when $z^{k+1} \simeq x^k$ (asymptotically the case), can be understood as an explicit discretization of \eqref{eq:second_order_dynamic_general} with $\beta \equiv 1$, step size $1$, and $Q:=I-T$, which is Lipschitz continuous and monotone. \rev{This explicit discretization differs from the one proposed in \cite{botkhoa2023} in the specific case $\theta = \frac{\alpha}{2}$. The latter relies on a clever product-space reformulation of the second-order dynamical system, whereas \eqref{eq:discrete_second_order} preserves the structural form of \eqref{eq:second_order_dynamic_general}. As we will see in Section~\ref{sec:Lyapunov_energy_discrete}, this will allow us to construct a new energy function that almost exactly mirrors the continuous-time one in \eqref{eq:lyapunov_second_order_continuous}, all the while allowing for twice larger relaxation parameters.}
\rev{
\begin{rmk}[Resolvent viewpoint]\label{rmk:resolvent_viewpoint}
	Consider the multivalued operator $\cQ \colon H\to 2^H$ defined by $\cQ:=T^{-1}-I$. In this way, $T$ can be represented as the resolvent operator of $\cQ$, i.e., $T=J_{\cQ}:=(I+\cQ)^{-1}$. Building on Minty's classical result \cite{minty62}, nonexpansivity properties of $T$ translate to monotonicity properties for $\cQ$, and vice-versa. Specifically, if $T$ is firmly nonexpansive, then $\cQ$ is maximal monotone. More generally, if $T$ is only nonexpansive, then $\cQ$ is $\frac{1}{2}$-\emph{comonotone}, i.e., 
	\begin{equation}
		\bra q' - q, z' - z\ket \geq -\frac{1}{2}\|q' - q\|^2\,, \quad \text{for all} \ (z, q), (z', q') \in \gra \cQ\,,
	\end{equation} 
	see \cite{bmw21} for a comprehensive discussion. Using this notation, we can write:
	\begin{equation}
		Q^{k+1} =(I-T)(x^k) \in \cQ(T(x^k))=\cQ(z^{k+1})\,, \quad \text{for all} \ k \geq 0\,.
	\end{equation}
	 Comparing now with \eqref{eq:discrete_second_order}, despite $\cQ$ not being necessarily Lipschitz continuous, not even single-valued, Algorithm \ref{alg:fast_km} can be understood as a \emph{semi-}implicit finite-difference discretization of \eqref{eq:second_order_dynamic_general}---implicit on the ``Hessian''-driven damping term $\frac{d}{dt}Q(x(t))$ and explicit on the force term $\frac{\theta}{t}Q(x(t))$---and with $\beta \equiv 1$, step size $1$, and $Q:=\cQ$, now formulated on a single sequence $(z^k)_{k \in \N}$.
	 
	 The main difference with the implicit discretization of the Fast-OGDA system proposed in \cite{bot2023} is that the latter, being implicit also on the force term, does not lead to an easily implementable numerical method as it involves computing the resolvent of $\cQ$ with an iteration-dependent step size, which is hardly possible in practice. While being hard to implement in practice, the approach in \cite{bot2023} allows to consider more general time-scaling parameters $\beta_k$, yielding improved rates in the same fashion as in the classical G\"{u}ler acceleration for proximal-point algorithms \cite{guler91, guler92}, see also \cite[Remark 2]{bot2023} for a more complete account on the role of $\beta$. 
\end{rmk}}

\subsection{Lyapunov Energy}\label{sec:Lyapunov_energy_discrete}

To analyze Algorithm~\ref{alg:fast_km}, we denote finite differences with the $\Delta$ symbol, i.e., for a sequence $(a^k)_{k \in \N}$, we define $\Delta a_k := a^{k+1}-a^k$ for all $k \geq 0$. Second order differences are denoted by $\Delta^2$, i.e., $\Delta^2 a_k = \Delta (\Delta a)_{k} := \Delta a_{k+1}-\Delta a_k$. These notations will prove handy in our analysis since we follow closely the continuous-time Lyapunov analysis. In this sense, it is convenient to recall some discrete-time calculus rule listed in Appendix \ref{sec:discrete_calculus}.

Inspired by the continuous-time approach, we introduce an energy function that matches almost exactly\footnote{Yet another difference with prior work and in particular with \cite{bot2023, acfr20} is that the term in $v^k$ is scaled differently.} \eqref{eq:lyapunov_second_order_continuous}. Specifically, we consider for all \rev{$k \geq -1$} the sequences $z^{k+1}:=T(x^k)$ and $Q^{k+1}:=Q(x^{k})$, where $Q:=I-T$, and for all \rev{$k \geq 1$}, $\lambda \in (0, \alpha- 1]$, and any $z^* \in \zer Q=\Fix T$, the energy function:
\begin{equation}\label{eq:nonexp_discrete_lyapunov}
	\begin{aligned}
		E_k^\lambda &:= \delta_k  \bra  Q^k, z^k-z^* \ket  + \frac{\delta_k}{2}\|Q^k\|^2 + \frac{\xi_k}{2}\|Q^k\|^2 + \frac{c}{2}\|z^{k-1} - z^*\|^2 + \frac{1}{2}\|v^k\|^2\,,\\
		v^k&:= \lambda (z^{k-1}-z^*) + t_{k-1}\left( \Delta z_{k-1}  + \rev{(1-\eta)} Q^k\right)\,, 
	\end{aligned}
\end{equation}
where $t_{k}:=k -1 + \sigma$, \rev{a constant $\lambda\geq 0$, and parameters given by
\begin{equation}
	\delta_k := \eta\lambda t_{k-1}\,, \quad \xi_k := (1-\eta)\eta t_{k-1}^2\,, \quad c := \lambda(\alpha - 1-\lambda)\,, \quad \text{for} \ k \geq 1\,.
\end{equation}}
Note that the only substantial difference with \eqref{eq:lyapunov_second_order_continuous} is that the first inner product $ \bra  Q^k, z^k - z^* \ket $ is no longer nonnegative, since $Q^k=Q(x^{k-1})$, i.e. $Q$ is evaluated at $x^{k-1}$ rather than at $z^k$. \rev{Observe that, in light of Remark \ref{rmk:resolvent_viewpoint}, this inner product would be nonnegative as soon as $T$ is firmly nonexpansive, in which case $Q^k \in \cQ(z^k)$, with $\cQ$ maximal monotone.} The second term allows us to make $E_k^\lambda$ nonnegative even for general nonexpansive operators. Indeed, observing that $Q/2$ is firmly nonexpansive, it follows from \cite[Proposition~4.4]{BCombettes} that for every $x, y \in H$: 
\begin{equation}\label{eq:non_monotonicity_of_Q}
	\begin{aligned}
		 \bra Q(x)-Q(y),T(x)-T(y) \ket 
		 &= \bra Q(x)-Q(y),x-y \ket - \|Q(x)-Q(y)\|^2  \geq -\frac{1}{2}\|Q(x)-Q(y)\|^2\,.
	\end{aligned}
\end{equation}
Using this with $x=x^{k-1}$ and $y=z^* \in \Fix T$, hence $Q(z^*)=0$, we get that $\bra  Q^k, z^k-z^* \ket  + \frac{1}{2}\|Q^k\|^2 \geq 0$.

\subsection{Energy Decrease}\label{sec:discrete_lyapunov_analysis} We start with the following preliminary result, which represents the discrete-time counterpart of Lemma \ref{lem:derivative_of_Lyapunov_second_order}, \rev{and our first main contribution:}
\begin{lem}\label{lem:nonexp_derivative_of_Lyapunov_second_order_discrete}
	Let $T\colon H\to H$ be a nonexpansive operator with $\Fix T \neq \emptyset$. Let $(x^k)_{k \in \N}$ be a sequence generated by Algorithm~\ref{alg:fast_km} for some $\eta \in (0, 1)$, $\sigma > 0$, and $\alpha\geq 2$. Then there exists $\underline{\lambda}< \alpha-1$ such that for all $\lambda \in (\underline{\lambda}, \alpha-1]$ and all \rev{$k \geq 1$} such that $t_k \geq \rev{\lambda}$ it holds:
	\begin{equation}\label{eq:nonexp_descent_E_claim}
		\begin{aligned}
			E_{k+1}^\lambda-E_k^{\lambda} & \leq - \lambda \rev{\eta}\left(\alpha - 2\right) \left(\bra x^{k-1}-T(x^{k-1}), T(x^{k-1}) - z^* \ket + \frac{1}{2} \|x^{k-1}-T(x^{k-1})\|^2\right) \\
			&\quad \bigg[ - \eta (1-\eta) (\alpha-2) t_k  -\frac{\rev{\eta}}{2}\Big(\rev{1 - \eta} + \rev{\eta}(\alpha-1)^2 -\lambda(\alpha-1) \Big) \bigg]\|x^{k-1} - T(x^{k-1})\|^2 \\
			& \quad +  (\alpha-1-\lambda)\frac{t_k-\left(2\rev{\eta}\rev{\theta} + \rev{\eta}+\rev{\theta}\right)}{2}\|x^{k-1}-T(x^{k-1})\|^2 \\
			& \quad - (\alpha-1-\lambda)\frac{t_k-\left(\alpha+2+\rev{\eta}(\alpha-1)\right)}{2}\|T(x^{k-1}) - T(x^{k-2})\|^2 \, .
		\end{aligned}
	\end{equation}
\end{lem}
Observe that the first term on the right-hand side is negative thanks to \eqref{eq:non_monotonicity_of_Q}. The proof of Lemma \ref{lem:nonexp_derivative_of_Lyapunov_second_order_discrete}, which is postponed to Appendix \ref{sec:appendix_lemmas}, follows exactly the continuous-time approach with $Q:=I-T$. We are now ready to establish the discrete counterpart of Theorem~\ref{thm:second_order_direct}.

\begin{thm}\label{thm:nonexp_second_order_discrete}
	Let $T\colon H\to H$ be a nonexpansive operator with $\Fix T \neq \emptyset$. Let $(x^k)_{k \in \N}$ be a sequence generated by Algorithm~\ref{alg:fast_km} for some $\eta \in (0, 1)$, $\sigma > 0$, and $\alpha\geq 2$. Then:
	\begin{enumerate}[label=(\roman*)]
		\item \label{item:nonexp_halpern} We have $\|x^k - T(x^k)\| = \cO(k^{-1})$ as $k \to +\infty$.
	\end{enumerate}
	Suppose now that $\alpha > 2$ \rev{and let $z^* \in \Fix T$}. Then:
	\begin{enumerate}[label=(\roman*)]
		\setcounter{enumi}{1}
		\item \label{item:nonexp_discr_summability} We have the summability properties:
        \begin{equation*}
            \begin{aligned}
             &\sum_{k \in \N} \left(\bra  x^{k}-T(x^{k}), T(x^{k}) - z^* \ket + \frac{1}{2} \| x^{k}-T(x^{k})\|^2\right) < +\infty\,,\\
             &\sum_{k \in \N} k\|x^k - T(x^k)\|^2 < +\infty\,,  \quad \text{and} \quad \sum_{k \in \N}k \|x^k - x^{k-1}\|^2  < +\infty\,.
            \end{aligned}
        \end{equation*}
		\item \label{item:nonexp_discr_iters} The sequence $(x^k)_{k \in \N}$ converges weakly to a fixed point of $T$ as $k\to + \infty$.
		\item \label{item:nonexp_discr_rates} We have the following \textit{fast} rates of convergence:
		\begin{equation*}
			\begin{aligned}
			&\rev{\bra  x^{k}-T(x^{k}), T(x^{k}) - z^* \ket + \frac{1}{2} \| x^{k}-T(x^{k})\|^2 = o(k^{-1})\,,}\\
			&\|x^k-x^{k-1}\|  = o(k^{-1}) \quad \mbox{and} \quad \|x^k - T(x^k)\| = o(k^{-1}) \quad \mbox{as} \ k \rightarrow +\infty\,.
			\end{aligned}
		\end{equation*}
	\end{enumerate}
\end{thm}
\begin{proof}
As in the proof of Lemma \ref{lem:nonexp_derivative_of_Lyapunov_second_order_discrete}, it is convenient to introduce the additional \rev{sequence} $(z^k)_{k \in \N}$. We recall that the latter is defined by $z^{k+1}:= T(x^k)$ and thus satisfies \eqref{eq:discrete_second_order} with $Q^{k+1}:= Q(x^k):=(I-T)(x^k)=x^k-z^{k+1}$ for all \rev{$k \geq 0$}.\smallskip 
	
\ref{item:nonexp_halpern} \rev{We set $\lambda = \alpha-1$ and \eqref{eq:nonexp_descent_E_claim} becomes, after elementary computations,
\begin{equation}\label{eq:nonexp_descent_E_alpha>2}
\Delta E_k^{\alpha-1} \leq - \rev{\eta}(\alpha - 1)(\alpha - 2) \left(\bra  Q^k, z^k - z^* \ket + \frac{1}{2} \| Q^k\|^2\right)	-\eta(1-\eta)(\alpha-2) \Big( t_k - \frac{\alpha}{2}\Big)\|Q^k\|^2\,,
\end{equation}
for all $k\geq 1$ such that $t_k \geq \alpha-1$, which also implies $t_k \geq \frac{\alpha}{2}$ as $\alpha \geq 2$. Therefore, using \eqref{eq:non_monotonicity_of_Q}, there is $\bar k \geq 1$ large enough such that for all $k \geq \bar k$ we have $\Delta E_k^{\alpha-1} \leq 0$. Therefore, for all $k \geq \bar{k}$, we have $\frac{\xi_{k+1}}{2}\|Q^{k+1}\|^2 \leq E_{k+1}^{\rev{\alpha-1}} \leq E_{\bar{k}}^{\rev{\alpha-1}}$ whence we get the claim recalling the expressions of $\xi_k$ and $t_k$.}
 
Suppose from now on that $\alpha > 2$.\smallskip
	
\ref{item:nonexp_discr_summability} Setting again $\lambda = \alpha - 1$ and embarking from \eqref{eq:nonexp_descent_E_alpha>2}, we have after summing on both sides that
\begin{align*}
\rev{\eta}(\alpha - 1)(\alpha - 2) \sum_{i = \bar{k}}^k\left(\bra  Q^k, z^k - z^* \ket + \frac{1}{2} \| Q^k\|^2\right)& + \eta (1-\eta)(\alpha-2)\sum_{i = \bar{k}}^k (t_i - t_{\bar{k}})\|Q^i\|^2\\
\leq E_{\bar{k}}^{\alpha-1} - E_{k+1}^{\alpha-1} \leq E_{\bar{k}}^{\alpha-1} < +\infty \,.
\end{align*}
Passing to the limit as $k \to +\infty$ we get the first two summability claims. For the last one, we set $\lambda = \alpha - 1 - \varepsilon$, for $\varepsilon > 0$ to be chosen small enough. Then, discarding the first negative term in \eqref{eq:nonexp_descent_E_claim}, we have
\begin{multline}\label{eq:nonexp_descent_E_alpha>2_2}
	\Delta E_k^{\alpha-1-\varepsilon} \leq \bigg[ - \left(\eta (1-\eta)(\alpha-2) - \frac{\varepsilon}{2}\right) t_k  + \frac{\eta(1-\eta)}{2}\left((\alpha-1)^2-1\right)\bigg]\|Q^k\|^2 \\
	- \varepsilon\frac{t_k-\left(\alpha+2+\rev{\eta}(\alpha-1)\right)}{2}\|\Delta z_{k-1}\|^2\,.
	\end{multline}
	Therefore, for $0 < \varepsilon < \eta (1-\eta)(\alpha - 2)$, there exists \rev{$\tilde{k} \geq 1$} large enough such that for all $k \geq \tilde{k}$,
	\begin{equation*}
	\Delta E_k^{\alpha-1-\varepsilon} \leq - \frac{\varepsilon}{2}(t_k-t_{\tilde{k}})\|\Delta z_{k-1}\|^2 \,,
	\end{equation*}
	where we have discarded the first negative term in \eqref{eq:nonexp_descent_E_alpha>2_2}. Summing again we get $\sum_{k \in \N} k\|\Delta z_{k-1}\|^2 < +\infty$. Observe now that since $Q^{k+1}=x^k-z^{k+1}$ by definition, we have 
	\[
	\|\Delta x_{k-1}\|^2  = \|x^k - x^{k-1}\|^2 = \| \Delta z_{k} +  \Delta Q_{k} \|^2 \leq 2\|\Delta z_{k}\|^2 + 4\|Q^{k+1}\|^2 + 4 \|Q^{k}\|^2 .
	\]
	Multiplying by $k$, summing and using the summability claims proved above, we conclude.
	\smallskip
	
\ref{item:nonexp_discr_iters} We argue as in the continuous-time setting and set $\lambda = \alpha - 1 - \varepsilon$, for $\varepsilon > 0$ small enough as before. We will emphasize the dependency of $(v^k)_{k \in \N}$ on $\lambda$ by writing explicitly $(v^k_\lambda)_{k \in \N}$. Let \rev{$k \geq 1$}. We have
\begin{equation*}
\begin{aligned}
    E_k^{\lambda}-E_k^{\alpha-1}& =-\rev{\eta}(\alpha-1-\lambda)  t_{k-1}\left( \bra  Q^k, z^k-z^* \ket  + \frac{1}{2}\|Q^k\|^2\right)\\
			&\quad + \frac{\lambda}{2}(\alpha-1-\lambda)\|z^{k-1}-z^*\|^2 + \underbrace{\frac{1}{2}\|v_\lambda^k\|^2 - \frac{1}{2}\|v_{\alpha-1}^k\|^2}_{(\labterm{lb:thmc_I}{I})}\,.
\end{aligned}
\end{equation*}
Let us now compute \eqref{lb:thmc_I} separately. By the three point identity, we get
	\begin{equation*}
			\eqref{lb:thmc_I}:=\frac{1}{2}\|v_\lambda^k\|^2 - \frac{1}{2}\|v_{\alpha-1}^k\|^2
			= \frac{\lambda^2-(\alpha-1)^2}{2}\|z^{k-1}-z^*\|^2 - t_{k-1}(\alpha-1-\lambda)  \bra    \Delta z_{k-1}  + \rev{(1-\eta)}   Q^k, z^{k-1}-z^*  \ket  \,.
	\end{equation*}
	Therefore,
	\begin{align*}
			E_k^{\lambda}-E_k^{\alpha-1} &=-(\alpha-1-\lambda)\Big[t_{k-1}\bra  Q^k, z^k - z^* - \rev{(1-\eta)}\Delta z_{k-1}\ket + \frac{\rev{\eta}}{2}t_{k-1}\|Q^k\|^2\\
			&\quad +\frac{(\alpha-1)}{2}\|z^{k-1}-z^*\|^2 + t_{k-1}\bra \Delta z_{k-1} , z^{k-1}-z^*  \ket  \Big] \,.
	\end{align*}
	Therefore, this can be written as
	\begin{equation}\label{eq:DeltaEkqk}
    \begin{aligned}
        E_k^{\lambda}-E_k^{\alpha-1} &= -(\alpha - 1-\lambda)\Big[(\alpha-1)q_k +  t_{k-1} \Delta q_{k} - (\alpha-1)I_k \\
		& \quad  -\rev{(1-\eta)} t_{k-1}\bra  Q^k, \Delta z_{k-1} \ket - \frac{t_{k-1}}{2} \left(\|\Delta z_{k-1}\|^2  +\rev{(1-\eta)}\|Q^k\|^2 \right) \Big]\,,
    \end{aligned}
	\end{equation}
	where we have denoted
	\begin{equation*}
        \begin{aligned}
            &q_k:= \frac{1}{2}\|z^{k-1}-z^*\|^2 + I_k\,, \quad \text{and} \quad I_k:= \sum_{j = 1}^{k-1} \left(\bra  Q^j, z^j - z^* \ket +\frac{1}{2}\|Q^j\|^2 \right)\,.
        \end{aligned}
	\end{equation*}	
	Recall that $I_k$ is nonnegative thanks to \eqref{eq:non_monotonicity_of_Q}, hence so is $q_k$. As argued in the proof of \ref{item:nonexp_halpern}, $\Delta E^{\alpha-1}_k \leq 0$ for all $k$ large enough and thus $\frac{c}{2}\|z^k-z^*\|^2 \leq E^{\alpha-1}_k < +\infty$, i.e. $(z^k)_{k \in \N}$ is a bounded sequence. It then follows from the last two summability properties of \ref{item:nonexp_discr_summability} that the terms terms in the second line of \eqref{eq:DeltaEkqk} converge to $0$. \ref{item:nonexp_discr_summability} also gives that $I_k$ admits a limit. Moreover, $E_k^{\lambda}$ and $E_k^{\alpha}$ are nonnegative and decreasing for $k$ large enough as argued above. Therefore, we deduce that
\begin{equation}\label{eq:nonexp_discr_limit_1}
\lim_{k\to +\infty} q_k + \frac{t_{k-1}}{ \alpha-1 }\Delta q_k = \lim_{k\to +\infty} (E_k^{\lambda}-E_k^{\alpha-1}) \in \R\,.
\end{equation}
We are then in position to apply \rev{\cite[Lemma A.4]{bcch25} (see also \cite[Lemma~A.2]{acfr23} for a weaker variant that requires $\sigma \in \N$) to \eqref{eq:nonexp_discr_limit_1} to infer that $(q_k)_{k \in \N}$ admits a limit}. Using again that $I_k$ has a limit by \ref{item:nonexp_discr_summability}, we get that $(\|z^{k}-z^{*}\|)_{k \in \N}$ converges. Since $\|Q^k\| = \|x^{k-1} - T(x^{k-1})\| = \|x^{k-1} - z^{k}\|\to 0$ by \ref{item:nonexp_halpern}, we have $\lim_{k \to +\infty}\|x^{k-1}-z^{*}\|^2 = \lim_{k \to +\infty}\|z^{k}-z^{*}\|^2$ and thus the limit of $\|x^{k}-z^{*}\|$ exists. The weak convergence claim of $(x^k)_{k \in \N}$ then follows from the standard Opial's lemma.\smallskip
	
\ref{item:nonexp_discr_rates} Let us first note that since the limit of $(q^k)_{k \in \N}$ exists (see the proof of \ref{item:nonexp_discr_iters}), \eqref{eq:nonexp_discr_limit_1} yields
	\begin{equation}\label{eq:nonexp_discr_limit_2}
    \begin{aligned}
        \R \ni \lim_{k\to +\infty} t_{k-1}  \Delta q_k  
			=\lim_{k\to +\infty}   t_{k-1} \bigg[ & \bra  Q^k, z^k- z^* \ket + \frac{1}{2} \|Q^k\|^2 \\
			&+ \bra \Delta z_{k-1}, z^k - z^* \ket - \frac{1}{2}\|\Delta z_{k-1}\|^2 \bigg] \,.
    \end{aligned}
	\end{equation}
	Now, from \eqref{eq:nonexp_discrete_lyapunov}, we have
	\begin{equation*}
		\begin{aligned}
			E_k^{\alpha-1} &= \frac{1}{2}\left\| (\alpha-1) (z^{k-1}-z^*) + t_{k-1}\left( \Delta z_{k-1}  +\rev{(1-\eta)}   Q^k\right)\right\|^2\\
			& \quad + \rev{\eta}(\alpha-1) t_{k-1} \left( \bra  Q^k, z^k- z^* \ket  + \frac{1}{2}\|Q^k\|^2\right) + \frac{(1-\eta)\eta  t_{k-1}^2}{2}\|Q^k\|^2\\
			& = \frac{(\alpha-1)^2}{2}\| z^{k-1}-z^*\|^2 +  \frac{t_{k-1}^2}{2}\left\| \Delta z_{k-1} + \rev{(1-\eta)} Q^k\right\|^2 + \frac{1}{2}(1-\eta)\eta t_{k-1}^2\|Q^k\|^2 \\
			& \quad -\frac{1}{2}(\alpha-1)\rev{(1-\eta)} t_{k-1}\|Q^k\|^2 - (\alpha-1)\rev{(1-\eta)} t_{k-1} \bra  Q^k, \Delta z_{k-1} \ket\\
			& \quad + (\alpha-1)t_{k-1} \bigg[\bra  Q^k, z^k- z^* \ket  + \frac{1}{2}\|Q^k\|^2 + \bra  \Delta z_{k-1}, z^{k}-z^* \ket -  \frac{1}{2}\|\Delta z_{k-1}\|^2  \bigg] \,.
		\end{aligned}
	\end{equation*}
	Combining \eqref{eq:nonexp_discr_limit_2} (hence the last line has a limit), \ref{item:nonexp_discr_summability} (hence the second line converges to $0$), and that $\lim_{k \to \infty}\| z^{k-1}-z^*\|^2$ and $\lim_{k \to \infty} E_k^{\alpha-1}$ exist as proved above, we deduce that
	\begin{equation*}
		\lim_{\rev{k}\to +\infty} \  t_{k-1}^2\left(\frac{1}{2}\left\| \Delta z_{k-1} +\rev{(1-\eta)} Q^k\right\|^2 + \frac{1}{2}(1-\eta)\eta  \|Q^k\|^2\right) \in \R\,.
	\end{equation*}
	Observe now that the summability claims of \ref{item:nonexp_discr_summability} already give us that
\begin{equation}
\liminf_{k\to +\infty}  k^2 \|\Delta z_k\|^2 = 0\,, \quad \text{and} \quad \liminf_{k \to +\infty} k^2 \|Q^k\|^2 = 0\,. 
\end{equation}
Thus these are actually limits. To transfer the rates from $z^k$ to $x^k$, we use the relation $x^{k}=Q^{k+1}+z^{k+1}$, \rev{yielding the last two rates in \ref{item:nonexp_discr_rates}. Using Cauchy--Schwarz and boundedness of $(z^k)_{k \in \N}$, we also deduce the first. This concludes the proof.}
\end{proof}

\rev{
\begin{rmk}[The discrete gap function rate]\label{rmk:gap_function_rate}
While the convergence rates for the discrete velocity $\|x^{k} - x^{k-1}\|$ and the fixed-point residual $\|x^k - T(x^k)\|$ are relatively natural in this context, the rate
\begin{equation}
	\left\langle x^{k} - T(x^{k}), T(x^{k}) - z^* \right\rangle + \frac{1}{2} \| x^{k} - T(x^{k}) \|^2 = o(k^{-1}) \quad \text{as} \ k\to +\infty\,,
\end{equation}
may initially appear less intuitive. However, as we will see in Section~\ref{sec:splitting}, it plays a key role in applications to primal-dual-type algorithms, where it serves as a bound for the primal-dual gap, see Example \ref{example:chambolle_pock}. See also Remark \ref{rmk:gap_function_rate_cont} for its continuous-time counterpart.
\end{rmk}%
\begin{rmk}[Non-asymptotic rates]\label{rmk:explicit_rates}
In this work, we primarily focus on asymptotic convergence rates, i.e., the behavior as $k\to+\infty$. Nevertheless, for each admissible parameter choice and each quantity of interest, we can also derive corresponding non-asymptotic rates with explicit constants. Let us also emphasize here the dependency of $E_k^{\lambda}$ on $\eta$ by writing $E_k^{\lambda, \eta}$. Embarking from \eqref{eq:nonexp_descent_E_alpha>2}, we have $E_{k+1}^{\alpha-1, \eta} \leq E_k^{\alpha-1, \eta}$ for all $k \geq 1$ such that $t_k \geq \alpha-1$. If we now only consider $\sigma \geq \alpha-1$, we obtain $E_{k+1}^{\alpha-1, \eta}\leq  E_k^{\alpha-1, \eta}$ for all $k \geq 1$, leading to the following explicit bounds:
\begin{equation}\label{eq:explicit_rates}
	\begin{aligned}
		&\bra x^{k-1}-T(x^{k-1}), T(x^{k-1})-z^*\ket + \frac{1}{2}\| x^{k-1}-T(x^{k-1})\|^2 \leq \frac{E_1^{\alpha - 1, \eta}}{\eta(\alpha-1)t_{k-1}}\,,\\
		& \| x^{k-1}-T(x^{k-1})\|^2\leq \frac{2 E_1^{\alpha-1,\eta}}{\eta(1-\eta) t_{k-1}^2}\,,
	\end{aligned}
\end{equation}
where $E_1^{\alpha-1,\eta}$ is a constant that only depends on the initial conditions, on $z^*$, and the algorithms parameters, including $\eta$---see \eqref{eq:nonexp_discrete_lyapunov} for the precise definition. A particularly illustrative situation occurs when $\alpha=2$, $\eta=\frac{1}{2}$, and $\sigma=1$, which yields $t_0=0$ and thus
\begin{equation}\label{eq:def_gap_function}
	\|x^{k-1}- T(x^{k-1})\|^2 \leq \frac{4\|z^0-z^*\|^2}{t_{k-1}^2} = \frac{4\|T(x^{-1})- x^*\|^2}{(k-1)^2} \quad \text{for all} \ k \geq 2\,.
\end{equation}
A more in-depth investigation of the two bounds in \eqref{eq:explicit_rates}, including the influence of the parameter $\eta$, is left for future work.
\end{rmk}
}

\subsection{Connection to Existing Algorithms}
\rev{Algorithm~\ref{alg:fast_km} offers a flexible and unifying framework that is connected to a broad class of existing methods. In this section, we highlight a few notable special cases and related algorithms to place them more clearly within the literature.}

\subsubsection{Connection to \cite{botkhoa2023}}\label{rmk:connection_to_tikhonov}

As we mentioned in the introduction, Algorithm~\ref{alg:fast_km} is a generalization of the scheme (2.8) in \cite{botkhoa2023}. The latter can be recovered \rev{by setting $\sigma = \alpha + 1$ and $\theta = \tfrac{\alpha}{2}$, which gives, renaming $x^{k}$ as $x^{k+1}$, for all $k \geq 1$}:
	\begin{equation}\label{eq:fast_km_khoa}
		x^{k+1}=\left(1-\frac{\alpha}{2(k+\alpha)}\right)x^{k} + \frac{\alpha}{2(k+\alpha)}T(x^k) + \left(1-\frac{\alpha}{k+\alpha}\right)\big(T(x^k)-T(x^{k-1})\big)\,,
	\end{equation}
	with initial points $x^0, x^1\in H$. This is also a special case of \cite[Algorithm~2.1]{botkhoa2023}, which can be recovered by introducing a step size $s\in (0, 1]$ and considering $T_s:= (1-s)I + sT$, which is now $s$-averaged and has the same fixed points as $T$. Indeed, substituting $T$ by $T_s$ in \eqref{eq:fast_km_khoa} we get for all \rev{$k \geq 1$}:
	\begin{equation}\label{eq:fast_km_khoa_s}
		\begin{aligned}
			x^{k+1}=&\left(1-\frac{s\alpha}{2(k+\alpha)}\right)x^{k} +\frac{(1-s)k}{k+\alpha}(x^k-x^{k-1}) + \frac{s\alpha}{2(k+\alpha)}T(x^k) + \frac{sk}{k+\alpha}\big(T(x^k)-T(x^{k-1})\big)\,.
		\end{aligned}
	\end{equation}
	Conversely, if $T$ is \rev{$\zeta$-averaged}, we shall consider \eqref{eq:fast_km_khoa_s} with \rev{$s \in (0, 1/\zeta]$}, which is equivalent to applying \eqref{eq:fast_km_khoa} with $T$ replaced by $sT + (1  - s)I$.

\subsubsection{Connection to Halpern-type Methods}\label{sec:halpern}

\rev{
Similarly to the continuous-time dynamics discussed in Section \ref{sec:edge_in_theta}, Algorithm \ref{alg:fast_km} reduces to an \emph{anchored} method when we set $\theta=1$. Firstly, when $\alpha =2$, $\sigma = \alpha$, and $x^0=T(x^{-1})$, it is clear that Algorithm~\ref{alg:fast_km} coincides with the \emph{momentum form} of the Optimal Halpern Method (OHM) \cite{lieder21, ss17}, and in turn with Kim's accelerated proximal point algorithm \cite{kim2021}, since the latter two are equivalent, see \cite[Appendix B]{ykr24}. In the edge case of setting $\theta$ equal to $1$, it turns out that Algorithm \ref{alg:fast_km} is indeed equivalent to anchored-type methods for any $\alpha\geq 2$ and $\sigma > 0$, as we will show in the following result.
	\begin{prop}\label{prop:fast_km_and_halpern}
		Let $(x^k)_{k \in \N}$ be the sequence generated by Algorithm \ref{alg:fast_km} with $\alpha \geq 2$, $\theta=1$ and $\sigma > 0$. Then, $(x^k)_{k \in \N}$ satisfies:
		\begin{equation}
			x^{k+1}= \varepsilon_k v + (1-\varepsilon_k)T(x^k)\,, \quad \text{for} \ k \geq 0\,,
		\end{equation}
		where $\varepsilon_k:=\frac{\alpha-1}{k+\sigma}$ and $v:=\frac{t_0}{\alpha-1}(x^0-T(x^{-1})) + T(x^{-1})$.
	\end{prop}
    \begin{proof}
    Multiplying the update rule of Algorithm~\ref{alg:fast_km} by $t_{k+1}:= k+\sigma$, we get
	\begin{equation*}
		\begin{aligned}
			&(k+\sigma)(x^{k+1}-x^k) = T(x^k)-x^k + (k+\sigma -\alpha)(T(x^k)-T(x^{k-1}))\\
			\iff \ & t_{k+1} (x^{k+1} - x^k) + x^k\\
			& \quad  - \Big( T(x^k) +  (k+\sigma -1)(T(x^{k+1})-T(x^k))\Big) + (\alpha-1) (T(x^k) - T(x^{k-1})) = 0\\
			\iff \ & \big( t_{k+1}x^{k+1}-t_{k}x^k\big) - \big( t_{k+1}T(x^k)-t_k T(x^{k-1})\big) + (\alpha-1) \big(T(x^k) - T(x^{k-1})\big) = 0\,.
		\end{aligned}
	\end{equation*}
Summing the above equation from $0$ to $k$, dividing the result by $t_{k+1}$ and rearranging terms we get the claim.
\end{proof}
}
\noindent Note that in the edge case $\theta=1$ we cannot expect to obtain little-$o$ rates for fixed-point residuals, and the convergence of the iterates is in fact in the strong topology, but is established with completely different techniques than those we have applied in Section \ref{sec:discrete_lyapunov_analysis}: \rev{The former rely on the Gr\"onwall's lemma (see, e.g., \cite{bc24}), the latter on Opial's lemma}.

\rev{
\subsubsection{Connection to Inertial KM and Nesterov's Method} The Fast-KM method is not the first attempt to introduce a Nesterov-type acceleration for iterative methods governed by nonexpansive operators. The earliest such attempt dates back to 2001, with the work of Attouch and Alvarez in \cite{aa01}. Starting from an implicit discretization of the celebrated Heavy Ball with Friction (HBF) dynamical system governed by a maximal monotone operator, they arrived at the following numerical algorithm, for a firmly nonexpansive operator $T$:
\begin{equation}\label{eq:inertial_km}
	x^{k+1} = T(x^k) + \alpha_k \big(T(x^k) - T(x^{k-1})\big)\,, \quad \text{for all} \ k \geq 0\,,
\end{equation}
with $\alpha_k \in (0, 1)$. Note that if $\alpha_k := 1 - \frac{\alpha}{k + \sigma}$ with $\sigma > 0$, $\alpha \geq 3$, and $T := I - \tau \nabla f$ for a $L$-smooth function $f \colon H \to \R$ and $\tau \in (0, \frac{1}{L}]$, then \eqref{eq:inertial_km} coincides with the \emph{Ravine} formulation of the celebrated Nesterov Accelerated Gradient method \cite{Nesterov1983AMF, af22}, which is known to yield the optimal convergence rate in convex optimization. While \eqref{eq:inertial_km} can still provide weakly converging sequences for general (firmly) nonexpansive operators---specifically if $\alpha_k \leq \bar \alpha < \frac{1}{3}$, cf.~\cite[Proposition 2.1]{aa01}---whether it achieves fast convergence rates in this general setting remains unclear. A first negative answer was provided only recently in \cite{mfp24}, in a broader context involving iteration-dependent operators and more general parameter choices. Specifically, in the context of \eqref{eq:inertial_km}, \cite[Theorem 4(ii)]{mfp24} shows that if $\alpha_k$ satisfies a specific growth condition---implying, in particular, that $\sup_{k \in \N} \alpha_k < 1$---then one only obtains
\begin{equation}\label{eq:rate_inertial_km}
	\min_{0 \leq k \leq n} \ \|x^k - T(x^k)\| \leq M\frac{\dist (x^0, \Fix T)}{\sqrt{n}}\,, \quad \text{for all} \ n \geq 1\,,
\end{equation}
for some constant $M > 0$. This result is \emph{ergodic} and of suboptimal order $\cO(n^{-\frac{1}{2}})$.

Comparing \eqref{eq:inertial_km} with Algorithm~\ref{alg:fast_km}, we see that \eqref{eq:inertial_km} could, in principle, be recovered from Algorithm~\ref{alg:fast_km} by replacing the vanishing regularization term $\frac{\theta}{k+\sigma}$ with $1$, and by capping $\alpha_k$ at a value strictly smaller than $1$. In the continuous-time model \eqref{eq:second_order_dynamic_general}, this would correspond to taking $b \equiv 1$ and a \emph{non}-vanishing damping function, respectively. However, doing so would worsen the rate to \eqref{eq:rate_inertial_km}. Thus, we learn that letting $\alpha_k \to 1$ and using a vanishing regularization parameter $\frac{\theta}{k+\sigma}$ are both essential for achieving fast convergence rates for general nonexpansive operators.

\subsubsection{Connection to Tran-Dinh's Method} In \cite{TranDinh24}, Tran-Dinh proposed a method that is able to achieve a fast little-$o$ rate for solving root-finding problems of the form:
\begin{equation}\label{eq:root_finding_trandinh}
	\text{Find} \ x \in H \ \text{such that:} \quad 0 = G(x)\,,
\end{equation}  
where $G\colon H\to H$ is $\frac{1}{L}$-cocoercive, for $L>0$. The method in \cite[Theorem 3]{TranDinh24} is formulated as
\begin{equation}\label{eq:trandinh_method}
	x^{k+1} = x^k + \bar\theta_k (x^k - x^{k-1}) - \bar \eta_k \big( G(x^k) - \bar \gamma_k G(x^{k-1})\big) \,, \quad \text{for} \ k \geq 0\,,
\end{equation}
with parameter sequences defined, for $\bar \gamma \in (0, 1)$, $\omega > 2$, and all $k \geq 0$, as 
\begin{equation}\label{eq:parameters_trandinh}
	\bar \theta_k:= \frac{k + 1}{k + 2\omega +2}\,, \quad \bar\eta_k := \frac{\bar \gamma}{L}\frac{k +\omega + 1}{k+2\omega + 2}\,, \quad \text{and} \quad \bar \gamma_k := \frac{\bar \gamma}{L}\frac{\bar \theta_k}{\bar \eta_k}\,.
\end{equation}
It turns out that \eqref{eq:trandinh_method} is a special case of Algorithm \ref{alg:fast_km}. Indeed, considering the operator $T:= I - \frac{\bar \gamma}{L}G$, solving \eqref{eq:root_finding_trandinh} is equivalent to finding a fixed point of $T$, $T$ being $\frac{\bar \gamma}{2}$-averaged. If we now apply Algorithm \ref{alg:fast_km} to find a fixed point of $T$ setting $\sigma := 2\omega + 2$, $\alpha := 2\omega + 1$ and $\theta:=\omega$, we obtain exactly \eqref{eq:trandinh_method} with the parameter choices \eqref{eq:parameters_trandinh}. Interestingly, Theorem \ref{thm:nonexp_second_order_discrete} allows us to further extend the parameter choices proposed in \cite{TranDinh24} to $\omega \geq \frac{1}{2}$, instead of $\omega \geq 2$, and $\bar \gamma \in (0, 2]$ instead of $\bar \gamma \in (0, 1)$, as well as to establish a last-iterate rate on the gap function.
}

\section{Application to Splitting Methods with Preconditioners}\label{sec:splitting}
The goal here is apply our results to operator splitting methods for solving monotone inclusion problems, i.e. finding a zero of one or sum of maximally monotone operators. We follow the algorithmic framework introduced in \cite{bredies2021degenerate}. For a maximal monotone operator $A\colon H\to 2^{H}$, we introduce a \textit{preconditioner} $M$, i.e., a self-adjoint positive semidefinite bounded linear operator $M\colon H\to H$, \rev{which induces a Hilbertian (semi-)inner product $\bra  u, v \ket _M:= \bra  u, Mv \ket$ and associated (semi) norm $\|u\|_M :=  \sqrt{\bra  u, u \ket _M}$ for all $u,v \in H$.} Given $A$ and $M$, we consider the resolvent:
\begin{equation}\label{eq:preconditioned_resolvent}
	J_{M^{-1}A} := (I+M^{-1}A)^{-1} = (M+A)^{-1}M\,.
\end{equation}
It is worth noting that without further hypothesis, $J_{M^{-1}A}$ is not necessarily single-valued nor defined everywhere on $H$. On the other hand, this representation is helpful in practice, and evaluating $J_{M^{-1}A}$ can be facilitated by the right choice of $M$. As first remarked by He and Yuan in \cite{HeYuan12}, this is in fact the case for primal-dual-type methods:
\rev{\begin{example}[PDHG]\label{example:chambolle_pock}
Let $f\colon H\to \R\cup \{+\infty\}$ and $g\colon K\to \R\cup\{+\infty\}$ be proper, convex and lower semicontinuous functions on the Hilbert spaces $H$ and $K$, and $L\colon H\to K$ a linear bounded map. The Primal-Dual-Hybrid-Gradient (PDHG) method by Chambolle and Pock \cite{cp11} is a classical scheme to find a solution to the following saddle point problem
\begin{equation}\label{eq:lagrangian}
	\min_{x \in H} \max_{y \in K} \ \cL(x, y):= f(x) + \langle Lx, y\rangle - g^*(y)\,,
\end{equation}
where $g^*$ is the Legendre--Fenchel conjugate of $g$. This has enormous applications in applied mathematics, see, e.g., \cite{cp16_book}. The PDHG iteration takes the form
\begin{equation}\label{eq:chambolle_pock}
	\left\{
	\begin{aligned}
		&x^{k+1}=\prox_{\tau_1 f_1}(x^k - \tau_1 L^Ty^k)\,,\\
		&y^{k+1}=\prox_{\tau_2 f_2^{*}}(y^k + \tau_2 L(2x^{k+1}-x^k))\,.
	\end{aligned}\right.
\end{equation}
Under a mild qualification condition, the generated sequence $u^k:=(x^{k}, y^{k})$ converges weakly to a saddle point of $\cL$ as soon as $\tau_1\tau_2\|L\|^2 < 1$, see \cite{HeYuan12, bredies2021degenerate}. Following He and Yuan \cite{HeYuan12}, the method can be understood a fixed-point iterations with respect to $J_{M^{-1}A}$ with
\begin{equation}\label{eq:operators_chambolle_pock_2}
	A:= \begin{bmatrix}
		\partial f & L^T\\ 
		-L & \partial g^*
	\end{bmatrix}\,, \quad \text{and} \quad M:= \begin{bmatrix}
		\frac{1}{\tau_1}I &  -L^T\\
		-L & \frac{1}{\tau_2}I
	\end{bmatrix}\,.
\end{equation}
The operator $J_{M^{-1}A}$ can in fact be easily computed due to the block-lower triangular structure of $A+M$ (recall the formula \eqref{eq:preconditioned_resolvent} or see \cite[Section 3]{bredies2021degenerate} for the detailed computations), and the resulting fixed-point iteration coincides with \eqref{eq:chambolle_pock}.

The step-size condition $\tau_1\tau_2\|L\|^2< 1$ is equivalent to the positive definiteness of $M$ \cite{bredies2021degenerate}, which makes $M^{-1}A$ a maximal monotone operator on the Hilbert space $H$ endowed with the norm induced by $M$. Being an instance of the (preconditioned) proximal point algorithm \cite{bredies2021degenerate}, the rate on the fixed-point residual is
\begin{equation}
	\|u^{k} - J_{M^{-1}A}(u^k)\|_M = o(k^{-\frac{1}{2}}) \quad \text{as} \ k \to +\infty\,.
\end{equation}
A second classical convergence measure for \eqref{eq:chambolle_pock} is the \emph{primal-dual gap}\footnote{\rev{More precisely, the primal-dual gap should be defined as $\sup_{u^* \in \zer A} \cG_{u^*}(x, y)$. However, as $\zer A$ might be unbounded, this quantity is not easily controllable, leading to notions such as \emph{restricted primal-dual gap} \cite{cp11} or to further assumptions such as boundedness of $\zer A$. In this example, we keep the presentation simple and avoid taking the $\sup$.}}, which we define as
\begin{equation}\label{eq:primal-dual_gap}
	\mathcal{G}_{u^*}(x, y):=\cL(x, y^*)- \cL(x^*, y)\,, \quad \text{for all} \ (x, y) \in H\times K\,.
\end{equation}
The quantity that in Remark~\ref{rmk:gap_function_rate} we referred to as the gap function \eqref{eq:def_gap_function} is indeed an upper bound for \eqref{eq:primal-dual_gap}. Specifically, using convexity and the definition of $A$ in \eqref{eq:operators_chambolle_pock_2}, we get
\begin{equation}\label{eq:primal_dual_bound}
	\begin{aligned}
		\mathcal{G}_{u^*}(x^{k+1}, y^{k+1}) \leq 
		&\bra u^k - J_{M^{-1}A}(u^k), J_{M^{-1}A}(u^k)-u^*\ket_M\\
	 \leq & \bra u^k - J_{M^{-1}A}(u^k), J_{M^{-1}A}(u^k)-u^*\ket_M + \frac{1}{2}\|u^k - J_{M^{-1}A}(u^k)\|^2_M\,.
	\end{aligned}
\end{equation}
Note that $u^k-J_{M^{-1}A}(u^k) = u^k - u^{k+1} \in M^{-1}A(u^{k+1})$, so the first inner product \eqref{eq:primal_dual_bound} is non-negative in this case. Secondly, while the right-hand side of \eqref{eq:primal_dual_bound} vanishes at the rate $o(k^{-\frac{1}{2}})$, it is also summable, leading, using convexity of $\cG_{u^*}$ to the classical $\cO(k^{-1})$ ergodic rate for \eqref{eq:primal-dual_gap}, see \cite{cp16_book}.
\end{example}
}
By allowing the preconditioner $M$ to be \textit{degenerate}, i.e., with $\ker M\neq \emptyset$, Bredies and Sun in \cite{Bredies2017APP, BrediesDRS}, and later \cite{bredies2021degenerate} recognized that the class of methods representable as iterations with respect to degenerate preconditioners can be greatly broadened to include methods such as the Douglas--Rachford Splitting (DRS) \cite{drs_mercier_lions} and its extensions \cite{bredies2022graph}.

\subsection{Preliminaries on Degenerate Proximal Point}\label{sec:preliminaries_degenerate}
The notion of monotonicity and cocoercivity can be extended to a \textit{preconditioned} setting, even with a degenerate preconditioner $M\colon H\to H$. For a maximal monotone operator $A$ and a preconditioner $M$ we shall consider the operator $M^{-1}A$ as a composition of the multi-valued operators $M^{-1}$ (pre-image) and $A$, defined as: \rev{$v \in M^{-1}A(u)$ if and only if $Mv \in A(u)$}. It is easy to check that $M^{-1}A$ is monotone with respect to the semi-inner product induced by $M$. We then say that $M^{-1}A$ is $M$-\textit{monotone}. Likewise, its resolvent $J_{M^{-1}A}$ is $M$-\textit{cocoercive} in the sense that
\begin{equation}\label{eq:M_cocoercivity}
	\bra  v-v', u-u' \ket _M \geq \|v-v'\|_M^2\,, \quad \text{for all} \ (u, v), \ (u', v')\in \gra J_{M^{-1}A}\,.
\end{equation}
Due to the degeneracy of the preconditioner, the operator $J_{M^{-1}A}$ is not necessarily single-valued and could even have empty domain. To avoid such pathological cases, we suppose $\dom J_{M^{-1}A} = H$. A further hypothesis we shall require is that the operator
\begin{equation}\label{eq:admissble_preconditioner}
	(M+A)^{-1} \ \text{is Lipschitz over} \ \Image M\,,
\end{equation}
which makes $J_{M^{-1}A}$ a well-defined Lipschitz operator on $H$, since indeed $J_{M^{-1}A}=(M+A)^{-1}M$. \rev{Additionally, it can be shown to satisfy the fundamental property:
\begin{equation}\label{eq:lipschitz_bound_to_M}
	\|J_{M^{-1}A}(u) - J_{M^{-1}A}(u')\|^2 \leq L \|u-u'\|_M^2\,, \quad \text{for all} \ u, u'\in H\,.
\end{equation}}
\rev{We refer to fixed-point iterations with respect to $J_{M^{-1}A}$ as \emph{preconditioned proximal point algorithm}, which for $\theta \in (0, 2)$ and $u^0 \in H$, reads as
\begin{equation}\label{eq:degenerate_ppp}
		u^{k+1}= u^k + \theta \big(J_{M^{-1}A}(u^k) - u^k\big)\,, \quad \text{for all} \ k \in \N\,.
	\end{equation}
	The convergence guarantees of \eqref{eq:degenerate_ppp} are given in \cite[Theorem 2.9]{bredies2021degenerate}. Specifically, if \eqref{eq:admissble_preconditioner} holds, then $(u^k)_{k \in \N}$ and $(J_{M^{-1}A}(u^k))_{k \in \N}$ converge weakly to the same zero of $A$ with the rate:
	\begin{equation}
		\|u^k - J_{M^{-1}A}(u^k)\|_M = o(k^{-\frac{1}{2}}) \quad \text{as}\ k \to +\infty\,.
\end{equation}}
An appealing property of positive semidefinite preconditioners is that evaluating $J_{M^{-1}A}$ automatically produces a quotient space operation. Suppose that $M$ has closed range and consider a factorization of $M$ as $M=CC^*$, with $C\colon D\to H$ one-to-one from a Hilbert space $D$ ($D\cong \Image M$) to $H$. We say that $CC^*$ is an \textit{onto decomposition} of $M$. In this case, one can prove that $C^*$ is onto \cite[Proposition 2.3]{bredies2021degenerate} and the operator $C^*\rhd A := (C^*A^{-1}C)^{-1}$ is maximal monotone \cite[Theorem 2.13]{bredies2021degenerate}. The resolvents of $M^{-1}A$ and of $C^*\rhd A$ are tightly connected: If $M$ has closed range and $M=CC^*$ is onto, then 
\begin{equation}\label{eq:resolvent_push_forward}
	C^* J_{M^{-1}A}( u) = J_{C^*\rhd A} \left(C^* u \right)\,, \quad \text{for all} \ u \in H\,.
\end{equation}
This is particularly useful as it reveals that the evaluation of $J_{M^{-1}A}$ is somewhat equivalent to evaluating $J_{C^*\rhd A}$, which is a resolvent of a maximal monotone operator and is defined on $D$, which is isomorphic to $\Image M$; see \cite[Theorem~3.3.3]{chenchene2023splitting} for a proof. This implies that \rev{the sequence $w^k:=C^*u^k \in D$ is a classical proximal point method with respect to the maximal monotone operator $C^*\rhd A$:
\begin{equation}\label{eq:reduced_degenerate_ppp}
	w^{k+1} = w^k + \theta \big(J_{C^*\rhd A}(w^k)-w^k\big) \quad \text{as}\ k \to +\infty\,.
\end{equation}}%
However, although $J_{M^{-1}A}(u^k)=(M+A)^{-1}(Cw^k)$ from \eqref{eq:preconditioned_resolvent} and $M=CC^*$, the weak convergence of $(u^k)_{k \in \N}$ and $(J_{M^{-1}A}(u^k))_{k \in \N}$, in infinite-dimensional spaces, does not follow directly from the that of $(w^k)_{k \in \N}$. For clarity, we present a prototypical example:

\begin{example}[Douglas--Rachford \cite{BrediesDRS, bredies2021degenerate}]\label{example:drs} Let $A_1, A_2\colon H\to 2^{H}$ be two maximal monotone operators on $H$ and assume that $\zer (A_1+A_2)\neq \emptyset$. Following \cite{bredies2021degenerate}, the DRS method for finding a zero of $A_1+A_2$ can be understood as iterating $J_{M^{-1}A}$ with the operators:
\begin{equation}\label{eq:operators_drs_2}
	A:= \begin{bmatrix}
		\phantom{-}A_1 & \phantom{-}I  & - I\\ 
		-I & \phantom{-}A_2 & \phantom{-}I \\
		\phantom{-}I & -I & \phantom{-}0
	\end{bmatrix}\,, \quad \text{and} \quad M:= \begin{bmatrix}
	\phantom{-}I &  -I & \phantom{-}I\\
	-I & \phantom{-}I & -I \\
	\phantom{-}I & -I & \phantom{-}I
\end{bmatrix}= \begin{bmatrix}
I \\ -I\\ I 
\end{bmatrix}\underbrace{\begin{bmatrix}
I & -I & I
\end{bmatrix}}_{:=C^*}\,.
\end{equation}
Indeed, since $M+A$ is block lower-triangular, evaluating $J_{M^{-1}A}=(M+A)^{-1}M$ only requires accessing the resolvents of $A_1$ and $A_2$, we refer to \cite{bredies2021degenerate} for the detailed computation. If $u^{k}:=(x_1^k, x_2^k,  v^k)\in H^3$ is such that $u^{k+1}:=J_{M^{-1}A}(u^k)$ the DRS method can be obtained via the substitution $w^k:=x_1^k - x_2^k + v^k$, which yields for all $k \in \N$:
\begin{equation}\label{eq:douglas_rachford}
	\left\{
	\begin{aligned}
		& x_1^{k+1} = J_{A_1}(w^k)\,,\\
		& x_2^{k+1}= J_{A_2}(2x_1^{k+1}-w^k)\,,\\
		& w^{k+1}= w^k + x_2^{k+1}-x_1^{k+1} \,.
	\end{aligned}\right. \quad \rev{ \text{and} \quad J_{M^{-1}A}(u^k):=\begin{pmatrix}
	x_1^{k+1}\\
	x_2^{k+1}\\
	w^k + 2(x^{k+1}_2-x^{k+1}_1)
\end{pmatrix}\,.}
\end{equation}
This corresponds to taking $w^k = C^* u^k$, which defines a fixed-point iteration with respect \rev{$J_{C^*\rhd A}$, the operator $C^*\rhd A$ now being maximal monotone. It follows from \cite[Theorem 2.9]{bredies2021degenerate} that $(w^k)_{k \in \N}$ weakly converges to some $w^*\in \zer C^*\rhd A$, with the rate
\begin{equation}\label{eq:rate_douglas_rachford}
	\|x_2^{k+1}-x_1^{k+1}\| = \|w^{k} - J_{C^*\rhd A}(w^k)\| = \|u^{k} - J_{M^{-1}A}(u^k)\|_M = o(k^{-\frac{1}{2}})\,, \quad \text{as} \ k \to+\infty\,. 
\end{equation}}%
Additionally, $(J_{M^{-1}A}(u^k))_{k \in \N}$ weakly converges to a zero of $A$, yielding that the \emph{solution estimates} $(x_1^{k+1})_{k \in \N}$ and $(x_2^{k+1})_{k \in \N}$ (a.k.a.~shadow sequences) weakly converge to the same zero of $A_1+A_2$.

Note that the weak convergence of the shadow sequences remained an open problem for almost $30$ years until Svaiter proposed a convergence proof in \cite{SVAITER2011}. The analysis of iterations associated to \eqref{eq:preconditioned_resolvent} with a degenerate preconditioner allows to recover this result from general principles applicable to a broad range of algorithms; see \cite{bredies2021degenerate}.
\end{example}

\rev{
\begin{example}[Graph Douglas--Rachford \cite{bredies2022graph, chenchene2023splitting, acgn25}]\label{example:graph_drs}
The graph-Douglas--Rachford Splitting method (Graph-DRS) is an extension of the Douglas--Rachford method to solve instances of:
\begin{equation}\label{eq:Nop}
	\text{Find} \ x \in H: \quad 0 \in A_1(x) + \dots + A_N(x) \,,
\end{equation}
with $A_1,\dots, A_N\colon H\to 2^H$ maximal monotone operators. The method only requires picking two real matrices $Z, \hat{Z}\in \R^{N\times (N-1)}$ and a step size $\tau>0$, with $\ker Z^T = \Span \{\mathbf{1}\} \subset \ker \hat{Z}^T$ where $\mathbf{1}:=(1,\dots, 1)^T \in \R^N$. Then, denoting by $L:=ZZ^T$, $\hat{L}:=\hat{Z}\hat{Z}^T$ and $d:=\diag ({L}+\hat{{L}})$, the update rule reads as:
\begin{equation}\label{eq:general_drs}
	\left\{
	\begin{aligned}
		& x_i^{k+1} =J_{\frac{\tau}{d_i}A_i}\bigg(-\frac{2}{d_i}\sum_{h=1}^{i-1}\left({L}_{hi}+\hat{{L}}_{hi}\right)x_h^{k+1} + \frac{1}{d_i}\sum_{j=1}^{N-1}Z_{ij}w_j^{k}\bigg)\,, \quad \text{for all} \ i \in \llbracket 1, N\rrbracket\,,\\
		& w_j^{k+1} = w_j^{k} - \sum_{i=1}^{N} Z_{ij}x_i^{k+1}\,, \quad \text{for all} \ j \in \llbracket1, N-1\rrbracket\,.
	\end{aligned}\right.
\end{equation}
This method generalizes a number of splitting methods and can in fact be shown to encompass \emph{all} averaged frugal resolvent splitting methods with minimal lifting in the sense of Ryu, see \cite{acgn25} and the references therein. Similarly to Example \ref{example:drs}, it can be understood in the framework of degenerate proximal point methods with respect to $J_{\bC^*\rhd \bA}$ for a maximal monotone operator $\bA\colon \bH\to 2^{\bH}$, and a degenerate preconditioner $\bM := \bC\bC^*$, this time with $\bC^*\colon \bH \to H^{N-1}$ and $\bH:=H^{2N-1}$, cf.~\cite{bredies2022graph} for their precise (but technical) definition. Moreover, $\bx^{k+1}:=(x_1^{k+1}, \dots, x_N^{k+1})$ are the first $N$ components of $(\bM+\bA)^{-1}(\bC \bw^k)$, where $\bw^k:=(w_1^{k}, \dots, w_{N-1}^k)$.

While in \eqref{eq:rate_douglas_rachford} the fixed-point residual coincides with the difference of the two solution estimates, for \eqref{eq:general_drs} it controls a natural dispersion functional between the components of $\bx^{k+1}$, i.e., their \emph{variance}:
\begin{equation}\label{eq:def_variance}	
\Var(\bx^{k+1}):= \frac{1}{N}\sum_{i=1}^{N}\|x_i^{k+1}-\bar{x}^{k+1}\|^2 \leq \frac{1}{\lambda_1N} \|\bw^k - J_{\bC^*\rhd \bA}(\bw^k)\|^2 = o(k^{-1})\,, \quad \text{as} \ k\to+\infty\,.
\end{equation}
where $\bar{x}^{k+1} := \frac{1}{N}\sum_{i=1}^{N} x_i^{k+1}$, and $\lambda_1$ is the smallest nonzero eigenvalue of $L$, see \cite[Proposition~3.8]{bredies2022graph} or \cite[Lemma~5.2.1]{chenchene2023splitting} for a proof of the inequality in \eqref{eq:def_variance}. Additionally, thanks to \cite[Theorem 2.9]{bredies2021degenerate}, all the solution estimates $(x_i^{k+1})_{k \in \N}$ for $i \in \llbracket 1,  N\rrbracket$ converge weakly to the same solution to \eqref{eq:Nop}.
\end{example}}

\subsection{Fast Degenerate Preconditioned Proximal Point Algorithms}

\rev{Once the preconditioned resolvent $J_{M^{-1}A}$ is formed, e.g., as in Examples \ref{example:chambolle_pock}, \ref{example:drs}, and \ref{example:graph_drs}, we can employ it in conjunction to Algorithm \ref{alg:fast_km} to study accelerated variants of degenerate preconditioned proximal point iterations \eqref{eq:degenerate_ppp}. Denoting by $\theta_k := \frac{\theta}{k+\sigma}$ and $\alpha_k := 1-\frac{\alpha}{k+\sigma}$, we can consider, for $s\in (0, 2]$, the $M$-averaged operator $T:= I + s(J_{M^{-1}A}-I)$ and apply Algorithm \ref{alg:fast_km}, yielding, for all $k\geq 0$,
\begin{equation}\label{eq:fast_degenerate_ppp}
	u^{k+1} = u^k + \alpha_k(1-s)(u^k-u^{k-1}) + s\theta_k J_{M^{-1}A}(u^k) + s\alpha_k\Big(J_{M^{-1}A}(u^k)- J_{M^{-1}A}(u^{k-1})\Big)\,.
\end{equation}
Similarly to the non-accelerated setting, if $M$ has closed range, using \eqref{eq:resolvent_push_forward}, the sequence $w^k:=C^*u^k \in D$ can be easily seen to define an instance of Algorithm \ref{alg:fast_km} with respect to $\tilde T:= I+ s(J_{C^*\rhd A}-I)$, the latter being $\frac{s}{2}$-averaged in the classical sense since $C^*\rhd A$ is maximal monotone. Therefore, convergence rates and weak convergence of $(w^k)_{k \in \N}$ follow immediately from Theorem \ref{thm:nonexp_second_order_discrete}. However, using this argument, the weak convergence of $(J_{M^{-1}A}(u^k))_{k \in \N}$ does not follow from that of $(w^k)_{k \in \N}$ since the convergence of $(w^{k})_{k \in \N}$ is in the weak topology and $J_{M^{-1}A}(u^k)=(M+A)^{-1}(Cw^k)$ is typically nonlinear.} The following lemma captures the argument that allows us to close this important gap \rev{and extend the analysis to general preconditioners, with non-necessarily closed range}. The mechanism underlying this result can be traced back to \cite{bredies2021degenerate}, \rev{although it was not explicitly established there.}

\begin{lem}[Opial in the Degenerate Setting]\label{lem:opial_degenerate}
	Let $A$ be a maximal monotone operator with $\zer A \neq \emptyset$, and $M$ a self-adjoint positive semidefinite preconditioner such that $(M+A)^{-1}$ is Lipschitz over $\Image M$ and $J_{M^{-1}A}$ has full domain. Let $(u^k)_{k \in \N}$ be a sequence, and suppose that:
	\begin{enumerate}[label=(\roman*)]
		\item \label{item:opial_sequence} For all $u^* \in \Fix J_{M^{-1}A}$ the limit  $\lim_{k\to +\infty} \ \|u^k-u^*\|_M$ exists;
		\item \label{item:asymptotic_regularity} and $\lim_{k\to +\infty}  \| u^k - J_{M^{-1}A}(u^k)\|_M = 0$.
	\end{enumerate}
	Then $J_{M^{-1}A}(u^k) \rightharpoonup u^* \in \zer A$ as $k \to +\infty$. If, additionally, $M$ has {closed range} and $M=CC^*$ is onto, then for $w^k:=C^*u^k$:
	\begin{equation*}
        w^k \rightharpoonup w^*\in \zer C^*\rhd A\,, \quad \text{and} \quad (A+M)^{-1}(Cw^k)=J_{M^{-1}A}(u^k)\rightharpoonup u^*\in \zer A\,, \quad \text{as} \ k \rightarrow +\infty\,.
	\end{equation*}
 \end{lem}%
\begin{proof}
    For brevity, let us denote by $v^k:=J_{M^{-1}A}(u^k)$ for all $k \in \N$. From \ref{item:opial_sequence}, we deduce that the sequence $(Mu^k)_{k\in \N}$ is bounded. Therefore, since $(M+A)^{-1}C$ is Lipschitz and $J_{M^{-1}A}=(M+A)^{-1}M$, which is then single-valued everywhere on $H$, also $(v^k)_{k \in \N}$ is bounded. Since $M(u^k - v^k) \in A(v^k)$ for all $k \in \N$, $\lim_{k \to +\infty}\| M(u^k - v^k)\| = 0$ from point \ref{item:asymptotic_regularity}, and $A$ is maximal monotone (in particular its graph is weak-to-strong sequentially closed \cite[Proposition 20.38]{BCombettes}), every weak cluster point of $(v^k)_{k \in \N}$ belongs to $\zer A$. Now we show that, in fact, it admits only one weak cluster point and thus $(v^k)_{k\in \N}$ converges weakly to it. Let $(v^{k_j})_{j \in \N}$ and $(v^{k_i})_{i \in \N}$ be two weakly converging subsequences and $v^*, v^{**} \in H$ their respective {weak} cluster points. Note that, from points \ref{item:opial_sequence} and \ref{item:asymptotic_regularity}, for all $u^* \in \zer A$, $\lim_{k\to +\infty}\|v^k - u^*\|_M = \lim_{k\to +\infty}\|u^k - u^*\|_M$, which we denote by $\ell(u^*)$. Suppose by contradiction that $\|v^* - v^{**}\|_M \neq 0$. Thus, 
    \begin{equation}
        \begin{aligned}
            \ell(v^{**}) = \lim_{i \to +\infty} \| v^{k_i} - v^{**} \|_M^2 &= \lim_{i \to +\infty} \left(\|v^{k_i} - v^{*}\|_M^2 + 2\langle v^{k_i} - v^{*}, v^{*}-v^{**}\rangle_M +\|v^{*} - v^{**}\|_M^2 \right)\\
            & = \|v^{*} - v^{**}\|_M^2  + \lim_{i\to +\infty} \|v^{k_i} - v^{*}\|_M^2 > \ell(v^{*})\,,
        \end{aligned}
    \end{equation}
    where we have used that $v^{k_i}\rightharpoonup v^{*}$. Reversing the role of $v^*$ and $v^{**}$, we also get $\ell(v^{**})<\ell(v^{*})$, hence the desired contradiction, which yields $M v^* = Mv^{**}$. Hence, $J_{M^{-1}A}(v^*) = J_{M^{-1}A}(v^{**})$, and, since $v^*, v^{**} \in \zer A = \Fix J_{M^{-1}A}$, we obtain $v^* = v^{**}$.
    
    \rev{If $M$ has closed range and $M=CC^*$ is onto, using \eqref{eq:resolvent_push_forward}, we obtain, since $C^*$ is onto, that for each $w^* \in \zer C^*\rhd A$ there exists $u^* \in \zer A$ such that $C^*u^*=w^*$, and thus
    \begin{equation}
    	\|w^k - w^*\|^2 = \| C^* u^k - C^* u^* \|^2 = \|u^k - u^*\|_M^2
    \end{equation}	
    admits a limit as $k\to+\infty$. Similarly, $\lim_{k \to+\infty}\|w^k - J_{C^*\rhd A}(w^k)\|^2= 0$. The classical Opial's lemma yields that $w^k\rightharpoonup w^* \in \zer C^*\rhd A$. Using now that $(A+M)^{-1}(C w^k)=J_{M^{-1}A}(u^k)$ for all $k \in \N$, which follows from \eqref{eq:preconditioned_resolvent}, and that $J_{M^{-1}A}(u^k)\rightharpoonup u^*$, we conclude the proof.}
\end{proof}

Naturally, a similar result holds in continuous time. We now describe how this lemma can be used to establish the convergence of several \rev{accelerated} splitting methods in a unified manner, \rev{thus demonstrating how Lemma~\ref{lem:opial_degenerate} extends the analysis of \cite{bredies2021degenerate} well beyond standard proximal point algorithms}. It is important to emphasize, however, that Lemma~\ref{lem:opial_degenerate} does not provide any information about the asymptotic behavior of the sequence $(u^k)_{k \in \N}$ itself, which requires a separate, method-specific analysis. We opted to present the adaptation of Theorem~\ref{thm:nonexp_second_order_discrete}, i.e., the one for algorithmic schemes. The same can be done for Theorem~\ref{thm:second_order_direct}, i.e., for continuous-time dynamics, but we leave this to the reader.

\begin{thm}\label{thm:second_order_discrete_degenerate}
	Let $A$ be a maximal monotone operator with $\zer A \neq \emptyset$, and $M$ a self-adjoint positive semidefinite preconditioner such that $(M+A)^{-1}$ is Lipschitz over $\Image M$ and $J_{M^{-1}A}$ has full domain. Let $(u^k)_{k \in \N}$ be the sequence\footnote{We denote by $(u^k)_{k \in \N}$ the sequence generated by Algorithm~\ref{alg:fast_km} to avoid confusion with $\bx^k=(x_1^k, \dots, x_N^k)$ generated by Graph-DRS discussed in Section \ref{sec:fast_graph_drs}.} generated by Algorithm~\ref{alg:fast_km} with $T=(1-s) I + sJ_{M^{-1}A}$, $s \in (0, 2]$, \rev{$\alpha>2$, $\theta \in (1, \alpha-1)$, and $\sigma >0$.} Then:
	\begin{enumerate}[label=(\roman*)]
		\item \label{item:summ_splitting} We have, for all $u^* \in \zer A$,
		\begin{equation*}
			\begin{aligned}
				&\rev{\sum_{k \in \N}  \bra  u^{k}-J_{M^{-1}A}(u^k), J_{M^{-1}A}(u^k)- u^* \ket_M < +\infty\,,}\\
				&\sum_{k \in \N} k\|u^k - J_{M^{-1}A}(u^k) \|_M^2 < \infty\,, \quad \sum_{k \in \N} k \|u^k- u^{k-1} \|_M^2 < +\infty \,.
			\end{aligned}
		\end{equation*}
		\item \label{item:conv_splitting} $(J_{M^{-1}A}(u^k))_{k \in \N}$ converges weakly to some $u^*\in \zer A$ as $k\to + \infty$. \rev{Additionally, if $s=1$, also $(u^k)_{k \in \N}$ converges weakly to $u^*$}.
		\item \label{item:rate_splitting} We have the following \textit{fast} rates of convergence:
		\begin{equation*}
			\begin{aligned}
			&\rev{ \bra u^k - J_{M^{-1}A}(u^k), J_{M^{-1}A}(u^k)-u^*\ket_M = o(k^{-1})\,,}\\
			&\|u^{k} - u^{k-1}\|_M  = o(k^{-1}) \quad \mbox{and} \quad \|u^k - J_{M^{-1}A}(u^k)\|_M  = o(k^{-1}) \quad \mbox{as} \ k \rightarrow +\infty\,.
			\end{aligned}
		\end{equation*}
	\end{enumerate}
	Suppose furthermore that $M$ has closed range and that $M=CC^*$ is onto with $C\colon D\to H$. Then, $w^k:=C^*u^k$ satisfies \ref{item:summ_splitting}-\ref{item:rate_splitting} with $M^{-1}A$ replaced by $C^*\rhd A$ and seminorms replaced by norms (in $D$), and $w^k\rightharpoonup C^* u^*$ as $k \to +\infty$.
\end{thm}
\begin{proof}
The proof \rev{of point \ref{item:summ_splitting} and point \ref{item:rate_splitting}} follow exactly the same lines of Theorem~\ref{thm:nonexp_second_order_discrete} by replacing every norm by $M$-seminorms, every inner product with $M$-semi-inner products and setting $Q=s(I-J_{M^{-1}A})$. In fact, the only property utilized is the nonexpansiveness of $T$, which in this context translates as the $M$-nonexpansiveness of $sJ_{M^{-1}A}+(1-s)I$ for $s \in (0, 2]$ see \cite[Remark~2.7]{bredies2021degenerate}. Specifically, \ref{item:summ_splitting} can be proven by following the arguments in the proof Theorem~\ref{thm:nonexp_second_order_discrete}\ref{item:nonexp_discr_summability}. This yields, in particular, that 
\begin{equation}
	\lim_{k\to +\infty} \|u^k - J_{M^{-1}A}(u^k)\|_M = 0\,. 
\end{equation}
Similarly, statement \ref{item:rate_splitting} can be shown following the proof of Theorem~\ref{thm:nonexp_second_order_discrete}\ref{item:nonexp_discr_rates}.

Regarding \ref{item:conv_splitting}, following the proof of statement Theorem~\ref{thm:nonexp_second_order_discrete}\ref{item:nonexp_discr_iters}, we get that $\lim_{k \to +\infty} \|u^k - u^*\|_M$ exists for any $u^* \in \zer A$. We are now in the setting of the degenerate Opial lemma (Lemma \ref{lem:opial_degenerate}), which yields \rev{weak convergence of $(J_{M^{-1}A}(u^k))_{k \in \N}$ to some $u^* \in \zer A$.} \rev{To conclude, we pick $s=1$ and establish the weak convergence of $(u^k)_{k \in \N}$ to $u^*$ by showing that $\varphi_k:=\tfrac{1}{2}\|u^k - J_{M^{-1}A}(u^k)\|^2$ vanishes as $k\to+\infty$. Using \eqref{eq:fast_degenerate_ppp} and denoting for brevity $y^k:= u^k - J_{M^{-1}A}(u^{k})$ and $\Delta_k := J_{M^{-1}A}(u^{k+1})-J_{M^{-1}A}(u^{k})$ for all $k \geq 0$, we get
\begin{equation*}
	\begin{aligned}
		y^{k+1} &= u^k + \theta_k(J_{M^{-1}A}(u^k)-u^k) + \alpha_k \big(J_{M^{-1}A}(u^{k})-J_{M^{-1}A}(u^{k-1})\big) -J_{M^{-1}A}(u^{k+1}) \\
		&= (1 - \theta_k) y^k + \alpha_k \big(J_{M^{-1}A}(u^{k})-J_{M^{-1}A}(u^{k-1})\big) + J_{M^{-1}A}(u^{k})-J_{M^{-1}A}(u^{k+1}) \,.\\
        &= (1 - \theta_k) y^k + \alpha_k \Delta_{k-1} - \Delta_k\,.
	\end{aligned}
\end{equation*}
Using the Young inequality with weight $\xi_k > 0$ and \eqref{eq:lipschitz_bound_to_M} we can write 
\begin{equation}\label{eq:descent_varphi}
	\begin{aligned}
	\varphi_{k+1} - \varphi_{k}&\leq  - \big(1- (1 - \theta_k)^2(1 + \xi_k )\big)\varphi_k + \big(1+ \xi_k^{-1}\big)\big(\|\Delta_{k-1}\|^2 + \|\Delta_{k}\|^2\big)\\
    &\leq  - \big(1- (1 - \theta_k)^2(1 + \xi_k )\big)\varphi_k + \big(1+ \xi_k^{-1}\big)L\big(\|u^k - u^{k-1}\|_M^2 + \|u^{k+1} - u^{k}\|_M^2\big)\,.
	\end{aligned}
\end{equation}
Picking now, for all $k$ large enough such that $\theta_k < 2$,
\begin{equation}
    \xi_k = \theta_k\frac{2-\theta_k}{2(1-\theta_k)^2}\,, \quad \text{we get:} \quad \varepsilon_k :=  \big(1- (1 - \theta_k)^2(1 + \xi_k )\big) =  \frac{\theta_k}{2} (2-\theta_k)\,.
\end{equation}
Therefore, since $(\varepsilon_k)_{k \in \N} \notin \ell^1$ by definition of $\theta_k$ and $(k\|u^{k}-u^{k-1}\|_M^2)_{k \in \N}$ is summable from point \ref{item:summ_splitting}, it follows from \eqref{eq:descent_varphi} that $\varphi_k\to 0$. Thus, $u^k\rightharpoonup u^*$ weakly as well.
}
\end{proof}

\begin{rmk}
	If $\alpha = \sigma = 2$, we are in the setting of the OHM method (cf.~Section \ref{sec:halpern}), which produces a sequence that converges \textit{strongly}. In this case, if $M$ has closed range, there is indeed no need to rely on a different analysis with degenerate preconditioners since the map $(M+A)^{-1}C$ is Lipschitz, and the convergence of $(u^k)_{k \in \N}$ would follow from that of $(w^k)_{k \in \N}$.
\end{rmk}

\subsection{Application to \rev{Splitting Algorithms}}\label{sec:application_to_splitting_algs}

In this section, we illustrate how Theorem~\ref{thm:second_order_discrete_degenerate} can be applied to show complete convergence results for accelerated splitting algorithms. \rev{We consider only three notable examples, but the analysis applies to a broad range of algorithms such as those in \cite[Chapter 4]{chenchene2023splitting} or in \cite[Section 3]{bredies2021degenerate}. See also the recent contribution \cite{syzz2025} for an application to ADMM, using similar techniques to ours.} 

\subsubsection{Fast PDHG}\label{sec:fast_chambolle_pock} \rev{We start from the PDHG method by Chambolle--Pock \cite{cp11} described in Example \ref{example:chambolle_pock}. Applying \eqref{eq:fast_degenerate_ppp} with $s\in (0, 2]$, and $J_{M^{-1}A}$ build as in Example \ref{example:chambolle_pock}, with step sizes satisfying $\tau_1\tau_2\|L\|^2 \geq 1$, we obtain a variant  of the PDHG algorithm \eqref{eq:chambolle_pock} enhanced with momentum. Denoting by $\wt{u}^{k+1}:=(\wt{x}^{k+1}, \wt{y}^{k+1}):=J_{M^{-1}A}(u^k)$ and recalling the definition of the primal-dual gap $\cG_{u^*}$ in \eqref{eq:primal_dual_bound}, Theorem \ref{thm:second_order_discrete_degenerate}\ref{item:rate_splitting}, yields the following fast rates on the last iterate:
\begin{equation}
	\|u^k - J_{M^{-1}A}(u^k)\|_M = o(k^{-1})\,, \quad \text{and} \quad \cG_{u^*}(\wt{x}^{k+1}, \wt{y}^{k+1}) = o(k^{-1}) \quad \text{as}\  k \to+\infty\,,
\end{equation}
for all $u^*:=(x^*, y^*)$ saddle point of \eqref{eq:lagrangian}, thus improving the corresponding $o(k^{-\frac{1}{2}})$ rates for the standard PDHG method. Additionally, from Theorem \ref{thm:second_order_discrete_degenerate}\ref{item:iters}, we obtain the weak convergence of $(\wt{u}^{k})_{k \in \N}$ and (if $s=1$) of $(u^{k})_{k \in \N}$ to the same saddle point of $\cL$. Note that in the non-degenerate setting  $\tau_1\tau_2\|L\|^2 > 1$ the application of OHM, i.e., Algorithm \ref{alg:fast_km} in the edge case $\alpha=\sigma = 2$, to PDHG was proposed in \cite{kim2021}.
}

\rev{
\subsubsection{Fast Douglas--Rachford Splitting}\label{sec:fast_drs} In the special case of Douglas--Rachford iterations according to Example \ref{example:drs}, applying \eqref{eq:fast_degenerate_ppp} we obtain, for $\alpha > 2$, $\theta \in (1, \alpha-1)$, the \emph{Fast-Douglas--Rachford} method:
\begin{equation}\label{eq:fast_drs}
	\left\{
	\begin{aligned}
		& x_1^{k+1} =J_{A_1}(w^k)\,,\\
		& x_2^{k+1}= J_{A_2}(2x_1^{k+1}-w^k)\,,\\
		& w^{k+1}=  w^k + \alpha_k(w^k - w^{k-1}) + s\theta_k (x_2^{k+1}-x_1^{k+1}) +  s\alpha_k\big(x_2^{k+1}-x_2^k-(x_1^{k+1} - x_1^k)\big)\,.
	\end{aligned}\right. 
\end{equation}
Theorem \ref{thm:second_order_discrete_degenerate}, together with the degenerate preconditioned proximal point formulation of the method described in Example \ref{example:drs}, allows us to improve the worst-case rate \eqref{eq:rate_douglas_rachford} from $o(k^{-\frac{1}{2}})$ to
\begin{equation}\label{eq:fast_rate_drs}
	\|x_2^{k+1} - x_1^{k+1}\| = o(k^{-1})  \quad \text{as} \ k \to +\infty\,,
\end{equation}
while keeping the weak convergence of the shadow sequences $(x_1^{k+1})_{k \in \N}$ and $(x_2^{k+1})_{k \in \N}$ to the same zero of $A_1 + A_2$. Let us underscore that while \eqref{eq:fast_rate_drs} and the weak convergence of $(w^k)_{k \in \N}$ follow from classical theory and were established in \cite{botkhoa2023} in the corresponding special case, the weak convergence of the shadow sequences requires dedicated analysis. Here, it follows from Theorem \ref{thm:second_order_discrete_degenerate}\ref{item:conv_splitting}.
}

\rev{
\subsubsection{Fast Graph Douglas--Rachford Splitting}\label{sec:fast_graph_drs}
In this section, we show the application of Theorem~\ref{thm:second_order_discrete_degenerate} to Graph-DRS presented in Example \ref{example:graph_drs}. Applying \eqref{eq:fast_degenerate_ppp}, we get, for $k\geq 0$ and $\bw^{-1}, \bw^0\in H^{N-1}$, $\alpha>2$, and $\theta \in (0, 1)$:
\begin{equation}\label{eq:fast_graph_drs}
	\left\{
	\begin{aligned}
		& x_i^{k+1} =J_{\frac{\tau}{d_i}A_i}\bigg(-\frac{2}{d_i}\sum_{h=1}^{i-1}\left({L}_{hi}+\hat{{L}}_{hi}\right)x_h^{k+1} + \frac{1}{d_i}\sum_{j=1}^{N-1}Z_{ij}w_j^k\bigg)\,, \quad \text{for all} \ i \in \llbracket 1, N\rrbracket\,,\\
		& w_j^{k+1} = w_j^k + \alpha_k(w_j^k - w_j^{k-1}) - s\theta_k\sum_{i=1}^{N} Z_{ij} x_i^{k+1} - s\alpha_k \sum_{i=1}^{N} Z_{ij}(x_i^{k+1}-x_i^k)\,, \quad \text{for all} \ j \in \llbracket1, N-1\rrbracket\,.
	\end{aligned}\right.
\end{equation}
Also in this case, Theorem \ref{thm:second_order_discrete_degenerate}, together with the degenerate preconditioned proximal point representation of the method mentioned in Example \ref{example:graph_drs}, allows us to improve the rates in \eqref{eq:def_variance} yielding
\begin{equation}
	\Var(\bx^{k+1}) = o(k^{-2}) \quad \text{as} \ k\to +\infty\,, 
\end{equation}
keeping the weak convergence of the solution estimates $(x_i^{k+1})_{i\in \N}$ to the same solution $x^*\in H$ of \eqref{eq:Nop}.}%

\section{Numerical Experiments}\label{sec:num}

In this section, we present our numerical experiments, which were performed in Python on a 12th-Gen.~Intel(R) Core(TM) i7--1255U, $1.70$--$4.70$ GHz laptop with $16$ Gb of RAM. The code is available for reproducibility at \url{https://github.com/echnen/fast-km}. 

\subsection{Toy problems}\label{sec:num_preparations} 
In these first numerical experiments, we consider two main instances of $T$:
\begin{itemize}[left=0pt]
\item We consider an optimization problem that can be written as
	\begin{equation}\label{eq:sum_problem}
		\min_{x \in \R^d}\ f_1(x)+\dots + f_N(x)\,, 
	\end{equation}%
	for $f_1,\dots, f_N\colon \R^d\to \R\cup\{+\infty\}$ proper, convex and lower semicontinuous such that $\prox_{\tau f_i}$ can be easily computed for each $\tau >0$. In this case, we take $T$ as given in \eqref{eq:general_drs} \rev{applied to the first-order optimality conditions of \eqref{eq:sum_problem},} and thus use Fast Graph-DRS initialized with $\bw^{-1}=\bw^0=0$.
\item As in \cite{botkhoa2023}, for dimension $d\in \N$ even, we consider $T\colon \R^d\to \R^d$ defined by  $T(x):=J_{\tau \Sigma}(x)$, where $\tau =10^{-1}$ and $\Sigma$ is the skew-symmetric (hence monotone) matrix given block-wise by:
	\begin{equation*}
		\Sigma := \begin{bmatrix}
			0 & I\\
			-I & 0
		\end{bmatrix}\,,
	\end{equation*}
	where $I$ is now the identity operator in $\R^{d/2}$ and $0$ represents the zero operator in $\R^{d/2}$. In this case, we use Algorithm~\ref{alg:fast_km} initialized with $x^{-1}=x^{0}=\mathbf{1}$.
\end{itemize}

\subsection{The influence of parameters}\label{sec:num_parameters} 
As we introduced two degrees of flexibility in Algorithm~\ref{alg:fast_km} consisting in the parameter \rev{$\theta \in (1, \alpha-1)$} and $\sigma >0$, we first numerically investigate the influence of these two parameters separately. We consider the two instances of $T$ \rev{described in} Section~\ref{sec:num_preparations}. In the first one, we set $d=2$, $N=2$ and thus \rev{apply the Fast-Douglas--Rachford method according to Section \ref{sec:fast_drs} to}:
\begin{equation*}
	f_1(x) := 10^{-3}|x|\,, \quad \text{and} \quad f_2(x):=\frac{1}{2}\dist^2(x, \mathbb{B})\,, \quad x \in \R^2\,,
\end{equation*}
where $\mathbb{B}$ is the unit euclidean ball centered at $(1,1)^T$. The two proximity operators are easily accessed through, soft-thresholding and averaged projection onto $\mathbb{B}$, respectively; see \cite[Example~14.5 and Proposition~24.8(vii)]{BCombettes}. In the second one, we set $d=10$. 

\subsubsection{Influence of \texorpdfstring{$\theta$}{$\theta$}}\label{sec:num_effect_of_eta}

In this section, we perform a systematic comparison between $\eta=\tfrac{1}{2}$, which was introduced (when $\sigma= \alpha + 1$) in \cite{botkhoa2023} and the choices $\eta \in (0, 1)$, \rev{interpolating the parameter ranges $\theta \in (1, \alpha-1)$}. Specifically, we consider $\alpha=2$, $4$, $16$, $32$ and two different choices of $\sigma$, $\sigma=2$ and \rev{$\sigma = 17$}. Note in particular that the special case of \cite{botkhoa2023} is considered with $\alpha = 32$. When $\alpha=2$ the methods with different values of $\eta$ all coincide (since $\rev{\theta} = 1$ for all $\eta \in (0, 1)$) and, \rev{when $\sigma=2$}, the resulting method reduces to OHM, see Section \ref{sec:halpern}. We plot the results in Figure~\ref{fig:experiment_1_pd} and Figure~\ref{fig:experiment_1_mt}.

\rev{Our numerical experiments demonstrate that Algorithm~\ref{alg:fast_km} performs particularly well compared to the baseline choice $\eta = \tfrac{1}{2}$. As shown in Figure~\ref{fig:experiment_1_pd}, performances improve noticeably when \rev{$\theta > \tfrac{\alpha}{2}$}; in fact, the larger the relaxation parameter, the faster the method tends to converge. However, this behavior does not occur in Figure~\ref{fig:experiment_1_mt}, which corresponds to the case where $T$ is the resolvent of a skew-symmetric matrix. In that setting, the baseline choice $\eta = \tfrac{1}{2}$ yields the best performance.

We suspect that this contrasting behavior stems from an intrinsic trade-off between minimizing the gap function and minimizing the fixed-point residual. An inspection of the constants in \eqref{eq:explicit_rates} suggests that larger relaxation parameters favor gap-function reduction, while their effect on the fixed-point residual appears to be strongly problem-specific. In particular, when $T$ is the resolvent of a skew-symmetric operator, the gap function is identically zero, making the residual the sole quantity of interest.}

\rev{Eventually, we can also observe that while OHM ($\alpha=\sigma= 2$), i.e., the accelerated Douglas--Rachford method in the sense of Kim \cite{kim2021}, achieves an \emph{optimal} non-asymptotic convergence rate \cite{ykr24}, Algorithm \ref{alg:fast_km}, which exhibits \emph{order-}optimal but little-$o$ rates, can perform much better in practice.}

\begin{figure}[t]
	\begin{subfigure}{0.49\linewidth}
		\includegraphics[width=\linewidth]{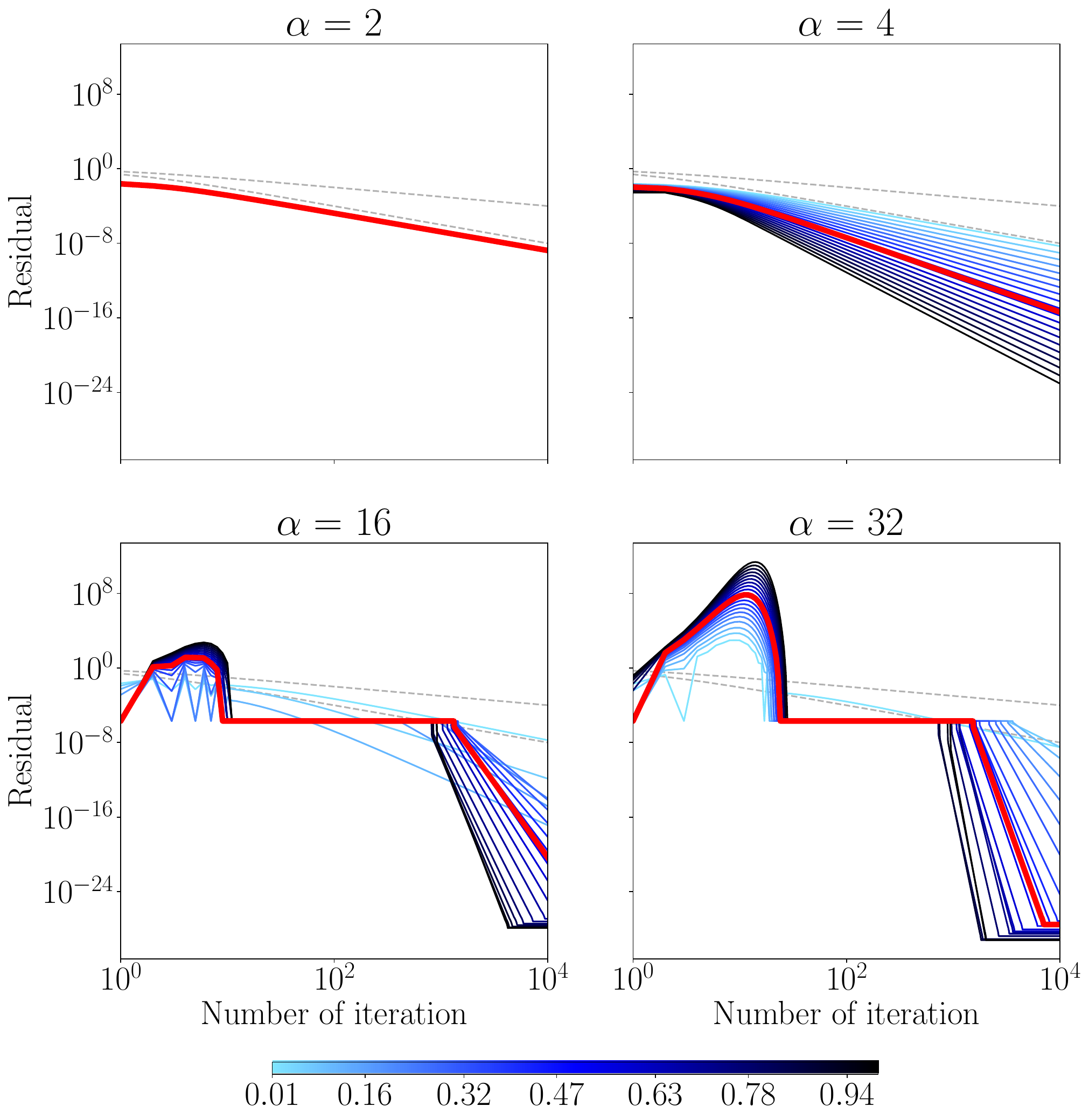}
		\caption{The case $\sigma = 2$}
		\label{fig:experiment_1_1}
	\end{subfigure}
	\begin{subfigure}{0.49\linewidth}
		\includegraphics[width=\linewidth]{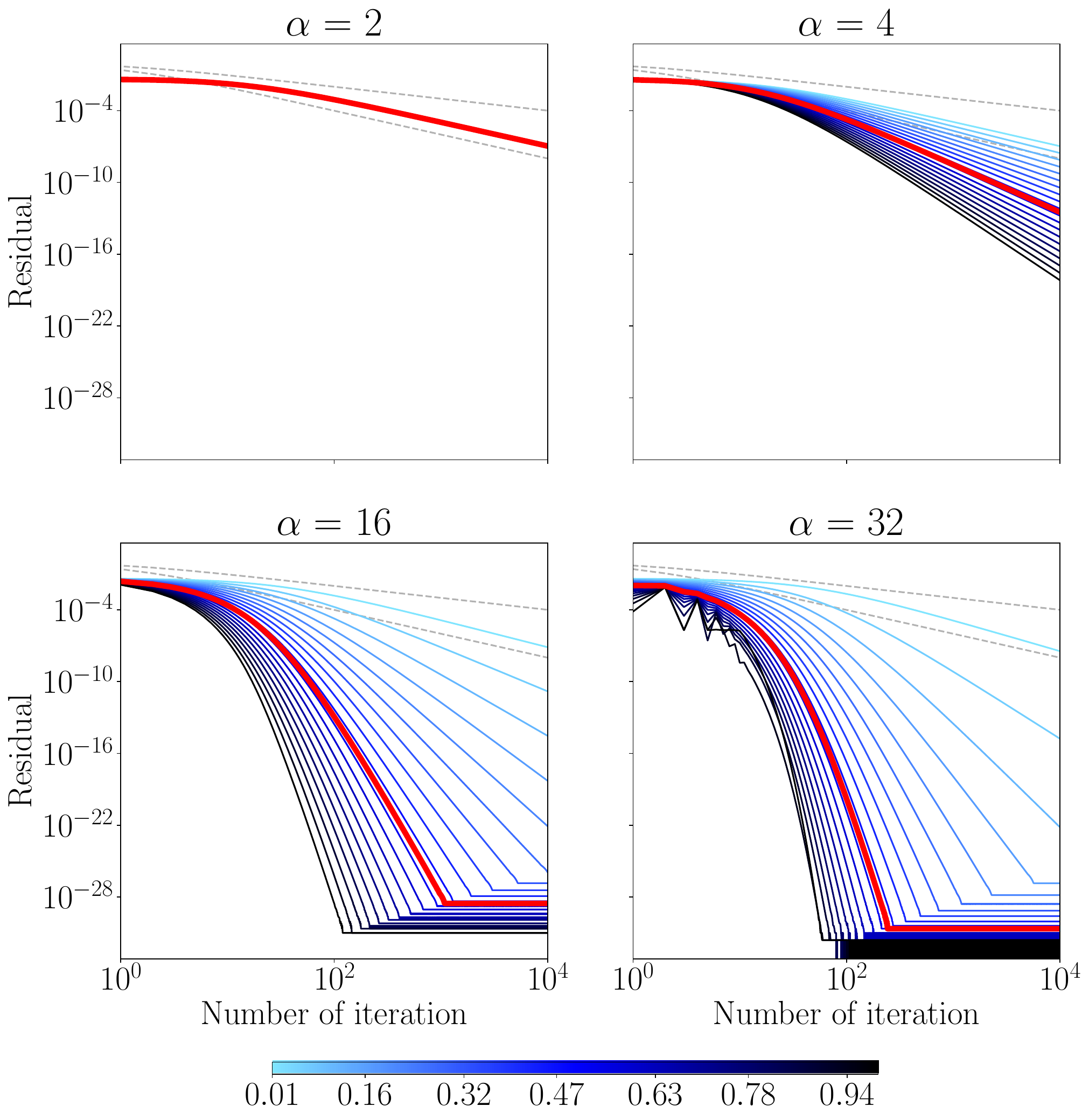}
		\caption{The case $\sigma = 17$}
		\label{fig:experiment_1_2}
	\end{subfigure}
	\caption{Results of Section \ref{sec:num_effect_of_eta} with the first instance of $T$ described in Section~\ref{sec:num_preparations}: Comparison of different choices of $\eta \in (0,1)$, with corresponding colors indicated in the colorbar. The case $\eta=\tfrac{1}{2}$ is highlighted in red.}
	\label{fig:experiment_1_pd}
\end{figure}

\begin{figure}[t]
	\begin{subfigure}{0.49\linewidth}
		\includegraphics[width=\linewidth]{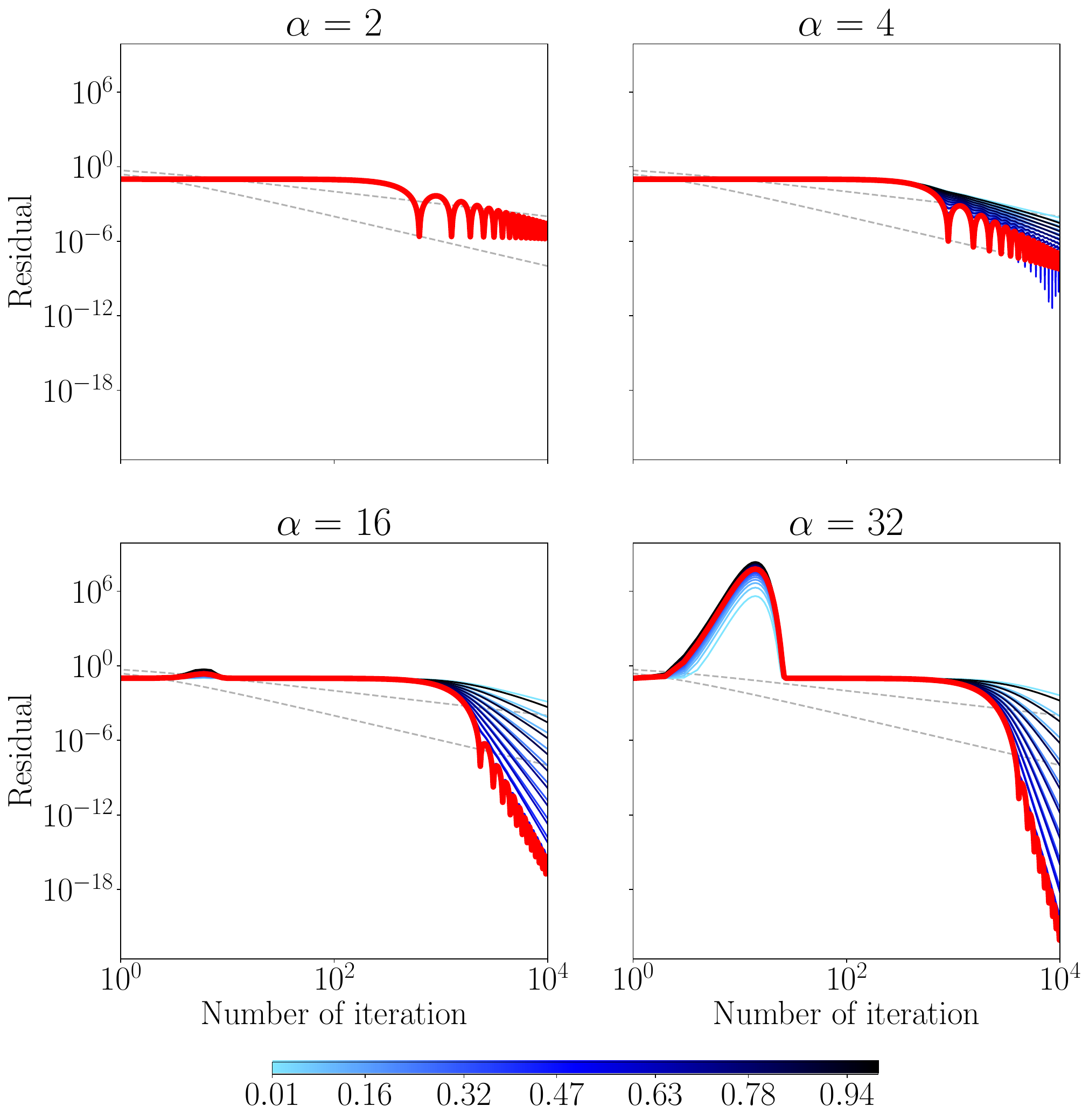}
		\caption{The case $\sigma = 2$}
		\label{fig:experiment_1_3}
	\end{subfigure}
	\begin{subfigure}{0.49\linewidth}
		\includegraphics[width=\linewidth]{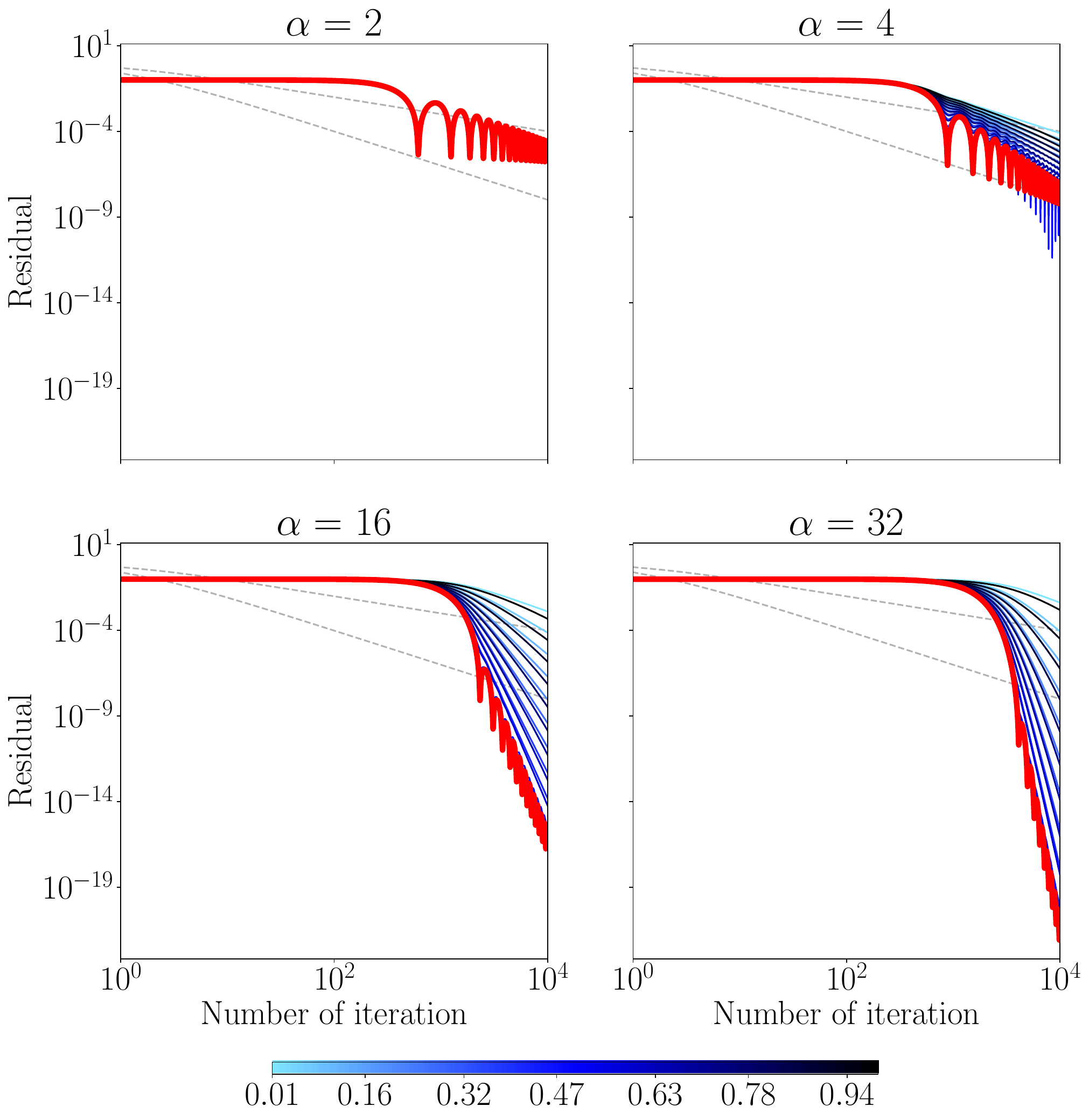}
		\caption{The case $\sigma = 17$}
		\label{fig:experiment_1_4_mt}
	\end{subfigure}
	\caption{Results of Section \ref{sec:num_effect_of_eta} with $T$ as described the second instance of $T$ described in Section~\ref{sec:num_preparations}: Comparison of different choices of $\eta \in (0,1)$, with corresponding colors indicated in the colorbar. $\eta=\tfrac{1}{2}$ is highlighted in red.}
	\label{fig:experiment_1_mt}
\end{figure}

\subsubsection{The influence of \texorpdfstring{$\sigma$}{sigma}}\label{sec:num_effect_of_sigma} Keeping in mind \eqref{eq:discrete_second_order}, the parameter $\sigma$ is related to the starting time $t_0$ in the continuous dynamics \eqref{eq:second_order_dynamic_general}. Recall in fact that the discrete time is defined by $t_k= k - 1 + \sigma $ so $t_0= \sigma - 1$. We test $20$ different choices of $\sigma$ spaced evenly from $1$ to $100$ and take the four choices $\alpha=2$, $4$, $16$, $32$. We put particular emphasis on the choice \rev{$\sigma =\alpha$, which, for $\alpha=2$ leads to OHM} \cite{lieder21, ykr24}. We consider only the operator $T$ corresponding to the second instance of Section \ref{sec:num_preparations}. The results are depicted in Figure~\ref{fig:experiment_2_pd}. We refrain from displaying the results for the first instance of $T$ in Section~\ref{sec:num_preparations} since the results therein are almost indistinguishable. Regarding the tuning of the parameter $\sigma$, our experiments suggest indeed setting $\sigma =\alpha$.

\begin{figure}[t]
	\begin{subfigure}{0.49\linewidth}
		\includegraphics[width=\linewidth]{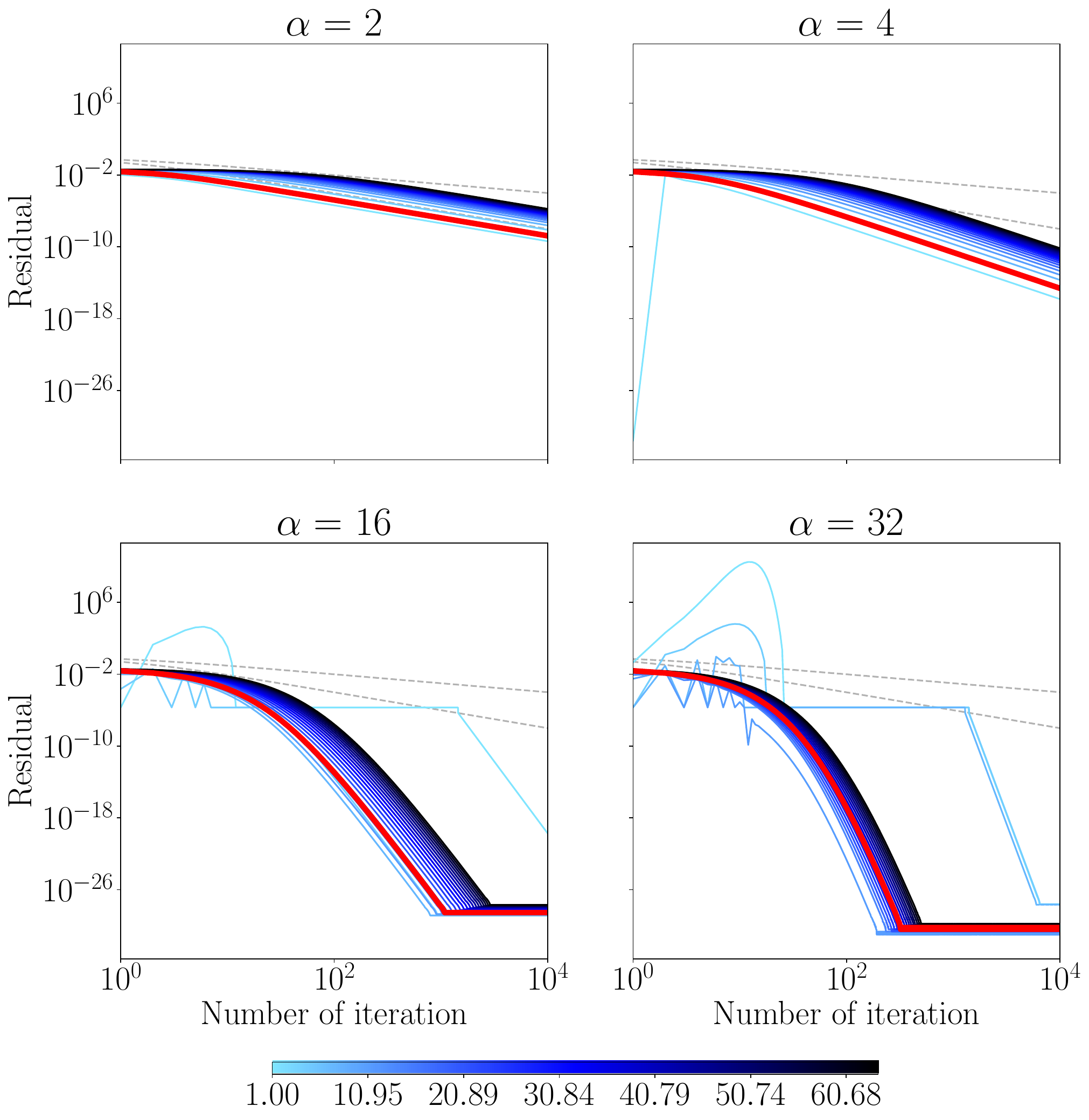}
		\caption{The case $\eta = 0.5$}
		\label{fig:experiment_2_1}
	\end{subfigure}
	\begin{subfigure}{0.49\linewidth}
		\includegraphics[width=\linewidth]{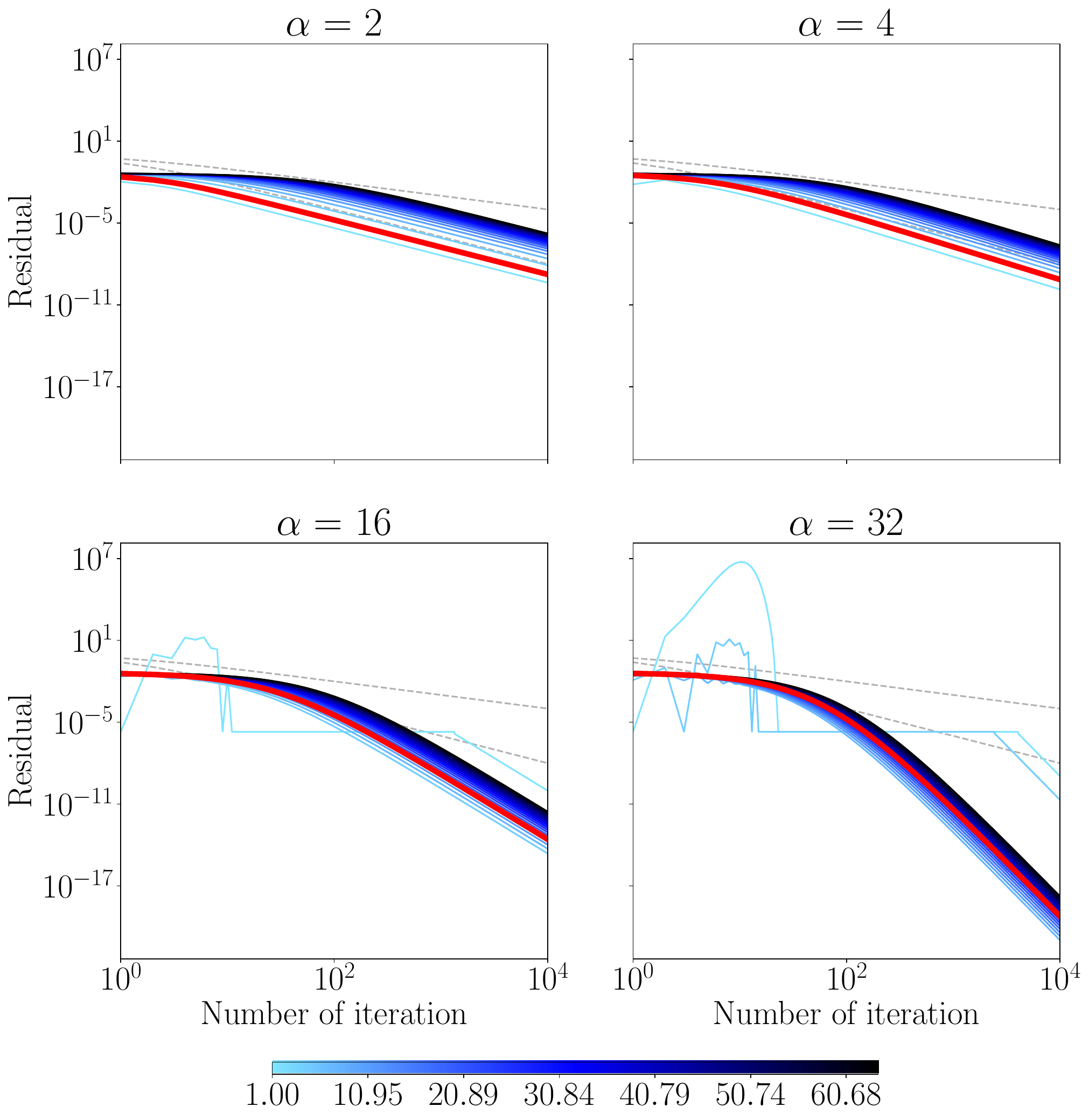}
		\caption{The case $\eta = 0.1$}
		\label{fig:experiment_2_2}
	\end{subfigure}
	\caption{Results of Section \ref{sec:num_effect_of_sigma}: Comparison between different choices of $\sigma\in (0,100)$. $\sigma=\alpha$ is highlighted in red.}
	\label{fig:experiment_2_pd}
\end{figure}

\subsection{Cooling Heuristics}\label{sec:num_cooling}
Motivated by the insights of Sections~\ref{sec:num_effect_of_eta}, and \ref{sec:num_effect_of_sigma}, we propose a ``\textit{cooling}'' heuristics\footnote{The name is borrowed from Consensus Based Optimization literature \cite{fhpp22}.}, which consists in gradually increasing the parameter $\alpha$ along the iterations until a certain upper bound $\alpha_{\max}$ is reached: $\alpha_0<\alpha_1<\dots < \alpha_{{\lfloor \texttt{maxit}/2\rfloor}} = \alpha_{\max}$. Specifically, we consider two different strategies: a \textit{linear} and a \textit{logarithmic} parameter scaling. In either case, we set $\alpha_0 >0$, $\alpha_{\max}=100 \alpha_0$, and we fix a maximum number of iterations \texttt{maxit}. Then, in first case, at each iteration we increase $\alpha$ with a fixed step size, in such a way that $\alpha_{\max}$ is reached for $k = \lfloor\texttt{maxit}/2\rfloor$. In the second case, the increase is defined in a logarithmic scale.

\rev{To set up our numerical experiment,} we consider $\alpha = 2$, $4$, $16$, $32$, and the operator $T$ as given in the second choice in Section~\ref{sec:num_preparations}. For each $\alpha$, we compare the performance of the vanilla Krasnoselskii--Mann (KM) method, Fast-KM with $\eta=\tfrac{1}{2}$ and $\alpha=\sigma$, and Fast-KM with linear and logarithmic scaling in the cooling heuristics and $\eta = \frac{1}{2}$, $\alpha_0 = \alpha$, abbreviated (Fast-KM-LNC) and (Fast-KM-LGC), respectively. Recall that $\alpha_{\max} = 100 \alpha_0$. The results are shown in Figure~\ref{fig:experiment_3}. A notable insight  from these experiment is the beneficial convergence speed of Algorithm~\ref{alg:fast_km} with large parameters $\alpha$, but without incurring into the undesirable initial instability state; see also \cite{AAViscosityParameter24} who thoroughly studied the choice of $\alpha$ for the ODE \eqref{eq:intro_general_second_order} when $\beta(t) \equiv 0$, $b(t) \equiv 1$ and $Q$ is the gradient of a convex function.

\begin{figure}[t]
\includegraphics[width=\linewidth]{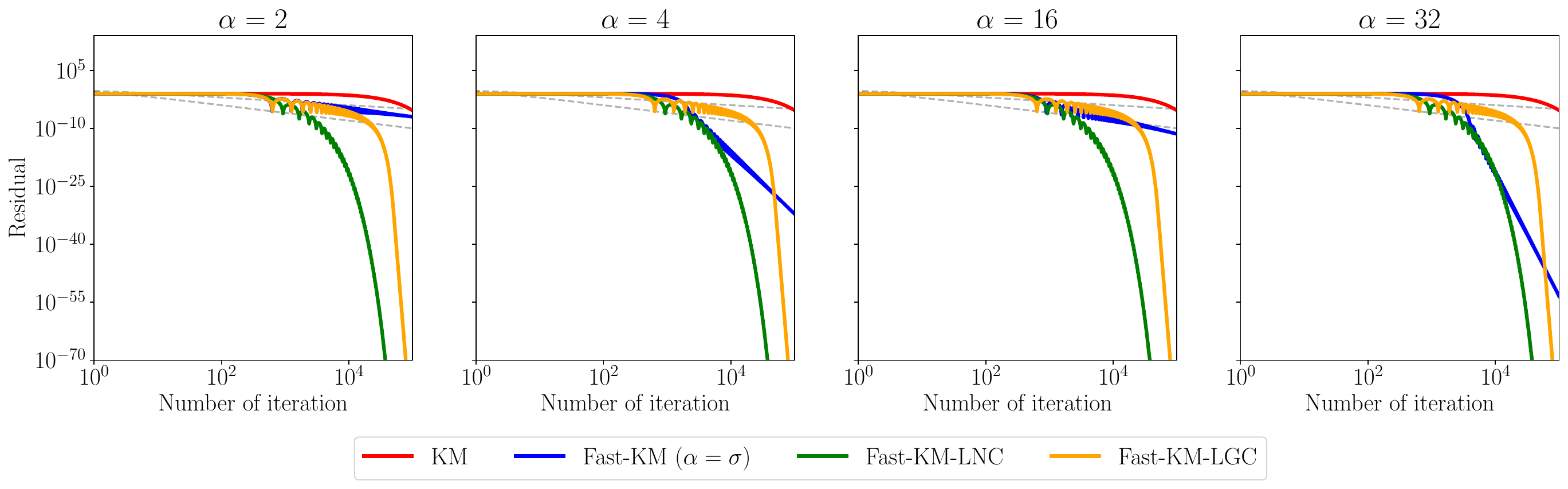}
\caption{Results of Section \ref{sec:num_cooling}: Testing cooling heuristics. }
\label{fig:experiment_3}
\end{figure}

\subsection{Application to Realistic Problems}\label{sec:num_app}

In this section, we test Algorithm~\ref{alg:fast_km} to solve optimization problems arising from realistic settings. We start with an application to Optimal Transport (OT), and in particular, to the \textit{Beckmann minimal flow formulation} of the Wasserstein $1$ distance, where DRS proved particularly effective \cite{ppo14, bc15, carlier2023wassersteinmediansrobustnesspde}. In the second experiment, we consider the geometric median problem, a standard nonsmooth benchmark especially for distributed optimization.

In these two experiments, we compare the performance of i) the vanilla Krasnoselskii--Mann (KM) method, ii) Algorithm~\ref{alg:fast_km} with $\eta = 0.5$, iii) Algorithm~\ref{alg:fast_km} with \rev{$\eta =0.9$} and iii) Algorithm~\ref{alg:fast_km} with a linear cooling heuristics as introduced in Section \ref{sec:num_cooling}. The results are reported in Figure~\ref{fig:experiment_applications}.

\subsubsection{Optimal Transport}\label{sec:num_ot}

In this experiment, we tackle a high dimensional OT problem, which is usually solved in its \rev{regularized} primal form using the celebrated Sinkhorn Algorithm~\cite{BenamouBregmanOT15}. Consider a square grid $\Omega:=\{(i, j) \mid i, j\in \llbracket 1, p\rrbracket \}$, with $p\in \N$, and two discrete probability measures $\mu$, and $\nu$ on $\Omega$, i.e., $\mu$, $\nu\in \R^n_+$ with $n=p^2$ such that $\sum_{i=1}^n \mu_i=\sum_{i=1}^n \nu = 1$. Denoted by $\Div$ the discrete divergence operator with Neumann boundary conditions on $\Omega$ defined trough finite differences (see, e.g., \cite{carlier2023wassersteinmediansrobustnesspde}), we aim at solving the following convex problem:
\begin{equation}\label{eq:beckmann}
	\min_{\sigma \in \R^{n\times 2}} \ \|\sigma\|_{2,1}\,, \quad \text{s.t.:}\quad \Div \sigma = \mu-\nu\,, 
\end{equation}
where $ \|\sigma\|_{2,1}$  is the \textit{group-lasso} penalty, i.e., $\|\sigma\|_{2, 1}:=\sum_{i=1}^{n}\|\sigma_i\|$, and for each $i\in \llbracket 1, n\rrbracket$, $\sigma_i\in \R^2$. \rev{Specifically, we consider $p=100$, which represents a relatively large scale instance for OT.}

\rev{This \emph{basis pursuit}-type problem can addressed by applying the \rev{Chambolle--Pock} method (cf.~Example \ref{example:chambolle_pock}) with $f(\sigma):=\|\sigma \|_{2,1}$, $g(\sigma):=\iota_{\{\mu - \nu\}}(\sigma)$, and $L := \Div$.} These two functions are indeed \textit{prox-friendly}. $\prox_{\tau f}$ is the group soft-thresholding operator (see \cite[Example~14.5]{BCombettes}), while $\prox_{\tau g}$ is just the constant function mapping to $\mu - \nu$. We apply the \rev{Fast-Chambolle--Pock solver} (Section \ref{sec:fast_chambolle_pock}) to solve \eqref{eq:beckmann} with $\tau_1 \tau_2 = 10^{-1}$, which is a lower bound for $\|L\|^2$ and $\tau_1 =10^{-3}, 10^{-2}$ and $10^{-1}$.

As can be seen from Figure~\ref{fig:experiment_ot}, the Fast-Chambolle--Pock method emerges here as a great practical alternative to the plain Chambolle--Pock: the method outperformed the standard approach with KM, \rev{especially in regimes with high oscillations}, and confirmed once again that \rev{large relaxation parameters are indeed beneficial in practice}.

\subsubsection{Solving the Geometric Median Problem}\label{sec:num_median}

We now turn to computing the geometric median of $N$ sample points in $\R^d$. Given a sample $\{x_i\}_i$ of $N$ points in $\R^d$ we aim at solving
\begin{equation}\label{eq:median}
	\min_{x \in \R^d} \  \|x-x_1\|+\dots + \|x-x_N\|\,.
\end{equation}
To solve \eqref{eq:median}, we consider the graph-DRS operator in \eqref{eq:general_drs} splitting \eqref{eq:median} in $N$ terms. \rev{We consider, $N=100$ and $d=100$.} All the terms $x\mapsto \|x-x_i\|$ are in fact prox-friendly with proximity operators that can be computed by simply applying the soft-thresholding operator suitably shifted by $x_i$. We apply \eqref{eq:fast_graph_drs} with $\hat{Z}=0$ and build an arbitrary $Z$ such that  $\ker Z^* = \Span\{\mathbf{1}\}$ \rev{and show the results in Figure \ref{fig:experiment_med}}. The same conclusions drawn from the OT experiment in Section~\ref{sec:num_median} remain true here.

\begin{figure}[t]
	\begin{subfigure}{0.27\linewidth}
		\includegraphics[width=\linewidth]{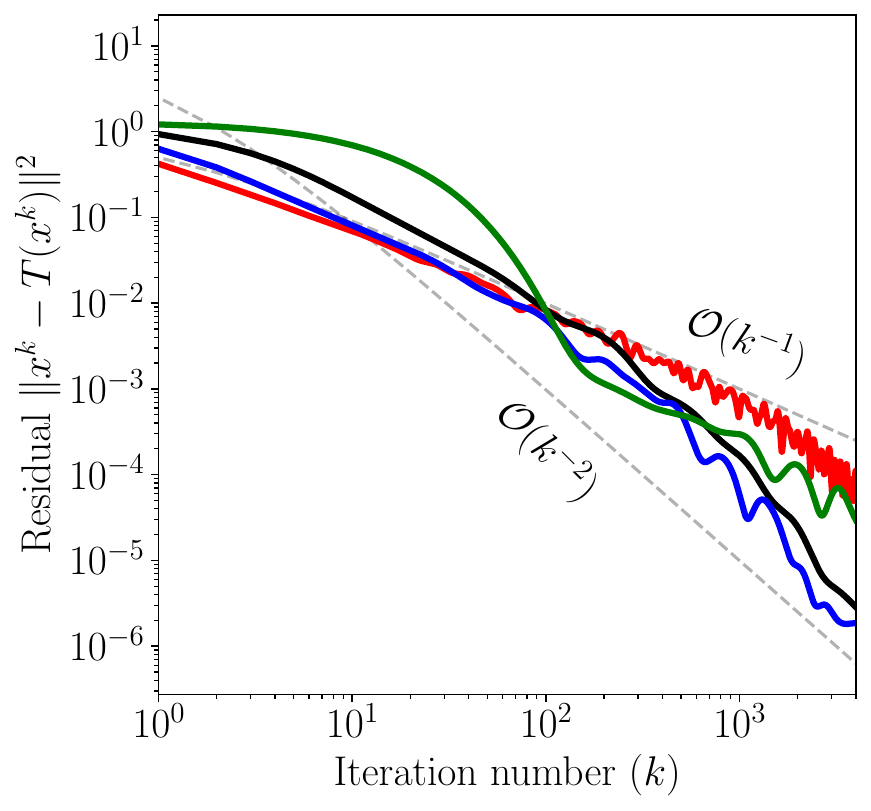}
		\caption{\rev{Geometric Medians}}
		\label{fig:experiment_med}
	\end{subfigure}
	\begin{subfigure}{0.71\linewidth}
		\includegraphics[width=\linewidth]{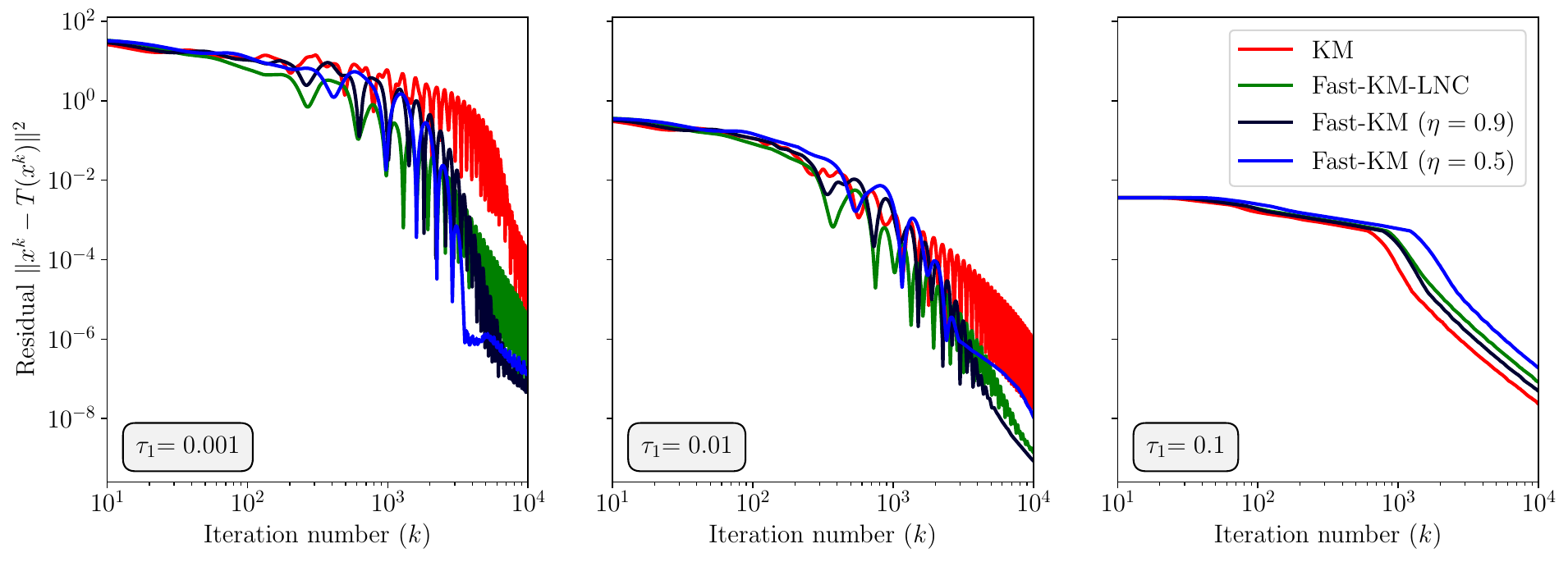}
		\caption{\rev{Optimal Transport}}
		\label{fig:experiment_ot}
	\end{subfigure}
	\caption{\rev{Results of Section \ref{sec:num_app}. In Figure \ref{fig:experiment_med}, we solve a geometric median problem (Section \ref{sec:num_median}), using Graph-DRS as the fixed-point operator $T$, cf.~Section \ref{sec:fast_graph_drs} and Example \ref{example:graph_drs}. In Figure \ref{fig:experiment_ot}, we solve an $\ell_1$ large scale optimal transportation problem, using Chambolle--Pock as the fixed-point operator $T$, cf.~Section \ref{sec:fast_chambolle_pock} and Example \ref{example:chambolle_pock}.}}
	\label{fig:experiment_applications}
\end{figure}

\section{Conclusion}

In this paper, we introduced a generalized version of the Fast-KM method proposed in \cite{botkhoa2023}, incorporating two additional degrees of flexibility through the parameters $\theta \in (1, \alpha-1)$ and $\sigma > 0$. We provided a new convergence proof based on a semi-implicit discretization of a continuous-time model. Notably, our approach encompasses the case $\alpha = 2$, which corresponds to the Optimal Halpern Method, thereby unifying the convergence analysis of these two acceleration mechanisms. We demonstrated how these results can be applied to operator splitting methods potentially with preconditioners, offering critical convergence guarantees for the so-called shadow sequences. Furthermore, our numerical experiments showcased the benefits afforded by the added flexibility of the method. This paper opens up new intriguing questions, among which: \rev{i) what is the tight non-asymptotic rate for Algorithm \ref{alg:fast_km}?} \rev{ii) Does Algorithm \ref{alg:fast_km} converge when $\theta=\alpha-1$?} We leave this for future research.

\section*{Appendices}
\addcontentsline{toc}{section}{Appendices}
\renewcommand{\thesubsection}{\Alph{subsection}}

\subsection{Discrete Calculus Rules}\label{sec:discrete_calculus}

\rev{Let $a:=(a^k)_{k \in \N}$, $b:=(b^k)_{k \in \N}$, and $c:=(c^k)_{k \in \N}$, be three real sequences. The product of two sequences $a$ and $b$ is denoted by $ab :=(a^k b^k)_{k \in \N}$. As usual, if $a=b$, we denote $a^2 = ab$. Recall that $\Delta a_k :=  a^{k+1} - a^k$ for all $k \in \N$, Therefore, for all $k\geq 0$ it holds
\begin{equation}\label{eq:discrete_calculus_rules}
	\begin{aligned}
		&\Delta (ab)_k = \Delta a_k b^k +  a^{k+1}\Delta b_k\,,\\
		&\Delta (abc)_k = \Delta a_k b^k c^k + a^{k+1}\Delta b_k c^k + a^{k+1}b^{k+1}\Delta c_k\,,\\
		& \Delta a^2_k = 2(a^{k+1}-a^k)a^k +(\Delta a_k)^2 = 2(a^{k+1} - a^k)a^{k+1} - (\Delta a_k)^2 \,.
	\end{aligned}
\end{equation}
This formalism can be readily extended to sequences in a Hilbert space $H$, with products replaced by inner products. For instance, if $(z^k)_{k \in \N}$ is a sequence in $H$, and $\zeta^k := \frac{1}{2}\|z^k\|^2$ for all $k \geq 0$, can write
\begin{equation}
	\Delta \zeta_k =  \bra \Delta z_k, z^k \ket + \frac{1}{2} \| \Delta z_k \|^2 = \bra \Delta z_k, z^{k+1} \ket  - \frac{1}{2} \| \Delta z_k \|^2\,.
\end{equation}
Similar rules as those in \eqref{eq:discrete_calculus_rules} apply for products of sequences in $H$ and in $\R$, for instance if $(\varepsilon^k)_{k \in \N}$ is a real sequence, then for all $k \geq 0$
\begin{equation}
	\Delta (\varepsilon z)_k = \Delta \varepsilon_k z^k +  \varepsilon^{k+1}\Delta z_k = \Delta z_k \varepsilon^k +  z^{k+1}\Delta \varepsilon_k\,.
\end{equation}}%

\subsection{Missing Proofs}\label{sec:appendix_lemmas}
The following proofs were postponed from the main part of the paper.

\begin{proof}[Proof of Lemma \ref{lem:derivative_of_Lyapunov_second_order}]
 For the sake of notation, we {also abbreviate $E^\lambda_t:=E^\lambda(t)$, $\dot{x}_t :=\dot{x}(t)$, \rev{$x_t := x(t)$}, $v^\lambda_t:=v^\lambda(t)$, $\delta_t:=\delta(t)$, and $\xi_t:=\xi(t)$}. Additionally, instead of writing $\frac{d}{dt}$ we write $\partial_t$. Using Leibniz rule, we get for almost all $t \geq t_0$:
	\begin{equation}\label{eq:lyapunov_decrease_1}
		\begin{aligned}
			\partial_t E^\lambda_t &= \dot{\delta}_t   \bra  Q_t, x_t-x^* \ket  + \delta_t  \bra  \partial_t Q_t, x_t-x^* \ket + \delta_t  \bra  Q_t, \dot{x}_t \ket \\
			& \quad + \frac{\dot{\xi_t}}{2}\|Q_t\|^2 + \xi_t  \bra  \partial_t Q_t, Q_t \ket + c \bra  x_t-x^*, \dot{x}_t  \ket +  \bra  v_t^{\lambda}, \dot{v}_t^\lambda  \ket \,.
		\end{aligned}
	\end{equation}
Let us compute the last inner product $ \bra  v^\lambda_t, \dot v^\lambda_t \ket $. For $\dot v_t^\lambda$ we have, \rev{using constitutive equation \eqref{eq:second_order_dynamic_general},} that for almost all $t \geq t_0$:
	\begin{equation*}
		\begin{aligned}
			\dot{v}^\lambda_t &=  (\lambda +1)\dot{x}_t + \rev{(1-\eta)} \beta_tQ_t + t\left(-\frac{\alpha}{t}\dot{x}_t - \beta_t \partial_t Q_t - b_t Q_t + \rev{(1-\eta)} \dot{\beta}_t Q_t + \rev{(1-\eta)} \beta_t \partial_t Q_t\right)\\
			& = (\lambda + 1-\alpha)\dot{x}_t + \rev{(1-\eta)} \beta_t Q_t - t\rev{\eta} \beta_t\partial_t Q_t + t(\rev{(1-\eta)}\dot{\beta}_t - b_t)Q_t\\
			& = -(\alpha-1-\lambda)\dot{x}_t - t\left(b_t -\rev{(1-\eta)} \dot{\beta}_t -\rev{(1-\eta)} \frac{\beta_t}{t}\right)Q_t - t\beta_t\rev{\eta}\partial_t Q_t\,.
		\end{aligned}
	\end{equation*}
	Denoting by $w_t :=  b_t -\rev{(1-\eta)} \dot{\beta}_t -\rev{(1-\eta)} \frac{\beta_t}{t}$, the inner product read for almost all $t \geq t_0$ as
	\begin{equation}\label{eq:v_times_vdot}
			\begin{aligned}
			 \bra  v^\lambda_t, \dot{v}^\lambda_t \ket  = & -\lambda(\alpha - 1- \lambda) \bra  x_t-x^*, \dot{x}_t   \ket - (\alpha - 1- \lambda)t\| \dot{x}_t\|^2 - \rev{(1-\eta)} (\alpha - 1- \lambda)t\beta_t \bra  Q_t, \dot{x}_t   \ket \\
			& - \lambda tw_t  \bra  x_t-x^*,Q_t \ket  -  t^2w_t  \bra  \dot{x}_t ,Q_t \ket   -  t^2 \beta_t w_t \rev{(1-\eta)} \| Q_t\|^2 \\
			& \rev{-} \rev{\eta} t\lambda \beta_t  \bra  x_t-x^*,  \partial_t Q_t \ket - \rev{\eta}t^2\beta_t  \bra  \dot{x}_t, \partial_t Q_t \ket - t^2\beta_t^2(1-\eta) \eta  \bra  Q_t,  \partial_t Q_t \ket \,.
		\end{aligned}
	\end{equation}
	Plugging \eqref{eq:v_times_vdot} into \eqref{eq:lyapunov_decrease_1}, \rev{using the definition of $\delta_t$, $\xi_t$ and $c$ from \eqref{eq:def_parameters},} and canceling out the appropriate terms, we obtain for almost all $t \geq t_0$ that
	\begin{equation}\label{eq:lyapunov_analysis_step_0}
		\begin{aligned}
			&\partial_t E^\lambda_t = \left(\dot{\delta}_t- t \lambda w_t\right)   \bra  Q_t, x_t-x^* \ket  - \rev{\eta} t^2\beta_t  \bra  \dot{x}_t,\partial_t Q_t \ket \\
			& \quad +\left(\delta_t - \rev{(1-\eta)} (\alpha - 1- \lambda)t\beta_t -  t^2 w_t\right)  \bra  Q_t, \dot{x}_t \ket  + \left(\frac{\dot{\xi}_t}{2}-\rev{(1-\eta)} t^2\beta_tw_t\right)\|Q_t\|^2 - (\alpha - 1- \lambda)t\|\dot{x}_t\|^2\,.
		\end{aligned}
	\end{equation}
	Let us study the quadratic form in the second line of \eqref{eq:lyapunov_analysis_step_0} separately\footnote{\rev{The coefficient $b_t$ in \eqref{eq:def_b} is in fact the only parameter that (i) does not depend on $\lambda$ (artificial parameter introduced in the Lyapunov analysis), and (ii) can make this quadratic form negative semidefinite for all $\lambda\in (0, \alpha -1]$. To see this, consider $\lambda =\alpha-1$. Then, $\delta_t$ must be equal to $t^2w_t$, which yields \eqref{eq:def_b}.}}. Plugging in the definitions of $\delta_t$, $\xi_t$ in \eqref{eq:def_parameters} and of $b$ in \eqref{eq:def_b}, we get, first, $w_t  = (\alpha -1)\rev{\eta}\frac{\beta_t}{t}$, hence, 
	\begin{equation*}
		\begin{aligned}
		&\delta_t - \rev{(1-\eta)}(\alpha - 1- \lambda)t\beta_t -  t^2 w_t = -t \beta_t (\alpha - 1 -\lambda)\,, \\
		&\frac{\dot{\xi}_t}{2}-\rev{(1-\eta)}t^2\beta_tw_t =  \beta_t t^2(1-\eta)\eta \left(\dot{\beta}_t + \frac{(2-\alpha)}{t}\beta_t\right)\,,
		\end{aligned}
	\end{equation*}
	so the quadratic form simplifies to:
	\begin{equation*}
		\begin{aligned}
			&\left(\delta_t - \rev{(1-\eta)}(\alpha - 1- \lambda)t\beta_t - t^2 w_t\right)  \bra  Q_t, \dot{x}_t \ket  + \left(\frac{\dot{\xi}_t}{2}- \rev{(1-\eta)} t^2\beta_tw_t\right)\|Q_t\|^2 - (\alpha - 1- \lambda)t\|\dot{x}_t\|^2\\
			= & - t \beta_t (\alpha-1-\lambda)  \bra  Q_t, \dot{x}_t \ket  + \beta_t t^2 (1-\eta)\eta \left(\dot{\beta}_t + \frac{(2-\alpha)}{t}\beta_t\right)\|Q_t\|^2 -  (\alpha-1-\lambda)t\|\dot{x}_t\|^2\\
			= & -\frac{t}{2}(\alpha-1-\lambda) \|\dot{x}_t + \beta_t Q_t\|^2-  \frac{t}{2}(\alpha-1-\lambda)\|\dot{x}_t\|^2\\
			&+ \beta_t t^2 \left( (1-\eta)\eta  \left(\dot{\beta}_t +\frac{\beta_t}{t}(2-\alpha) \right) + \frac{1}{2}(\alpha-1-\lambda)\frac{\beta_t}{t}\right) \|Q_t\|^2\,.
		\end{aligned}
	\end{equation*}
	Plugging this and the definition of $\xi_t$ and of $w_t$ into \eqref{eq:lyapunov_analysis_step_0}, we get the claim.
\end{proof}

\begin{proof}[Proof of Lemma \ref{lem:nonexp_derivative_of_Lyapunov_second_order_discrete}]
Let \rev{$k \geq 1$}.	Apply the discrete Leibniz rule (see Appendix~\ref{sec:discrete_calculus}) to get:
	\begin{equation}\label{eq:disclyapunov_decrease_1}
		\begin{aligned}
			\Delta E_k^\lambda & = \Delta \delta_k  \bra  Q^k, z^k - z^* \ket  + \delta_{k+1} \bra  \Delta Q_k, z^k - z^* \ket + \delta_{k+1} \bra  Q^{k+1}, \Delta z_k \ket +\Delta \left(\frac{\delta}{2}\|Q\|^2\right)_{{k}}\\
			& \quad +\frac{1}{2}\Delta \xi_k \| Q^k\|^2 + \xi_{k+1}  \bra  \Delta Q_k, Q^{k+1} \ket - \frac{1}{2}\xi_{k+1}\|\Delta Q_{k}\|^2\\
			& \quad + c \bra  \Delta z_{k-1}, z^k - z^* \ket - \frac{c}{2}\|\Delta z_{k-1}\|^2\\
			& \quad + \bra  \Delta v_k, v^{k+1} \ket  - \frac{1}{2}\|\Delta v_k\|^2\,.
		\end{aligned}
	\end{equation}
	Following the continuous-time approach, we first write $\Delta v_k$ only in terms of $z^{k-1}-z^*$, $Q^k$ and $\Delta Q_k$, then compute $ \bra \Delta v_k, v^{k+1} \ket $ and finally plug everything into $\Delta E_k^\lambda$. This will give the necessary cancellations. Using that $\Delta t_k = 1$ and the update of Algorithm~\ref{alg:fast_km} (see \eqref{eq:discrete_second_order}), we get
	\begin{equation}
		\begin{aligned}
			\Delta v_k &= \lambda \Delta z_{k-1} + \Delta t_{k-1}\left( \Delta z_{k-1}  + \rev{(1-\eta)}  Q^k\right) + t_k\left( \Delta^2 z_{k-1}  +  \rev{(1-\eta)} \Delta Q_k\right)\\
			& = \lambda \Delta z_{k-1} +  \left( \Delta z_{k-1}  + \rev{(1-\eta)}  Q^k\right) + t_k\left(-\frac{\alpha}{t_k} \Delta z_{k-1} -  \Delta Q_k - \frac{\rev{\theta} }{t_k}Q^k +  \rev{(1-\eta)} \Delta Q_k\right)\\
			& = -(\alpha - 1-\lambda) \Delta z_{k-1} - (\rev{\theta}-\rev{(1-\eta)})Q^k - t_k  \rev{\eta}\Delta Q_k\\
			&= -(\alpha - 1-\lambda) \Delta z_{k-1} - \rev{\eta}(\alpha-1)Q^k - t_k  \rev{\eta} \Delta Q_k\,.
		\end{aligned} 
	\end{equation} 
	Let us now evaluate $ \bra  \Delta v_k, v^{k+1} \ket $, we get:
	\begin{equation}\label{eq:discv_times_vdot}
		\begin{aligned}
			& \bra \Delta v_k, v^{k+1} \ket \\
			 = & -\lambda (\alpha - 1- \lambda) \bra  \Delta z_{k-1}, z^{k} - z^* \ket -  t_k  (\alpha - 1- \lambda) \bra  \Delta z_{k-1}, \Delta z_{k} \ket  - (\alpha - 1- \lambda)t_k  \rev{(1-\eta)}  \bra  \Delta z_{k-1}, Q^{k+1} \ket \\
			& - \lambda \rev{\eta} (\alpha - 1) \bra  Q^k, z^k - z^* \ket - t_k \rev{\eta}(\alpha - 1) \bra  Q^k, \Delta z_k \ket    -  t_k \eta (1-\eta)(\alpha - 1) \bra  Q^k, Q^{k+1} \ket \\
			&  - \lambda t_k \rev{\eta} \bra  \Delta Q_k, z^k - z^* \ket -  t_k^2 \rev{\eta} \bra  \Delta Q_k, \Delta z_k \ket   - \eta  (1-\eta) t_k^2  \bra  \Delta Q_k, Q^{k+1} \ket \,.
		\end{aligned}
	\end{equation}
Let us now inspect $\|\Delta v_k\|^2$: 
\begin{equation}\label{eq:nonexp_term_in_delta_v}
    \begin{aligned}
        -\frac{1}{2}\|\Delta v_k\|^2 & = - \frac{1}{2}\rev{\eta}^2(\alpha-1)^2 \|Q^k\|^2 - \frac{t_k^2}{2}\rev{\eta}^2 \|\Delta Q_k\|^2 - \frac{1}{2}(\alpha-1-\lambda)^2 \|\Delta z_{k-1}\|^2\\
        & \quad  - (\alpha-1) \rev{\eta}^2 t_k  \bra  Q^k, \Delta Q_k \ket  - \rev{\eta} (\alpha-1-\lambda)(\alpha-1) \bra  Q^k, \Delta z_{k-1} \ket \\
        & \quad - \rev{\eta}t_k(\alpha-1-\lambda) \bra  \Delta z_{k-1}, \Delta Q_k \ket \,.
    \end{aligned}
\end{equation}
Plugging \eqref{eq:discv_times_vdot} and \eqref{eq:nonexp_term_in_delta_v} into \eqref{eq:disclyapunov_decrease_1} and rearranging suitably we get
\begin{align}
		\label{eq:nonexp_negative_terms}\Delta E_k^\lambda & = - \lambda \rev{\eta} \left(\alpha - 2\right)   \bra  Q^k, z^k - z^* \ket  -  t_k^2 \rev{\eta} \bra  \Delta Q_k, \Delta z_k \ket  \\
        & \nonumber \quad -\rev{\eta} t_k(\alpha-1-\lambda) \bra  \Delta z_{k-1}, \Delta Q_k \ket   \\
		& \label{eq:nonexp_term_in_Q} \quad +\frac{1}{2}\eta (1-\eta)  \Delta t_{k-1}^2 \| Q^k\|^2  -  t_k \eta (1-\eta)(\alpha - 1) \bra  Q^k, Q^{k+1} \ket \\
            & \nonumber \quad - \frac{\rev{\eta}^2}{2}(\alpha-1)^2 \|Q^k\|^2 - (\alpha-1)\rev{\eta}^2 t_k  \bra  Q^k, \Delta Q_k \ket  \\
		& \label{eq:nonexp_term_in_Qx} \quad  + \lambda  \rev{\eta} t_k \bra  Q^{k+1}, \Delta z_k \ket  - (\alpha - 1- \lambda)t_k  \rev{(1-\eta)} \bra  \Delta z_{k-1}, Q^{k+1} \ket  - t_k  \rev{\eta}(\alpha - 1) \bra  Q^k, \Delta z_k \ket  \\
        & \nonumber \quad - \rev{\eta} (\alpha-1-\lambda)(\alpha-1) \bra  Q^k, \Delta z_{k-1}  \ket  \\
		& \label{eq:nonexp_term_in_x} \quad -  t_k  (\alpha - 1- \lambda) \bra  \Delta z_{k-1}, \Delta z_{k} \ket \\
            & \nonumber \quad -\frac{1}{2}(\alpha-1-\lambda)^2 \|\Delta z_{k-1}\|^2 \\
		&  \label{eq:nonexp_aux_term} \quad - \frac{1}{2}\rev{\eta} t_{k}^2 \|\Delta Q_{k}\|^2 - \frac{\lambda}{2}(\alpha-1-\lambda)\|\Delta z_{k-1}\|^2 + \Delta \left(\frac{\delta}{2}\|Q\|^2\right)_{{k}} \,,
\end{align}
        where in the last line we have used $\eta(1-\eta) + \rev{\eta}^2 = \rev{\eta}$. To deal with the various nonnegative off-sets in the inner products \eqref{eq:nonexp_term_in_Q}, \eqref{eq:nonexp_term_in_Qx} and \eqref{eq:nonexp_term_in_x}, we make a discrete version of the continuous-time quadratic form appear and control the remainder with Young's inequality. This gives
	\begin{equation*}
		\begin{aligned}
			\eqref{eq:nonexp_term_in_Q}:=&\frac{1}{2}\eta (1-\eta) \Delta t_{k-1}^2  \| Q^k\|^2  -  t_k \eta (1-\eta)(\alpha - 1) \bra  Q^k, Q^{k+1} \ket \\
             &- \frac{\rev{\eta}^2}{2}(\alpha-1)^2 \|Q^k\|^2 - (\alpha-1)\rev{\eta}^2 t_k  \bra  Q^k, \Delta Q_k \ket \\
             =&  \rev{\eta} \left(\frac{\rev{(1-\eta)}}{2}\Delta t_{k-1}^2 - \rev{(1-\eta)} t_k (\alpha - 1)-\frac{\rev{\eta}}{2}(\alpha-1)^2 \right)\|Q^{k}\|^2  - t_k (\alpha-1)\rev{\eta}  \bra  Q^k, \Delta Q_k \ket    \\
               =& \underbrace{ -\left[\eta (1-\eta) \left((\alpha-2) t_k + \frac{1}{2}\right) + \frac{1}{2}\rev{\eta}^2 (\alpha - 1)^2 \right]\|Q^{k}\|^2}_{(\labterm{lb:lem1_Q1}{Q.1})}  - t_k (\alpha-1)\rev{\eta}  \bra  Q^k, \Delta Q_k \ket  \,,
		\end{aligned}
	\end{equation*}
        where in the coefficient that multiplies $\|Q^{k}\|^2$ we have used, by the definition of $t_k$, 
        \begin{equation}\label{eq:nonexp_Deltatk}
		\begin{aligned}
			\frac{1}{2}&\Delta t_{k-1}^2 =\frac{1}{2}\left((k - 1 +\sigma)^2-(k-2+\sigma)^2\right)=  t_k -\frac{1}{2}\,,
		\end{aligned} 
	\end{equation}
	hence $\frac{1}{2}\Delta t_{k-1}^2 - t_k (\alpha - 1) = -\frac{1}{2}- t_k (\alpha - 2)$. The first term of our quadratic form involving $\|Q^{k}\|^2$ appears. We now force $ \bra  Q^{k}, \Delta z_{k-1} \ket $ to appear. This comes from the fact that
	\begin{equation*}
		\begin{aligned}
			\eqref{eq:nonexp_term_in_Qx}&:= \lambda  \rev{\eta} t_k \bra  Q^{k+1}, \Delta z_k \ket  - (\alpha - 1- \lambda)t_k \rev{(1-\eta)}  \bra  \Delta z_{k-1}, Q^{k+1} \ket  - t_k \rev{\eta}(\alpha - 1) \bra  Q^k, \Delta z_k \ket  \\
        & \quad - \rev{\eta} (\alpha-1-\lambda)(\alpha-1) \bra  Q^k, \Delta z_{k-1}  \ket  \\
			& = \lambda \rev{\eta} t_k  \bra  Q^k, \Delta z_{k-1} \ket  - (\alpha-1-\lambda) t_k \rev{(1-\eta)} \bra  Q^k, \Delta z_{k-1} \ket  - t_k\rev{\eta}(\alpha-1) \bra  Q^k, \Delta z_{k-1} \ket \\
             & \quad + \lambda \rev{\eta} t_k  \bra  Q^{k+1}, \Delta^2 z_{k-1} \ket  - (\alpha-1-\lambda) t_k \rev{(1-\eta)} \bra  \Delta Q_{k}, \Delta z_{k-1} \ket  - t_k \rev{\eta} (\alpha-1)  \bra  Q^k, \Delta^2 z_{k-1} \ket   \\
             & \quad + \lambda \rev{\eta}t_k  \bra  \Delta Q_k, \Delta z_{k-1} \ket  - \rev{\eta}(\alpha-1-\lambda)(\alpha-1) \bra  Q^k, \Delta z_{k-1} \ket  \\
			& = \underbrace{-(\alpha - 1 -\lambda)t_k   \bra  Q^{k}, \Delta z_{k-1} \ket }_{(\labterm{lb:lem1_Q2}{Q.2})} - t_k \rev{\eta} (\alpha-1-\lambda)  \bra  Q^k, \Delta^2 z_{k-1} \ket\\
            & \quad + \lambda \rev{\eta} t_k  \bra  \Delta Q_k, \Delta^2 z_{k-1} \ket  - (\alpha-1-\lambda) t_k \rev{(1-\eta)} \bra  \Delta Q_{k}, \Delta z_{k-1} \ket   \\
             & \quad + \lambda \rev{\eta} t_k  \bra  \Delta Q_k, \Delta z_{k-1} \ket  - \rev{\eta} (\alpha-1-\lambda)(\alpha-1) \bra  Q^k, \Delta z_{k-1} \ket  \\
             & = \eqref{lb:lem1_Q2}\\
        & \quad \underbrace{- t_k \rev{\eta} (\alpha-1-\lambda)  \bra  Q^k, \Delta^2 z_{k-1} \ket }_{(\labterm{lb:lem1_II}{II})}+ \lambda \rev{\eta} t_k  \bra  \Delta Q_k, \Delta z_{k} \ket  - (\alpha-1-\lambda) t_k \rev{(1-\eta)}  \bra  \Delta Q_{k}, \Delta z_{k-1} \ket   \\
             & \quad  \underbrace{- \rev{\eta} (\alpha-1-\lambda)(\alpha-1) \bra  Q^k, \Delta z_{k-1} \ket }_{(\labterm{lb:lem1_N1}{N.1})}\,.
		\end{aligned}
	\end{equation*}
	We will handle \eqref{lb:lem1_II} separately later. To conclude, we make the term in $\|\Delta z_{k-1}\|^2$ appear. Starting with \eqref{eq:nonexp_term_in_x} we have:
	\begin{equation*}
		\begin{aligned}
			\eqref{eq:nonexp_term_in_x}&:=- t_k  (\alpha - 1- \lambda) \bra  \Delta z_{k-1}, \Delta z_{k} \ket  - \frac{1}{2}(\alpha-1-\lambda)^2 \|\Delta z_{k-1}\|^2\\
   & = \underbrace{\vphantom{\frac{1}{2}} -   t_k  (\alpha - 1- \lambda)\| \Delta z_{k-1}\|^2}_{(\labterm{lb:lem1_Q3}{Q.3})}\underbrace{\vphantom{\frac{1}{2}}- t_k  (\alpha - 1- \lambda) \bra  \Delta z_{k-1}, \Delta^2 z_{k-1} \ket }_{(\labterm{lb:lem1_III}{III})}\underbrace{- \frac{1}{2}(\alpha-1-\lambda)^2 \|\Delta z_{k-1}\|^2}_{(\labterm{lb:lem1_N2}{N.2})}\,. 
		\end{aligned}
	\end{equation*}
	Once again, we will deal with \eqref{lb:lem1_III} later. Gathering all the terms together, we get
		\begin{equation}\label{eq:DEk_2}
        \begin{aligned}
		 \Delta E_k^\lambda & = - \lambda \rev{\eta} \left(\alpha - 2\right)   \bra  Q^k, z^k - z^* \ket  -  t_k^2 \rev{\eta} \bra  \Delta Q_k, \Delta z_k \ket   \\
        &  \quad -\rev{\eta} t_k(\alpha-1-\lambda) \bra  \Delta z_{k-1}, \Delta Q_k \ket  \\
		&  \quad +\eqref{lb:lem1_Q1}  - t_k (\alpha-1)\rev{\eta} \bra  Q^k, \Delta Q_k \ket  \\
		&  \quad + \eqref{lb:lem1_Q2} + \eqref{lb:lem1_II} +  \lambda \rev{\eta} t_k  \bra  \Delta Q_k, \Delta z_k \ket  - (\alpha-1-\lambda) t_k \rev{(1-\eta)}  \bra  \Delta Q_{k}, \Delta z_{k-1} \ket   + \eqref{lb:lem1_N1}\\
		&   \quad + \eqref{lb:lem1_Q3} + \eqref{lb:lem1_III} + \eqref{lb:lem1_N2}\\
		&   \quad - \frac{1}{2} \rev{\eta} t_{k}^2 \|\Delta Q_{k}\|^2 - \frac{\lambda}{2}(\alpha-1-\lambda)\|\Delta z_{k-1}\|^2 + \Delta \left(\frac{\delta}{2}\|Q\|^2\right)_{{k}}\\
        & = - \lambda \rev{\eta} \left(\alpha - 2\right)   \bra  Q^k, z^k - z^* \ket  -  (t_k^2 - \lambda t_k) \rev{\eta} \bra  \Delta Q_k, \Delta z_k \ket   \\
        &  \quad - t_k(\alpha-1-\lambda) \bra  \Delta z_{k-1}, \Delta Q_k \ket  \\
		&  \quad + \eqref{lb:lem1_Q1}  - t_k (\alpha-1)\rev{\eta} \bra  Q^k, \Delta Q_k \ket  \\
		&  \quad + \eqref{lb:lem1_Q2} + \eqref{lb:lem1_II} + \eqref{lb:lem1_N1}\\
		&   \quad + \eqref{lb:lem1_Q3} + \eqref{lb:lem1_III} + \eqref{lb:lem1_N2} - \frac{1}{2}\rev{\eta} t_{k}^2 \|\Delta Q_{k}\|^2 - \frac{\lambda}{2}(\alpha-1-\lambda)\|\Delta z_{k-1}\|^2 + \Delta \left(\frac{\delta}{2}\|Q\|^2\right)_{{k}}\,,
        \end{aligned}
        \end{equation}   
         Let us now collect all the terms in \eqref{lb:lem1_Q2} and \eqref{lb:lem1_Q3} and argue analogously to the continuous-time setting to get
	\begin{equation}\label{eq:nonexp_quadratic_form_discrete}
		\begin{aligned}
			\eqref{lb:lem1_Q2} + \eqref{lb:lem1_Q3} & = -(\alpha - 1 -\lambda)t_k   \bra  Q^{k}, \Delta z_{k-1} \ket  -   t_k  (\alpha - 1- \lambda)\| \Delta z_{k-1}\|^2\\
			&=-(\alpha - 1- \lambda) t_k\left(  \bra   Q^{k},  \Delta z_{k-1}   \ket   +\left\| \Delta z_{k-1} \right\|^2 \right)\\
			&=-(\alpha - 1- \lambda){  t_k}\left(\frac{1}{2}\left\| Q^{k} + \Delta z_{k-1} \right\|^2+\frac{1}{2}\left\| \Delta z_{k-1} \right\|^2 -\frac{1}{2}\|Q^{k}\|^2 \right)\\
			& = (\alpha-1-\lambda)\frac{t_k}{2}\|Q^{k}\|^2 - (\alpha-1-\lambda)\frac{t_k}{2}\|\Delta z_{k-1}\|^2 -(\alpha - 1 - \lambda)\frac{t_k}{2}\left\| Q^{k} +  \Delta z_{k-1} \right\|^2\,.
		\end{aligned}
	\end{equation}
        Note that the positive term $(\alpha-1-\lambda)\frac{ t_k}{2}$ in the last display is not a concern when $\alpha > 2$ since we shall take $\lambda$ arbitrarily close to $\alpha - 1$ in our Lyapunov analysis in such a way that when combined with \eqref{lb:lem1_Q1}, the coefficient in front of $\|Q^k\|^2$ will be negative. The case $\alpha=2$ will be treated with $\lambda=\alpha-1$. We now deal with \eqref{lb:lem1_II} and \eqref{lb:lem1_III}, while the terms labeled with a capital ``$N$'' will be handled properly later. We have
        \begin{equation}
            \begin{aligned}
                \eqref{lb:lem1_II} & = -(\alpha-1-\lambda)t_k \rev{\eta} \bra  Q^{k}, \Delta^2 z_{k-1} \ket \\
                & = -(\alpha-1-\lambda)t_k \rev{\eta} \bra   Q^{k}, -\frac{\alpha}{t_k}\Delta z_{k-1} - \Delta Q_k - \frac{\rev{\theta}}{t_k} Q^k  \ket  \\
                & = (\alpha-1-\lambda)\rev{\eta} t_k  \bra  Q^k, \Delta Q_k  \ket  + \underbrace{(\alpha-1-\lambda)\left[\alpha\rev{\eta} \bra  Q^k, \Delta z_{k-1} \ket  + \rev{\eta}\rev{\theta} \|Q^k\|^2\right]}_{(\labterm{lb:lem1_N3}{N.3})}\,.
            \end{aligned}
        \end{equation}
        As far as \eqref{lb:lem1_III} is concerned, we have
        \begin{equation}
            \begin{aligned}
                \eqref{lb:lem1_III}&= - t_k  (\alpha - 1- \lambda) \bra  \Delta z_{k-1}, \Delta^2 z_{k-1} \ket \\
                & = - t_k  (\alpha - 1- \lambda)  \bra   \Delta z_{k-1}, -\frac{\alpha}{t_k}\Delta z_{k-1} - \Delta Q_k - \frac{\rev{\theta}}{t_k} Q^k   \ket  \\
                & =   t_k (\alpha - 1- \lambda)  \bra   \Delta z_{k-1},  \Delta Q_k   \ket    +  \underbrace{\alpha (\alpha - 1- \lambda)\|\Delta z_{k-1}\|^2  +  \rev{\theta}(\alpha - 1- \lambda)   \bra   \Delta z_{k-1}, Q^k   \ket  }_{(\labterm{lb:lem1_N4}{N.4})}\,.
            \end{aligned}
        \end{equation}
        Plugging now \eqref{lb:lem1_II} and \eqref{lb:lem1_III} into \eqref{eq:DEk_2} we get
        \begin{equation}\label{eq:delta_E_step_almost_final}
        \begin{aligned}
        \Delta E_k^\lambda & = - \lambda \rev{\eta} \left(\alpha - 2\right)   \bra  Q^k, z^k - z^* \ket  -  (t_k^2 - \lambda t_k)\rev{\eta} \bra  \Delta Q_k, \Delta z_k \ket \\
        &  \quad + \eqref{lb:lem1_Q1}  \underbrace{- t_k \lambda \rev{\eta} \bra  Q^k, \Delta Q_k \ket }_{(\labterm{lb:lem1_IV}{IV})}   \\
        &  \quad + \eqref{lb:lem1_Q2}  + \eqref{lb:lem1_N1}  + \eqref{lb:lem1_N3} \\
        &   \quad + \eqref{lb:lem1_Q3} + \eqref{lb:lem1_N2} + \eqref{lb:lem1_N4}   \\
        & \quad - \frac{1}{2}\rev{\eta} t_{k}^2 \|\Delta Q_{k}\|^2 \underbrace{- \frac{\lambda}{2}(\alpha-1-\lambda)\|\Delta z_{k-1}\|^2}_{(\labterm{lb:lem1_N5}{N.5})} + \Delta \left(\frac{\delta}{2}\|Q\|^2\right)_{{k}}  \,,
        \end{aligned}
        \end{equation}
        where the term $t_k (\alpha - 1- \lambda)  \bra   \Delta z_{k-1},  \Delta Q_k   \ket  $ canceled out with the corresponding one in \eqref{lb:lem1_III}, and we simplified the coefficient that multiplies $ \bra  Q^k, \Delta Q_k \ket $ using the contribution from \eqref{lb:lem1_II}. Let us also define for brevity: $(\labterm{lb:lem1_N}{N}):=\eqref{lb:lem1_N1} + \eqref{lb:lem1_N2} + \eqref{lb:lem1_N3} + \eqref{lb:lem1_N4} + \eqref{lb:lem1_N5}$ \rev{and inspect \eqref{lb:lem1_IV}: Using the polarization identity and recalling that  $\delta_k = t_{k-1}\lambda \rev{\eta} $, we get:
        \begin{equation}
        	\begin{aligned}
        		\eqref{lb:lem1_IV}&= - t_k \lambda \rev{\eta}  \bra  Q^k, \Delta Q_k \ket \\
        		& = - t_k \lambda \rev{\eta} \left(\frac{1}{2}\|Q^{k+1}\|^2 - \frac{1}{2}\|Q^k\|^2 - \frac{1}{2}\|\Delta Q_k\|^2\right)\\
        		& = - \frac{1}{2}t_k\lambda \rev{\eta}\|Q^{k+1}\|^2 + \frac{1}{2}t_{k-1}\lambda \rev{\eta}\|Q^{k}\|^2+ \frac{\lambda}{2}\rev{\eta}\|Q^k\|^2 + \frac{t_k}{2} \lambda\rev{\eta}\|\Delta Q_k\|^2\\
        		& = - \Delta \left(\frac{\delta}{2}\|Q\|^2\right)_{{k}} + \frac{\lambda}{2}\rev{\eta}\|Q^k\|^2 + \frac{t_k}{2} \lambda\rev{\eta}\|\Delta Q_k\|^2 \,.
        	\end{aligned}
        \end{equation}
    Inserting \eqref{lb:lem1_IV}, \rev{\eqref{lb:lem1_Q1}, \eqref{lb:lem1_Q2}, and \eqref{lb:lem1_Q3} into \eqref{eq:delta_E_step_almost_final}}, simplifying and grouping terms yields
    \begin{equation}\label{eq:nonexp_descent_E_4_bis}
    	\begin{aligned}
    		\Delta E_k^\lambda & = -(t_k^2 -\lambda t_k) \eta \left( \bra \Delta Q_k, \Delta z_k\ket + \frac{1}{2}\|\Delta Q_k\|^2 \right)  - \lambda \rev{\eta}\left(\alpha - 2\right) \left(\bra  Q^k, z^k - z^* \ket + \frac{1}{2} \| Q^k\|^2\right) \\				
    		&\quad +\left[\frac{\lambda}{2} \rev{\eta}\left(\alpha - 2\right) + \frac{\lambda}{2}\rev{\eta} \right]\| Q^k\|^2  - \left[\eta (1-\eta) \left((\alpha-2) t_k + \frac{1}{2}\right) + \frac{1}{2}\rev{\eta}^2 (\alpha - 1)^2 \right]\|Q^{k}\|^2 + \eqref{lb:lem1_N} \\
    		& \quad + (\alpha-1-\lambda)\frac{t_k}{2}\|Q^{k}\|^2 - (\alpha-1-\lambda)\frac{t_k}{2}\|\Delta z_{k-1}\|^2 -(\alpha - 1 - \lambda)\frac{t_k}{2}\left\| Q^{k} +  \Delta z_{k-1} \right\|^2\,.
    	\end{aligned}
    \end{equation}
    
    }

        We now make use of \eqref{eq:non_monotonicity_of_Q} to infer that
        \begin{equation}
        	 \bra  \Delta Q_k, \Delta z_k \ket \geq -\frac{1}{2}\|\Delta Q_k\|^2 \,.
        \end{equation}
        This allows us to get for $k\geq 1$ and such that $t_k \geq \rev{\lambda}$
        \begin{equation}\label{eq:nonexp_descent_E_4}
         \begin{aligned}
              \Delta E_k^\lambda & \leq - \frac{\rev{\eta}}{2}\Big(\rev{(1-\eta)} + \rev{\eta}(\alpha-1)^2 - \lambda(\alpha-1)\Big) \|Q^k\|^2 
            -\left(\eta(1-\eta)(\alpha-2) - \frac{\alpha-1-\lambda}{2}\right) t_k \|Q^k\|^2  \\
            & \quad  - (\alpha-1-\lambda)\frac{t_k}{2}\|\Delta z_{k-1}\|^2 -(\alpha - 1 - \lambda)\frac{t_k}{2}\left\| Q^{k} +  \Delta z_{k-1} \right\|^2 + \eqref{lb:lem1_N} \,.
         \end{aligned}
        \end{equation}
        Now, invoking the Young's inequality on all the inner products appearing in \eqref{lb:lem1_N}, we obtain
        \begin{equation}\label{eq:Nterm}
            \begin{aligned}
                \eqref{lb:lem1_N} &= - \rev{\eta} (\alpha-1-\lambda)(\alpha-1) \bra  Q^k, \Delta z_{k-1} \ket - \frac{1}{2}(\alpha-1-\lambda)^2 \|\Delta z_{k-1}\|^2\\
                & \quad +(\alpha-1-\lambda)\left[\alpha\rev{\eta} \bra  Q^k, \Delta z_{k-1} \ket  + \rev{\eta}\rev{\theta} \|Q^k\|^2\right]\\
                & \quad + \alpha (\alpha - 1- \lambda)\|\Delta z_{k-1}\|^2  +  \rev{\theta}(\alpha - 1- \lambda)   \bra   \Delta z_{k-1}, Q^k   \ket  - \frac{\lambda}{2}(\alpha-1-\lambda)\|\Delta z_{k-1}\|^2\\
                &= (\alpha-1-\lambda)\left(\rev{\eta}\rev{\theta} \|Q^k\|^2 + (\rev{\eta}+\rev{\theta})\bra  Q^k, \Delta z_{k-1} \ket + \frac{\alpha+1}{2}\|\Delta z_{k-1}\|^2\right)\\
                & \leq(\alpha - 1- \lambda)\left(\rev{\eta}\rev{\theta} + \frac{\rev{\eta}+\rev{\theta}}{2}\right) \| Q^k \|^2 + (\alpha - 1- \lambda)\frac{\alpha+2+\rev{\eta}(\alpha-1)}{2} \|\Delta z_{k-1}\|^2 \, .
            \end{aligned}
        \end{equation}
        Inserting \eqref{eq:Nterm} into \eqref{eq:nonexp_descent_E_4} and rearranging we get the claim.
\end{proof}

\printbibliography
% \bibliography{references}

\end{document}

%% file: definitions.tex
\usepackage{amsmath,amssymb,amsfonts,amsthm,amstext,verbatim}
\usepackage{enumitem,color,graphicx,stmaryrd,psfrag} 
\usepackage[utf8]{inputenc}
\usepackage{url}
\usepackage{mdwlist}   % lists and itemize 
\usepackage{mathtools} % special math symbols (eg: xarrows)
\usepackage{wasysym}   % some special caracters (hexagon)
\usepackage{amssymb}
\usepackage{amsmath}
\usepackage{amsthm}
\usepackage{graphicx}
\usepackage[usenames,dvipsnames]{xcolor}
\usepackage{tikz}
\usepackage{subcaption}
\usepackage{tikz}
\usepackage{float}
\usepackage{color}
\usepackage[ruled, vlined]{algorithm2e}
\usepackage{appendix}
\usepackage{array}
\usepackage{footnote}
\usepackage{booktabs}
\usepackage{hhline}
\makesavenoteenv{tabular}
\usepackage{scrextend}
\usepackage[colorinlistoftodos, textwidth=2.5cm, textsize=tiny, linecolor=red, backgroundcolor=red!25, bordercolor=red]{todonotes}

\definecolor{link}{rgb}{0.18,0.25,0.63}
\definecolor{mygrey}{rgb}{0.34,0.34,0.34}

\definecolor{myred}{rgb}{0.7,0.25,0.2}
\usepackage[top=3cm,bottom=3cm,left=2.8cm,right=2.8cm, marginparwidth=2.2cm]{geometry}

\definecolor{mygreen}{rgb}{0, 0.5, 0}

\newcommand{\cL}{\mathcal L}

\newcommand{\bA}{\boldsymbol A}
\newcommand{\bM}{\boldsymbol M}

\newcommand{\bC}{\boldsymbol C}

\newcommand{\bH}{\boldsymbol H}
\newcommand{\bx}{\boldsymbol x}
\newcommand{\bw}{\boldsymbol w}

\newtheorem{prop}{Proposition}[section]
\newtheorem*{prop*}{Proposition}
\newtheorem{rmk}[prop]{Remark}
\newtheorem*{rmk*}{Remark}
\newtheorem{thm}[prop]{Theorem}
\newtheorem*{thm*}{Theorem}

\newtheorem*{defn*}{Definition}
\newtheorem{lem}[prop]{Lemma}
\newtheorem*{lem*}{Lemma}

\newtheorem*{cor*}{Corollary}
\newtheorem{example}[prop]{Example}
\newtheorem*{example*}{Example}

\DeclareMathOperator{\zer}{zer}
\DeclareMathOperator{\Var}{Var}
\DeclareMathOperator{\Fix}{Fix}

\DeclareMathOperator{\Span}{span}
\DeclareMathOperator{\diag}{diag}
\DeclareMathOperator{\Image}{Im}
\DeclareMathOperator{\gra}{gra}

\DeclareMathOperator{\dist}{dist}
\DeclareMathOperator{\dom}{dom}
\DeclareMathOperator{\Div}{div}
\DeclareMathOperator{\prox}{prox}

\newcommand{\cG}{ \mathcal{G} }
\newcommand{\cQ}{ \mathcal{Q} }
\newcommand{\cO}{ \mathcal{O} }

\newcommand{\R}{\mathbb{R}}
\newcommand{\N}{\mathbb{N}}

\newcommand{\wt}[1]{ \widetilde{#1} }

\newcommand{\bra}{ \left \langle  }
\newcommand{\ket}{ \right \rangle  }

\makeatletter
\newcommand{\customlabel}[2]{%
	\protected@write \@auxout {}{\string \newlabel {#1}{{#2}{\thepage}{#2}{#1}{}} }%
	\hypertarget{#1}{#2}
}
\makeatother
\newcommand{\labterm}[2]{\customlabel{#1}{\text{$#2$}}}

\definecolor{verydarkgreen}{rgb}{0.0, 0.4, 0.0}
\newcommand{\rev}[1]{#1}